\documentclass[11pt,a4paper]{article}

\pdfoutput=1

\usepackage{amsfonts,amsmath,amsthm,amssymb,amstext}
\usepackage{geometry}                
\usepackage{fullpage}
\usepackage{graphicx}
\usepackage{subfigure}
\usepackage{latexsym}
\usepackage{wrapfig}
\usepackage{epstopdf}
\usepackage[pdftex,bookmarks,colorlinks,breaklinks]{hyperref}  % PDF hyperlinks, with coloured links
\usepackage{color}

\definecolor{dullmagenta}{rgb}{0.4,0,0.4}   % #660066
\definecolor{darkblue}{rgb}{0,0,0.4}
\hypersetup{linkcolor=red,citecolor=blue,filecolor=dullmagenta,urlcolor=darkblue} % coloured links

\newtheorem{thrm}{Theorem}[section]

\newtheorem{prpstn}[thrm]{Proposition}

\newtheorem{cnjctr}[thrm]{Conjecture}
\newenvironment{rmrk}%
  {\par\medbreak\refstepcounter{thrm}%
    \noindent\textbf{Remark~\thethrm. }}%
  {\par\medskip}

\newcommand{\vol}{\mathrm{vol}}

\title{Maximization of Laplace-Beltrami eigenvalues \\ on closed Riemannian surfaces }

\date{\today}

\author{Chiu-Yen Kao\thanks{Department of Mathematical Sciences, Claremont McKenna College, CA 91711 ({\tt Ckao@ClaremontMckenna.edu}). Chiu-Yen Kao is partially supported by NSF DMS-1318364.}, Rongjie Lai\thanks{Department of Mathematics, Rensselaer Polytechnic Institute, NY 12180 ({\tt lair@rpi.edu}). Rongjie Lai is partially supported by NSF DMS-1522645.}, and Braxton Osting\thanks{Corresponding author. Department of Mathematics, University of Utah, Salt Lake City, UT 84112 ({\tt osting@math.utah.edu}). Braxton Osting is partially supported by NSF DMS-1103959 and DMS-1461138.} 
}

\begin{document}
\maketitle
\begin{abstract} 
Let $(M,g)$ be a connected, closed, orientable Riemannian surface and denote by $\lambda_k(M,g)$ the $k$-th eigenvalue of the Laplace-Beltrami operator on $(M,g)$. In this paper, we consider  the mapping $(M, g)\mapsto \lambda_k(M,g)$. We propose a computational method for finding  the conformal spectrum $\Lambda^c_k(M,[g_0])$, which is defined by the eigenvalue optimization problem of maximizing $\lambda_k(M,g)$ for $k$ fixed as $g$ varies within a conformal class $[g_0]$ of fixed volume $\vol(M,g) = 1$. We also propose a computational method for the problem where $M$ is additionally allowed to vary over surfaces with fixed genus, $\gamma$. 
This is known as the topological spectrum for genus $\gamma$ and denoted  by $\Lambda^t_k(\gamma)$.  
Our computations support a conjecture of N. Nadirashvili (2002)
that $\Lambda^t_k(0) = 8 \pi k$, attained by a sequence of surfaces degenerating to a union of $k$ identical round spheres. Furthermore, based on our computations, we conjecture that 
$\Lambda^t_k(1) = \frac{8\pi^2}{\sqrt{3}} + 8\pi (k-1)$, attained by a sequence of surfaces degenerating into a union of an equilateral flat torus and $k-1$ identical round spheres. 
The values are compared to several surfaces where the Laplace-Beltrami eigenvalues are well-known, including spheres, flat tori, and embedded tori. In particular, we show that among flat tori of volume one, the  $k$-th Laplace-Beltrami eigenvalue has a local maximum with value  
$\lambda_k = 4\pi^2 \left\lceil \frac{k}{2} \right\rceil^2 \left( \left\lceil \frac{k}{2} \right\rceil^2  - \frac{1}{4}\right)^{-\frac{1}{2}}$. 
Several properties are also studied computationally, including uniqueness, symmetry, and eigenvalue multiplicity. \end{abstract}


\section{Introduction}
Let $(M,g)$ be a connected, closed, orientable Riemannian surface and $\Delta_{M,g} \colon C^\infty(M) \to C^\infty(M)$ denote the Laplace-Beltrami operator. 
The Laplace-Beltrami eigenproblem is to find eigenvalues $\lambda(M,g)$ and eigenfunctions, $\psi(x;M,g)$ for  $x\in M$, satisfying 
\begin{equation}
\label{eq:eig}
-\Delta_{M,g} \ \psi(x; M,g) = \lambda(M,g) \  \psi(x;M,g) \qquad x \in M. 
\end{equation}
Denote the spectrum of $-\Delta_{M,g}$  by
$\sigma (M,g) := \{ 0=\lambda_0(M,g) < \lambda_1(M,g) \leq \ldots \}$.
 For a general introduction to properties of $\Delta_{M,g}$ and $\sigma (M,g)$, we refer to \cite{Chavel1984,Berard1985}. 
Given a fixed manifold $M$, consider the mapping $g\mapsto \sigma(M,g) $. 
Let $G(M)$ denote the class of Riemannian metrics $g$ on $M$.
We recall that a metric $g$ is \emph{conformal} to $g_0$ if there exists a smooth  function $\omega\colon M \to \mathbb R_+$ such that $g = \omega g_0$. The \emph{conformal class}, $[g_0]$, consists of all metrics conformal to $g_0$. 
Following \cite{Colbois2003}, for $k$ fixed, we define the \emph{conformal $k$-th eigenvalue of $(M,[g_0])$} to be 
\begin{equation}
\label{eq:ConEig}
\Lambda^c_k(M,[g_0]) := \sup \{  \Lambda_k (M,g) \colon g \in [g_0] \}, 
\end{equation}
where $\Lambda_k (M,g) := \lambda_k(M,g) \cdot \vol(M,g)$.\footnote{Note that by the dilation property of eigenvalues, $\lambda_k(M,cg) = c^{-1} \lambda_k(M,g)$, this is equivalent to minimizing $\lambda_k(M,g_0)$ over $\{g\in [g_0]\colon \vol(M,g) = 1\}$.}
Let $\mathcal M(\gamma)$ denote the class of orientable, closed surfaces with genus $\gamma$ and consider the mapping $(M,g)\mapsto \sigma(M,g)$. 
For $k$ fixed, the $k$-th  \emph{topological eigenvalue for genus $\gamma$} is defined 
\begin{equation}
\label{eq:TopEig}
\Lambda^t_k(\gamma) := \sup \{  \Lambda_k (M,g) \colon M \in \mathcal M(\gamma), \ g \in G(M) \}. 
\end{equation}
The conformal and topological eigenvalues are finite; see \S\ref{sec:background}. We refer to the conformal eigenvalues and topological eigenvalues collectively as the conformal spectrum and topological spectrum, respectively. 

For some conformal classes,   the first few conformal eigenvalues are known explicitly.  However,  little is known about the larger conformal eigenvalues  of any conformal class, $(M,[g_0])$. The topological spectrum is only known for $\gamma = 0$ with $k=1,2$ and $\gamma=1$   with $k=1$ (a conjecture exists for $\gamma=2$, $k=1$). We discuss these results and provide some references in  \S\ref{sec:background}. 

In this work, we study the conformal and topological spectra computationally. To the best of our knowledge, this is the first computational study of the conformal and topological spectra. 
To achieve this goal, for constants $\omega_+ > \omega_- > 0$,  we define the admissible set,
$$
\mathcal A(M,g_0,\omega_-, \omega_+) := \{ \omega \in L^\infty(M)\colon  \omega_-\leq \omega \leq \omega_+ \ a.e. \}. 
$$
For a fixed Riemannian surface $(M,g_0)$ and a function $\omega \in \mathcal A(M,g_0,\omega_-, \omega_+) $, we consider the generalized eigenvalues, characterized by the Courant-Fischer formulation
\begin{equation}
\label{eq:GenEigs}
\lambda_{k-1}(M,g_0,\omega)= \min_{\begin{array}{c}
E_{k}\subset H^{1}(M)\\
\text{subspace of dim }k
\end{array}}\max_{\psi\in E_{k},\psi\neq0}\frac{\int_{M}| \nabla \ \psi|^{2} d\mu_{g_0} }{\int_{M}\psi^{2} \omega d\mu_{g_0}}, 
\end{equation}
where $E_{k}$ is in general a $k$-dimensional subspace of $H^{1}(M)$ and $d\mu_{g_0}$ is the measure induced by the metric $g_0$.  Note that for $\omega \in \mathcal C^\infty \cap \mathcal A(M,g_0,\omega_-, \omega_+) $, the identity $\Delta_{M,\omega g} = \frac{1}{\omega} \Delta_{M,g}$ implies that $\lambda_{k}(M,g_0,\omega) =  \lambda_k(M,\omega g_0)$. As above, we  define a volume-normalized quantity, $\Lambda_k(M,g,\omega) = \lambda_k(M,g,\omega) \cdot \int_M \omega d\mu_g$  
and consider the optimization problem,
\begin{equation}\label{eq:opt:Linfty} 
\Lambda^\star_k(M,g_0,\omega_-,\omega_+) = \sup \{ \Lambda_k (M,  g_0, \omega)\colon  \omega \in \mathcal A(M,g_0,\omega_-, \omega_+) \}.
\end{equation}  
\begin{prpstn} \label{prop:Existence} Fix $k\in \mathbb N$. 
Let $(M,g_0)$ be a smooth, closed Riemannian surface and $0< \omega_-< \omega_+$. Then there exists an $\omega^\star \in \mathcal A(M,g_0,\omega_-, \omega_+) $ which attains $\Lambda^\star_k(M,g_0,\omega_-,\omega_+)$, the supremum in \eqref{eq:opt:Linfty}.
Furthermore, for any $\epsilon >0$, there exist constants $\omega_+(\epsilon)$ and $\omega_-(\epsilon)$ satisfying  $\omega_+(\epsilon) > \omega_-(\epsilon) > 0$ such that
$$
\Lambda^c_k(M,[g_0])   - \epsilon \leq  \Lambda^\star_k\left(M,g_0,\omega_-(\epsilon),\omega_+(\epsilon)\right) \leq \Lambda^c_k(M,[g_0])   . 
$$
\end{prpstn}
Our proof of Proposition \ref{prop:Existence}, which we postpone to \S\ref{sec:ExistenceProof}, uses  the direct method in the calculus of variations. As discussed further in \S\ref{sec:rework}, similar results are given in \cite{Nadirashvili2010,Petrides2013,Kokarev2011} and considerably more regularity can be shown for a metric attaining the first conformal eigenvalue.  Our strategy is thus to approximate the solution to \eqref{eq:ConEig} by computing the solution to  \eqref{eq:opt:Linfty}  for a sequence of values $\omega_+$ and $\omega_-$ such that $\omega_+ \uparrow \infty$ and $\omega_- \downarrow 0$. The bound in Proposition \ref{prop:Existence}  justifies this strategy.
 Similarly, we approximate \eqref{eq:TopEig}, the topological spectrum for genus $\gamma$, by 
\begin{equation}\label{eq:opt:LinftyTop}
\sup \{ \Lambda_k (M, g_0,\omega)\colon M \in \mathcal M(\gamma), \ g_0 \in G(M), \text{ and }  \omega \in \mathcal A(M,g_0,\omega_-, \omega_+) \}.
\end{equation}

For a given closed Riemannian surface $(M,g_0)$ and constants $k\geq1$ and $\omega_+> \omega_->0$, we develop a computational method for seeking the conformal factor $\omega \in \mathcal A(M,g_0,\omega_-, \omega_+)$ which attains the supremum in \eqref{eq:opt:Linfty}. To achieve this aim, we evolve $\omega$ within $\mathcal A(M,g_0,\omega_-, \omega_+)$ to increase $\Lambda_k(M,g_0,\omega)$. If $\omega$ were assumed smooth, this would be equivalent to evolving a metric $g$  within its conformal class, $[g_0]$ to increase $\Lambda_k(M,g)$.  We also develop a computational method for approximating the topological spectrum  for genus $\gamma=0$ and $\gamma=1$ via \eqref{eq:opt:LinftyTop}. The method depends on an explicit parameterization of moduli space, and in principle could be extended to higher genus \cite{Imayoshi1992,Buser2010}. 

Our computations support a conjecture of N. Nadirashvili  \cite{Nadirashvili2002}
that $\Lambda^t_k(0) = 8 \pi k$, attained by a sequence of surfaces degenerating to a union of $k$ identical round spheres (see \S\ref{sec:genus0}). 
 That is, for dimension $n=2$, and a genus $\gamma=0$ surface, the inequality, 
 $\Lambda^t_k(0) \geq 8 \pi k$,
   of \cite[Corollary 1]{Colbois2003} is tight. 
Based on our computations, we  further conjecture that $\Lambda^t_k(1) = \frac{8\pi^2}{\sqrt{3}} + 8\pi (k-1)$, attained by a sequence of surfaces degenerating into a union of an equilateral flat torus and $k-1$ identical round spheres (see \S\ref{sec:TopSpecOne}).  This surface was also recently studied by Karpukhin \cite{karpukhin2013b}.  
As a comparison, we show that among flat tori, $\Lambda_k$  has a local maximum with value  
$\Lambda_k = 4\pi^2 \left\lceil \frac{k}{2} \right\rceil^2 \left( \left\lceil \frac{k}{2} \right\rceil^2  - \frac{1}{4}\right)^{-\frac{1}{2}}$. 
We conjecture that this is the global maximum among flat tori. 
A detailed study of the first non-trivial conformal eigenvalue of flat tori is also conducted in \S\ref{sec:FundConfEig}.

\paragraph{Outline.} 
In \S\ref{sec:background}, we provide some background material and review related work. This includes a discussion of properties of the Laplace-Beltrami  eigenproblem and its solution, 
a brief discussion of moduli spaces, 
variations of eigenvalues with respect to the conformal structure,  and 
the spectrum of the disconnected union of a surface and a sphere. 
We also provide a proof of  Proposition \ref{prop:Existence}.
In \S\ref{sec:SphereandTori}, we discuss the Laplace-Beltrami eigenproblem on a sphere and flat tori, which are central to later sections. 
In \S\ref{sec:compMeth}, we describe our computational methods. 
In \S\ref{sec:compRes}, we compute the conformal spectrum of several Riemannian surfaces and the topological spectrum for genus $\gamma=0$ and $\gamma=1$ surfaces.
We conclude in \S\ref{sec:Disc} with a discussion.

\section{Background and related work} \label{sec:background} 
Let $(M,g)$ be a connected, closed,  smooth Riemannian manifold of dimension $n \geq 2$.  
The first fundamental form on $M$ can be written (using Einstein notation) in local coordinates as $g = g_{ij} dx^i dx^j$, where $g_{ij} = g(\partial_{x^i},\partial_{x^j})$. 
 Let $d\mu_g$  denote the measure on $(M,g)$ induced by the Riemannian metric. 
Let $\langle \cdot, \cdot\rangle_g$ denote the $L^2$-inner product on $(M,g)$ and denote $\| f \|_g = \langle f, f\rangle_g^{\frac12}$. 
%Let $\langle \cdot, \cdot \rangle_g$ be the inner product on the tangent space, $TM$. 
In local coordinates the divergence and gradient are written 
$ (\nabla f)^i = \partial^i f = g^{ij} \partial_j f $ and 
$\displaystyle \text{div} X = \frac{1}{\sqrt{|g|}} \partial_i \sqrt{|g|} X^i $.
Here $g^{ij}$ is the inverse of the metric tensor $g = g_{ij}$ and $|\cdot|$ is the determinant. 
The Laplace-Beltrami operator, $\Delta_{M,g} \colon C^\infty(M) \to C^\infty(M)$ is written in local coordinates
\begin{equation}
\label{eq:LapBel}
\Delta_{M,g} f = \text{div}\nabla f =  \frac{1}{\sqrt{|g|}} \partial_i \sqrt{|g|} g^{ij} \partial_j f. 
\end{equation}
Denote the spectrum of $-\Delta_{M,g} $ by $\sigma (M,g)$. 
 For a general introduction to properties of $\Delta_{M,g}$ and $\sigma (M,g)$, we refer to \cite{Chavel1984,Berard1985}.

\paragraph{ Properties of $\Delta_{M,g}$ and $\sigma(M,g)$.} 
\begin{enumerate}
\item The eigenvalues $\lambda_{k}(M,g)$  are characterized by the Courant-Fischer formulation
\begin{equation}
\label{eq:CF}
\lambda_{k-1}(M,g)= \min_{\begin{array}{c}
E_{k}\subset H^{1}(M)\\
\text{subspace of dim }k
\end{array}}\max_{\psi\in E_{k},\psi\neq0}\frac{\int_{M}| \nabla \ \psi|^{2} d\mu_g }{\int_{M}\psi^{2} d\mu_g}, 
\end{equation}
where $E_{k}$ is in general a $k$-dimensional subspace of $H^{1}(M)$ and at the minimizer, $E_{k}=\text{span}(\{\psi_{j}(\cdot ;M,g)\}_{j=1}^{k})$.

\item For fixed $(M,g)$, $\lambda_k(M,g)\uparrow \infty$ as $k\uparrow \infty$ and each eigenspace is finite dimensional.  We have $\lambda_0=0$ and the corresponding eigenspace is one dimensional and spanned by the constant function.  Eigenspaces belonging to distinct eigenvalues are orthogonal in $L^2(M)$ and $L^2(M)$ is spanned by the eigenspaces. Every eigenfunction is $C^\infty$ on $M$. 

\item (dilation property) For $(M,g)$ fixed, the quantity $\lambda_k(M,g) \ \vol(M,g)^{\frac{2}{n}}$, where $n$ is the dimension, is invariant to dilations of the metric $g$. That is, for any $\alpha \in \mathbb R_+$, 
$$
\lambda_k(M,\alpha g) \ \vol(M,\alpha g)^{\frac{2}{n}} = 
\lambda_k(M, g) \ \vol(M, g)^{\frac{2}{n}}. 
$$
Since $ \vol(M,\alpha g) = \alpha^{\frac{n}{2}} \vol(M, g)$, this is equivalent to 
$\lambda_k(M,\alpha g) = \alpha^{-1} \lambda_k(M, g) $. For surfaces ($n=2$), $\Lambda_k(M,g) = \lambda_k(M,g)\ \vol(M,g)$ is invariant to dilations of the metric $g$. 

\item (Spectrum of disconnected manifolds) If $(M,g)$ is a disconnected manifold, $M = M_1 \cup M_2$, then $\sigma(M,g) = \sigma(M_1,g) \cup \sigma(M_2,g)$.

\item (Weyl's Law) Let $N(\lambda) := \# \{ \lambda_k(M,g)\colon \lambda_k(M,g) \leq \lambda\}$, counted with multiplicity. Then 
$$
N(\lambda) \sim \frac{ \omega_n  \vol(M,g)  }{(2\pi)^n} \lambda^{n/2} \qquad \text{as } \lambda \uparrow \infty,
$$
where  $\omega_n = \frac{ \pi^{\frac{n}{2}}}{ \Gamma(\frac{n}{2}+1)}$ 
is the volume of the unit ball in $\mathbb R^n$. In particular, 
$$
\lambda_k \sim  \frac{(2 \pi)^2  }{\omega_n^{\frac{2}{n}}  \vol(M,g)^{\frac{2}{n}}} \ k^{\frac{2}{n}} \qquad \text{as } k \uparrow \infty.
$$

\end{enumerate}

\subsection{Related work}  \label{sec:rework}
We briefly summarize some related work. A recent review was given by Penskoi \cite{penskoi2013c}. 

Although eigenvalue optimization problems were already proposed by Lord Rayleigh in the late 1870s \cite{Rayleigh1877} (see also the surveys \cite{Henrot:2006fk,AsBe2007}), eigenvalue optimization problems posed on more general surfaces  were not studied until the 1970s.  The first result in this direction is due to J. Hersch, who showed that 
$$
\Lambda^t_1(0) =  \Lambda_1(\mathbb S^2, g_0) = 8\pi \approx 25.13,
$$
attained only by the standard metric  (up to isometry) on $\mathbb S^2$  \cite{hersch1970} (see also \cite[p.94]{Chavel1984} or \cite[Chapter III]{Schoen:1994}).  
P.~C.~Yang and S.-T.~Yau generalized this result  in \cite{Yang1980}, proving 
$$
\Lambda^t_1(\gamma) \leq 8\pi (1+ \gamma). 
$$
In \cite{Korevaar1992}, N. Korevaar generalized this result to larger eigenvalues, showing there exists a constant $C$, such that 
$$
\Lambda^t_k(\gamma)  \leq C \  (1+\gamma)  \ k. 
$$
This result shows that the topological spectrum is finite and since $\Lambda^c_k(M,[g_0]) \leq \Lambda^t_k(\gamma)$ for any $M\in \mathcal M(\gamma)$ and $g_0 \in G(M)$, that  conformal eigenvalues are finite as well. In \cite{Nadirashvili1996}, N. Nadirashvili proved that  
$$
\Lambda^t_1(1) =  \Lambda_1(\mathbb T^2, g_0) = \frac{8\pi^2}{\sqrt{3}} \approx 45.58,  
$$
attained only by the flat metric on the equilateral torus (generated by $(1,0)$ and $(\frac{1}{2}, \frac{\sqrt{3}}{2}$), see \S\ref{sec:FlatTori}). Indeed, it was already known to M. Berger that the maximum of $\Lambda_1$ over all flat tori is attained only by the equilateral torus \cite{Berger1973}. For $k=2$, N. Nadirashvili  showed that 
$$
\Lambda^t_2(0) = 16\pi \approx 50.26,
$$
attained by a sequence of surfaces degenerating to a union of two identical round spheres \cite{Nadirashvili2002}.  Nadirashvili also conjectured that  $\Lambda^t_k(0) = 8 \pi k$, attained by a sequence of surfaces degenerating to a union of $k$ identical round spheres. In \cite{Jakobson2005}, the first eigenvalue of genus $\gamma=2$ surfaces are studied both analytically and computationally and it is conjectured that 
\begin{equation}
\label{eq:lam1genus2}
\Lambda^t_1(2) = 16 \pi \approx 50.26,
\end{equation} 
 attained by a Bolza surface,  a singular surface which is realized as a double branched covering of the sphere. %Some progress on this conjecture has been made recently by M. Karpukhin \cite{Karpukhin2013}. 

We next state several relevant results\footnote{We state the 2-dimensional results here for simplicity, but several of these results are proven for general dimension.} of B. Colbois and A. El Soufi \cite{Colbois2003}, from whom we have also adopted notation for the present work. It is shown that for any Riemannian surface $(M,g)$ and any integer $k\geq 0$, 
$\Lambda^c_k(M,[g]) \geq \Lambda^t_k(0)$. 
Furthermore, for all $k$, 
\begin{equation} \label{eq:SpectralGap}
\Lambda^c_{k+1}(M,[g]) - \Lambda^c_k(M,[g]) \geq \Lambda^t_1(0) = 8\pi 
\end{equation}
which implies that
$\Lambda^c_k(M,[g]) \geq 8 \pi k$. 
This implies that 
\begin{equation}
\label{eq:topGap}
\Lambda^t_k(\gamma) \geq \Lambda^t_\ell(\gamma) + 8 \pi (k-\ell),  \qquad  \text{for } k \geq \ell \geq 0.
\end{equation}
Intuitively, \eqref{eq:topGap} states that the $k$-th topological eigenvalue must be at least as large as the eigenvalue associated with the surface constructed by gluing $k-\ell$ balls of the appropriate volume to the  surface which maximizes the $\ell$-th eigenvalue; see \S\ref{sec:disconnectedUnion}. Taking $\ell=0$, \eqref{eq:topGap} gives  
$$
\Lambda^t_k(\gamma)  \geq 8 \pi k.
$$ 
Finally, for any fixed integer $k\geq0$, the function $\gamma \mapsto \Lambda^t_k (\gamma)$ is non-decreasing. 

Recently it has been  shown (independently by several authors) that the supremum in \eqref{eq:ConEig} for the first conformal eigenvalue, $\Lambda^c_1(M,[g_0])$, is attained by an extremal metric, $g^\star \in [g_0]$, and several results on the regularity of $g^\star$ have been proven \cite{Nadirashvili2010,Petrides2013,Kokarev2011}. In particular, $g^\star$ is smooth and positive, up to a finite  set of some conical singularities on $M$. G. Kokarev also studies the existence and regularity of higher conformal eigenvalues $\Lambda^c_k(M,[g_0])$  \cite{Kokarev2011}.

\bigskip

Closely related to conformal and topological spectra  is the study of extremal metrics on closed surfaces, on which there has recently been significant development 
\cite{Jakobson2003,Lapointe2008,Penskoi2012,penskoi2013b,penskoi2013,karpukhin2013b,karpukhin2014}. A Riemannian metric $g$ on a closed surface $M$ is said to be an \emph{extremal metric} for $\Lambda_k(M, g)$ if for any analytic deformation $g_t$ such that $g_0 = g$ the following inequality holds:
$$
\frac{d}{dt} \Lambda_k(M,g_t)\Big|_{t\downarrow0} \leq 0 \leq \frac{d}{dt} \Lambda_k(M,g_t)\Big|_{t\uparrow0}.
$$
Recently, M. Karpukhin \cite{karpukhin2013b} investigated a number of extremal metrics studied in \cite{Penskoi2012,penskoi2013,karpukhin2014,Lapointe2008} and showed, by direct comparison with the equilateral torus glued to kissing spheres,  that none are maximal. This is precisely the configuration which, based on numerical evidence, is conjectured to be maximal in the present paper.

\bigskip

For dimension $n\geq 3$, the topological spectrum does not  exist. Indeed, H.~Urakawa \cite{Urakawa1979} found a sequence of Riemannian metrics, $\{g_n\}_n$, of volume
one on the sphere $\mathbb S^3$ such that $\lambda_1(\mathbb S^3, g_n)  \rightarrow  \infty$. B.~Colbois and J. Dodziuk  showed that every compact manifold, $M$, with dimension $n\geq 3$ admits a unit-volume metric $g$ with arbitrarily large first eigenvalue, $\lambda_1(M,g)$ \cite{Colbois1994}. 

In \cite{Friedland1979}, S. Friedland studies the problem of finding a metric with $L^\infty$ constraints within its conformal class to \emph{minimize} (increasing) functions of the Laplace-Beltrami eigenvalues. For the sphere, $\mathbb S^2$, he shows that the infimum is attained at a metric which is bang-bang, {\it i.e.}, activates the pointwise constraints almost everywhere. Note that these results do \emph{not} shed light on the   \emph{maximization problem}, \eqref{eq:opt:Linfty}; we do not expect  a conformal factor achieving the supremum in \eqref{eq:opt:Linfty} to be bang-bang.

There are also a number of other types of bounds for eigenvalues on Riemannian manifolds. In particular, there are a number of both upper and lower bounds for Laplace-Beltrami eigenvalues of manifolds with positive Ricci curvature (see, for example, \cite[Ch. III]{Chavel1984}, \cite{Kroger1992}, and \cite{Ling2010}).   \cite{Girouard2009,Petrides2012} give upper bounds on the second eigenvalue of $n$-dimensional spheres for conformally round metrics. \cite{pacard2009,Colbois2010,Colbois2010b} study isoperimetric problems for Laplace-Beltrami eigenvalues of compact submanifolds. 

\subsection{A  brief discussion of moduli spaces} \label{sec:moduliSp}

Given two oriented, 2-dimensional Riemannian manifolds, $(M_1,g_1)$ and $(M_2,g_2)$, a \emph{conformal mapping} is an orientation-preserving diffeomorphism $h\colon M_1 \to M_2$ such that $h^*(g_2) = \omega g_1$ where $\omega$ is a real-valued positive smooth function on $M_1$. We say that $(M_1,g_1)$ and $(M_2,g_2)$ are conformally equivalent (or have the same complex structure if one identifies the induced Riemann surface) if there exists a conformal mapping between them. The moduli space of genus $\gamma$, $\mathfrak{M}_{\gamma}$, is the set of all conformal equivalence classes of closed Riemannian surfaces of genus $\gamma$. Roughly speaking, the moduli space parameterizes the conformal classes of metrics for a given genus.

Here, we introduce some very basic results from moduli theory for genus $\gamma=0$ and $\gamma=1$ surfaces.  By the Uniformization Theorem, every closed Riemann surface of genus $\gamma=0$ is conformally equivalent to the Riemann sphere, so the moduli space consists of a single point \cite{Imayoshi1992}. 

Every genus $\gamma=1$ Riemann surface is conformally equivalent to a Riemann surface $\mathbb C / \Gamma_\tau$ where, for given $\tau \in H$, $\Gamma_\tau =  \{m+n \tau \colon  m,n \in \mathbb Z \}$ is a lattice group on $\mathbb C$.
Here $H = \{ \tau \in \mathbb C\colon  \Im \tau > 0\}$ denotes the upper half plane. 
\begin{thrm} \cite[Theorem 1.1]{Imayoshi1992}  For any two points $\tau$ and $\tau'$ in the upper half-plane, the two tori $\mathbb C / \Gamma_\tau$ and $\mathbb C/\Gamma_{\tau'}$ are conformally equivalent if and only if 
$$
\tau' \in PSL(2,\mathbb Z)\tau := \left\{\frac{a\tau + b}{ c \tau + d}\colon a,b,c,d\in \mathbb Z \text{ and } ad-bc =1 \right\} 
$$
where $PSL(2,\mathbb Z)$ denotes the projective special linear group of degree two over the ring of integers.
\end{thrm}
Thus, the moduli space for genus $\gamma=1$, can be represented as the quotient space $H /  PSL(2,\mathbb Z)$ and the fundamental domain is the green shaded area in Figure  \ref{fig:coordDefs}(right). 
The moduli spaces for surfaces with genus $\gamma\geq 2$ have been studied in great detail (see, for example, \cite{Imayoshi1992}). However, a computationally tractable parameterization for general $\mathfrak{M}_{\gamma}$ is non-trivial. 

To find the topological spectrum \eqref{eq:TopEig} in practice, we use the moduli space to parameterize the conformal classes of metrics $[g_0]$. In the following section we discuss how the conformal factor $\omega$ is varied within each conformal class.

\subsection{Variations of Laplace-Beltrami eigenvalues within the conformal class }
\label{sec:VarEigs}

In this section, we compute the variation of a simple Laplace-Beltrami eigenvalue within the conformal class. General variations of a Laplace-Beltrami eigenvalue with respect to the conformal factor are discussed in  \cite{Soufi2008}. In this work, we only require the variation of a simple eigenvalue. 

Let $(M,g)$ be a fixed Riemannian manifold and consider the conformal class, consisting of metrics $\omega g$, where $\omega$ is a smooth, positive-valued function on $M$. 
Using \eqref{eq:LapBel}, 
the Laplace-Beltrami operator on $(M,\omega g)$ is expressed as
\begin{equation}
\label{eq:LapBelomega}
\Delta_{M,\omega g} f = \frac{1}{\omega^{n/2} \sqrt{|g|}} \partial_i  \left( \omega^{\frac{n}{2}-1}\sqrt{|g|} g^{ij} \partial_j f \right). 
\end{equation}

\begin{prpstn} \label{prop:Hadamard}
Let $(\lambda, \psi)$ be a simple eigenpair of $-\Delta_{M,\omega g}$. The variation of $\lambda$ with respect to a perturbation of the conformal function $\omega$ is given by 
\begin{equation}
\label{eq:Hadamard}
\left\langle  \frac{\delta \lambda}{\delta \omega}, \ \delta \omega \right\rangle_{\omega g} = 
\frac{1}{\langle\psi, \psi\rangle_{\omega g}}
\left\langle
- \frac{n}{2} \lambda \omega^{-1} \psi^2 + \frac{n-2}{2} \omega^{-1} \| \nabla_{\omega g} \psi \|^2_{\omega g} 
\ , \delta \omega
\right\rangle_{\omega g}
\end{equation}
In particular, for $n=2$, 
\begin{equation}
\label{eq:Hadamard2}
\left\langle  \frac{\delta \lambda}{\delta \omega}, \ \delta \omega \right\rangle_{\omega g} = 
- \lambda \frac{  \left\langle \omega^{-1} \psi^2
\ , \delta \omega 
\right\rangle_{\omega g} } 
{\langle\psi, \psi\rangle_{\omega g}}
= - \lambda \frac{  \left\langle  \psi^2
\ , \delta \omega 
\right\rangle_{ g} } 
{\langle\omega \psi, \psi\rangle_{ g}}
\end{equation}
\end{prpstn}
\begin{proof} Taking variations with respect to $\omega$, taking the $(M,\omega g)$-inner product with $\psi$, and using the eigenvalue equation, $-\Delta_{M,\omega g} \ \psi = \lambda \ \psi$, yields
$$
\delta \lambda \ \langle\psi, \psi\rangle_{\omega g} = 
\left\langle \psi ,\ 
\frac{n}{2} \omega^{-1} \delta \omega (- \lambda  \psi) - \frac{n-2}{2} \mathrm{div} \left[  (\omega^{-1} \delta w) \nabla_{\omega g} \psi \right] 
\right\rangle_{\omega g}. 
$$
Applying Green's formula yields \eqref{eq:Hadamard}.
\end{proof}

\subsection{Spectrum of the disconnected union of a surface and a sphere} \label{sec:disconnectedUnion}
It is useful to consider the spectrum of a disconnected union of a surface
$(M,g)$ and the sphere $(\mathbb S^2,g_0)$, denoted $(M',g')$.
Generally, the spectrum of disconnected manifolds consists of a union of the spectra of the connected components.  
Here, we consider the case where the sphere is dilated such that the $k$-th eigenvalue of $(M,g)$ is equal to the first eigenvalue of $(\mathbb S^2,g_0)$.  Consider the dilation
$$ 
(\mathbb S^2, g_0) \mapsto (\mathbb S^2, \alpha g_0).
$$
We choose $\alpha$ such that $\lambda_1(\mathbb S^2,\alpha g_0) = \lambda_k(M, g)$ implying
$$
\alpha^{-1} \lambda_1(\mathbb S^2, g_0) = \lambda_k(M, g).
$$
Since  $(\mathbb S^2, \alpha g_0)$ contributes an extra zero eigenvalue,  the $(k+1)$-th eigenvalue of the disjoint union $(M',g')$ is then $\lambda_k(M, g)$. 
The $(k+1)$-th volume-normalized eigenvalue of $(M',g')$ is then 
\begin{align*}
\Lambda_{k+1}(M',g') &= \lambda_k(M, g) \cdot \left( \vol(\mathbb S^2, \alpha g_0) +  \vol(M,g) \right) \\
&=  \lambda_k(M, g) \cdot \alpha \vol(\mathbb S^2,g_0) + \lambda_k(M, g) \cdot \vol(M,g)  \\
& = \Lambda_1(\mathbb S^2,g_0) + \Lambda_k(M, g).
\end{align*}
We remark that $(M',g')$ can be viewed as the degenerate limit of a sequence of surfaces \cite{Colbois2003}.

\subsection{Proof of Proposition \ref{prop:Existence}} \label{sec:ExistenceProof}
Fix $k\geq 1$. 
Let $(M,g_0)$ be a smooth, closed Riemannian surface and $0< \omega_-< \omega_+$. 
Write $\mathcal A = \mathcal A(M,g_0,\omega_-, \omega_+) $. 
Our proof of existence employs the direct method in the calculus of variations and follows \cite{Cox-McLaughlin:1990,Henrot:2006fk}. We first show that the supremum of $\Lambda_k(M,g_0,\cdot)$ on $\mathcal A$, as defined in \eqref{eq:opt:Linfty}, is finite and $\Lambda_k^\star(M,g_0,\omega_-, \omega_+) \leq \lambda_k^c(M,[g_0])$. Let $\omega \in \mathcal A$ be arbitrary.  By assumption, $(M,g_0)$ is  compact, so $\mathcal A \subset L^2$. Thus, $C^\infty$ is dense in $\mathcal A$. Using the 
 weak* continuity of $\Lambda_k(M, g_0, \cdot)$, there exists an $\tilde \omega \in C^\infty \cap \mathcal A$ with 
 $$
 \Lambda_k(M,g_0,\omega) \leq \Lambda_k(M,g_0,\tilde\omega) + \epsilon.
 $$
 Taking $\epsilon \downarrow 0$ we obtain $\Lambda_k^\star(M,g_0,\omega_-, \omega_+) \leq \lambda_k^c(M,[g_0]) < \infty$.

 Let  $\{ \omega_\ell \}_{\ell=1}^\infty$ 
be a maximizing sequence, {\it i.e.}, $\lim_{\ell \uparrow \infty} \Lambda_k(M,g_0, \omega_\ell ) \to \Lambda^\star_k$. Since $\mathcal A$ is weak* sequentially compact, there exists a $\omega_\star \in \mathcal A$ and a weak* convergent sequence $\{ \omega_\ell \}_{\ell=1}^\infty$ such that $\omega_\ell \to \omega_\star$ \cite{Cox-McLaughlin:1990,Henrot:2006fk}. Since the mapping $\omega \to \Lambda_k(M,  g_0,\omega)$ is weak* continuous over $\mathcal A$, 
$\Lambda^\star_k = \lim_{\ell\uparrow \infty} \Lambda_k(M,g_0,\omega_\ell ) = \Lambda_k(M, g_0,\omega_\star)$ \cite{Cox-McLaughlin:1990,Henrot:2006fk}.

For any $\epsilon >0$, by the definition of  supremum in \eqref{eq:ConEig}, there exists an $\bar \omega \in C^\infty(M)$ such that 
$$
0 \leq \Lambda^c_k(M,[g_0]) - \Lambda_k(M, \bar \omega g_0) \leq \epsilon . 
$$
Since $M$ is a compact surface, there exists $\omega_+(\epsilon) > \omega_-(\epsilon) > 0$
such that  $\bar \omega \in \mathcal A(M,g_0, \omega_-(\epsilon), \omega_+(\epsilon) )$. Using the optimality of $\Lambda^\star_k$, we have that 
$$
\Lambda^c_k(M,[g_0]) - \epsilon \leq \Lambda_k(M,\bar \omega g_0) 
=  \Lambda_k(M, g_0, \bar \omega) 
\leq \Lambda^\star_k(M,g_0, \omega_-(\epsilon), \omega_+(\epsilon) ). 
$$
$\hfill \square$

\section{The Laplace-Beltrami spectrum for spheres and tori}\label{sec:SphereandTori}

\subsection{Spectrum of a sphere} \label{sec:Sphere}
Consider $\mathbb S^2 = \{ (x,y,z)\in \mathbb R^3\colon x^2+y^2+z^2=1\}$
and let
$\iota\colon \mathbb S^2\hookrightarrow \mathbb R^3$ be the inclusion. 
Let $g_0:= \iota^*(dx^2+dy^2+dz^2)$ be the Riemannian
metric on $\mathbb S^2$ induced from the Euclidean metric
$dx^2+dy^2+dz^2$ on $\mathbb R^3$.
Consider the parameterization 
$$
x = \cos\phi \sin \theta, \quad 
y = \sin\phi \sin \theta, \quad 
z = \cos\theta, 
$$
where $\theta \in [0,\pi]$ is the colatitude and $\phi \in [0,2\pi]$ is the azimuthal angle. 
We compute $\vol(\mathbb S^2, g_0) =  4 \pi$. %\int_0^{2 \pi} \int_0^\pi \sin(\theta) d\theta d \phi 
In these coordinates, the Laplace-Beltrami operator  is given  by 
$$
\Delta f = \frac{1}{\sin \theta} \partial_\theta \left( \sin \theta \  \partial_\theta f \right) + \sin^{-2}\theta \ \partial^2_\phi f. 
$$
The eigenvalues of the Laplacian on $(\mathbb S^2, g_0)$ are of the form $\ell (\ell+1)$, $\ell = 0,1,\ldots$, each with multiplicity $2\ell + 1$. It follows by scaling that the eigenvalues of a sphere of area 1 are $\Lambda(\mathbb S^2, g_0) = 4\pi\ell(\ell + 1)$. 
Typically, the spherical harmonic functions\footnote{See \url{http://dlmf.nist.gov/14.30}. }, denoted $Y_{\ell,m}(\theta, \phi)$, are chosen as a basis for each eigenspace. Numerical values of the volume-normalized eigenvalues, $\Lambda_k(\mathbb S^2,g_0) $, are listed in Table \ref{tab:ConfofmalSpectrum} for comparison. 

\begin{rmrk} \label{rem:isom} We remark that there are other (spatially dependent) metrics on the sphere isometric to $g_0$ and hence have the same Laplace-Beltrami spectrum. This impacts the uniqueness of optimization results presented later.  An example of such a metric is constructed as follows.

Let $N= (0,0,1)$ and $S=(0,0,-1)$ be the north pole and south pole of $\mathbb S^2$. There is a $C^\infty$ diffeomorphism 
(stereographic projection) 
$\pi\colon \mathbb S^2 -  \{ N\} \longrightarrow \mathbb R^2$, $\pi(x,y,z)= \left(\frac{x}{1-z}, \frac{y}{1-z} \right)$. 
The inverse map is given by $
\pi^{-1}\colon \mathbb R^2\longrightarrow \mathbb S^2 - \{ N\}$,
$$
\quad \pi^{-1}(u,v)=\left(\frac{2u}{1+u^2+v^2}, \frac{2v}{1+u^2+v^2}
,\frac{-1+u^2+v^2}{1+u^2+v^2} \right).
$$
Let $h:= (\pi^{-1})^* g_0$ be the pullback Riemannian 
metric on $\mathbb R^2$. Then 
$$
h= \frac{4(du^2+dv^2)}{(1+u^2+v^2)^2}.
$$
For any $\alpha\in \mathbb R$, define the dilation 
$T_\alpha: \mathbb R^2\longrightarrow \mathbb R^2$ by $T_\alpha(u,v)= (e^\alpha u, e^\alpha v)$. 
In particular, $T_0$ is the identity map.
For each $\alpha\in \mathbb R$, we define the following
Riemannian metric on $\mathbb S^2-\{N\}$,
$$
g_\alpha:= (\pi^{-1}\circ T_\alpha\circ \pi)^* g_0 = \pi^* T_\alpha^*h
=\frac{1}{(\cosh(\alpha) + \sinh(\alpha)\cdot z)^2} g_0.
$$
Then $g_\alpha$ extends to a $C^\infty$ Riemannian metric on $\mathbb S^2$ with constant sectional curvature $+1$. Note that when $\alpha=0$, the right hand side recovers $g_0$. 

The diffeomorphism $\pi^{-1}\circ T_\alpha\circ \pi\colon \mathbb S^2-\{N\}\longrightarrow \mathbb S^2-\{N\}$
extends to a diffeomorphism $\phi_\alpha\colon \mathbb S^2\longrightarrow \mathbb S^2$, and  $g_\alpha=\phi_\alpha^*g_0$.
So $\phi_\alpha\colon (\mathbb S^2, g_\alpha)\to (\mathbb S^2,g_0)$ is an isometry and
$\iota\circ\phi_\alpha\colon (\mathbb S^2,g_\alpha)\longrightarrow (\mathbb R^3,dx^2+dy^2+dz^2)$ is
an isometric embedding.  The isometric conformal factor for $\alpha=\frac 1 2$ is plotted in Figure \ref{fig:isometry}. 
To plot this conformal factor on the sphere in Figure  \ref{fig:isometry} (and again for Figures \ref{fig:optSpheres} and \ref{fig:sphereSequence}(left)), we have used the Hammer projection, 
\begin{align*}
x =  \frac{ 2\sqrt2 \cos \phi \sin \frac{\theta}{2}}{\sqrt{1 + \cos \phi \cos \frac{\theta}{2}}}, \quad 
y =  \frac{\sqrt 2 \sin \phi }{\sqrt{1 + \cos \phi \cos \frac{\theta}{2}}},
\end{align*}
where $\theta\in [0,2\pi]$ is the azimuthal angle (longitude) and $\phi \in \left[ -\frac{\pi}{2}, \frac{\pi}{2}\right]$ is the  altitudinal angle (latitude).

\end{rmrk}

\begin{figure}[t]
\begin{center}
 \includegraphics[width=8cm]{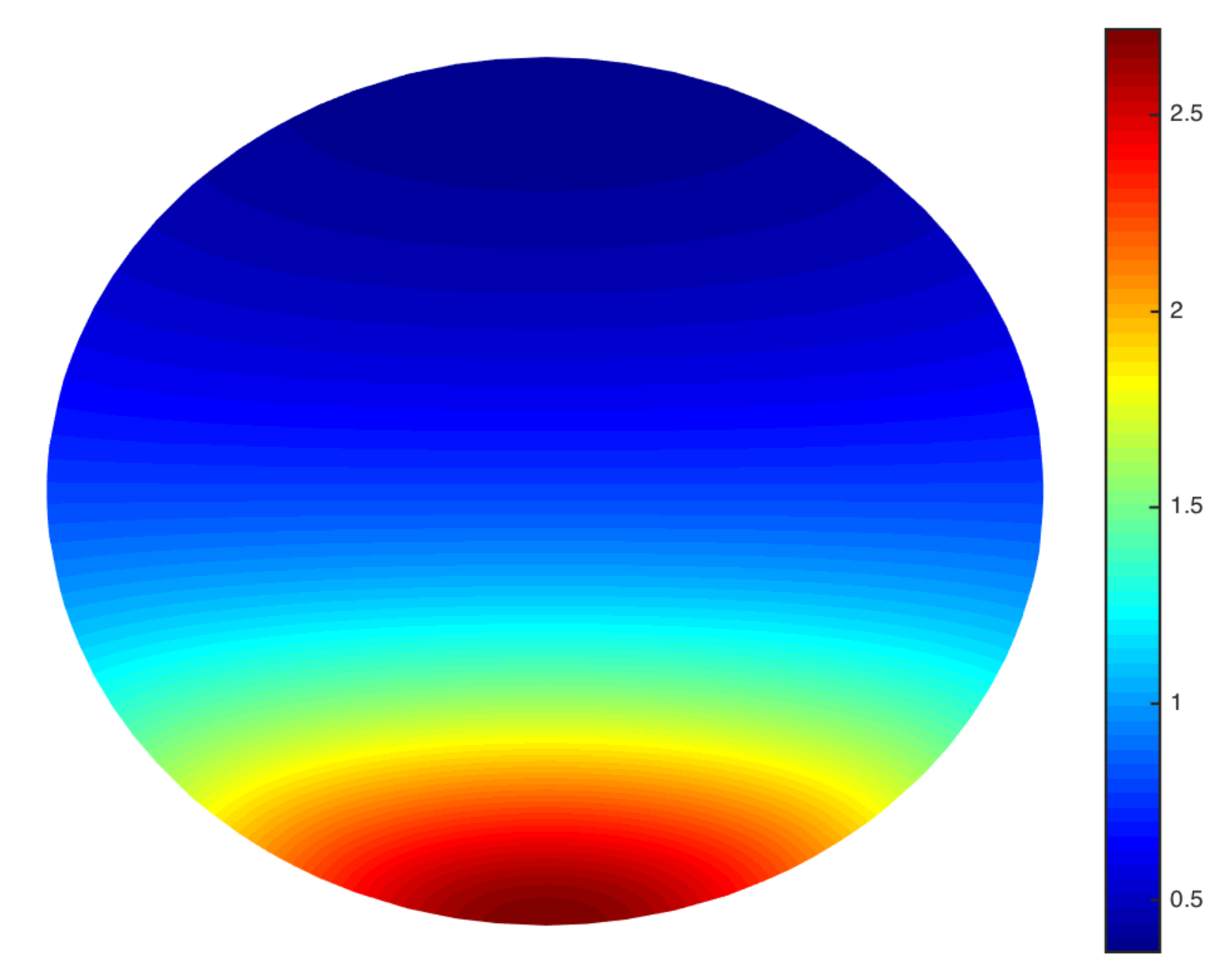} 
\caption{A Hammer projection of  a conformal factor on the sphere  that is isometric to the round sphere (and hence has the same spectrum). See \S\ref{sec:Sphere}.}
\label{fig:isometry}
\end{center}
\end{figure}

\subsection{Spectrum of $k$ ``kissing'' spheres}  \label{sec:KissingSpheres}
Let ($\mathbb S^2,g_0)$ be the sphere embedded in $\mathbb R^3$ with the canonical metric.  We consider $k$ copies of $(\mathbb S^2, g_0)$ and bring them together in $\mathbb R^3$, so that they are ``barely touching''. (This can be made precise by considering a sequence of surfaces  degenerating in this configuration  \cite{Colbois2003}.) 
We refer to this configuration as \emph{$k$~kissing spheres}.
It follows from \S\ref{sec:disconnectedUnion} that  $k$~kissing spheres will have $k$   zero eigenvalues ($0 = \lambda_0  = \ldots = \lambda_{k-1}$) with corresponding eigenfunctions localized and  constant on each sphere. The first nonzero volume-normalized  eigenvalue is 
\begin{equation}
\label{eq:kissingspheres}
\Lambda_k =  8 \pi k \qquad \text{($\lambda_k$ has multiplicity $3k$)}.
\end{equation} 
The corresponding eigenfunctions can be chosen to be  spherical harmonic functions supported on each single sphere.  Numerical values of the $k$-th eigenvalue of $k$ kissing spheres are listed in Table \ref{tab:ConfofmalSpectrum} for comparison.

\subsection{Spectrum of flat tori}\label{sec:FlatTori}
The flat torus is generated by identification of opposite sides of a parallelogram with the same orientation.   Consider the flat torus with corners $(0,0)^t$, $(1,0)^t$, $(a,b)^t$, and $(1+a,b)^t$.  
We  refer to this torus as the $(a,b)$-flat torus. 
This is isometric to the quotient of the Euclidean plane by the lattice $L$, $\mathbb R^2 / L$, where $L$ is the lattice generated by the two linearly independent vectors,  $b_1 = (1,0)^t$ and $b_2 = (a,b)^t$. 

The spectrum of the $(a,b)$-flat torus can be explicitly computed \cite{milnor1964,Giraud2011,laugesen2011}. Define
\begin{align*}
B=(b_{1},b_{2})=\begin{pmatrix}1 & a\\
0 & b
\end{pmatrix}. 
\end{align*}
The dual lattice $L^{*}$ is defined
$
L^{*}=\left\{ y\in \mathbb R^{2} \colon  x\cdot y  \in \mathbb Z, \ \forall x\in L \right\}
$
and has a basis given by the columns of  $D= (B^{t})^{-1} $. 
For the  $(a,b)$-flat torus, we compute 
\begin{align*}
D=(d_{1},d_{2})=(B^{t})^{-1}=\begin{pmatrix}1 & 0\\
-\frac{a}{b} & \frac{1}{b}
\end{pmatrix} & .
\end{align*}
Each $y\in L^{*}$ determines an eigenfunction $\psi(x)=e^{2\pi \imath x \cdot y  }$
with corresponding  eigenvalue $\lambda = 4\pi^{2} \| y \|^2 $. Since $y\in L^{*}  \implies - y\in L^{*}$, each nontrivial eigenvalue has even multiplicity.  
It follows that the eigenvalues of the  $(a,b)$-flat torus are of the form 
$$
\lambda(a,b)=4\pi^{2}\left[ c_{1}^{2}\left(1+a^2/b^2 \right)-2c_{1}c_{2}a/b^2+c_{2}^{2} / b^2  \right], \qquad (c_1, c_2) \in \mathbb Z^2. 
$$
More precisely, we can write a Courant-Fischer type expression for the $k$-th eigenvalue, 
\begin{equation} \label{eq:EigsFlatTori}
\lambda_{k}(a,b)= \min_{\substack{E \subset \mathbb Z^2\\ |E|=k+1}} \ 
\max_{(c_1,c_2)\in E} \ 
4\pi^{2}\left[ c_{1}^{2}\left(1+a^2/b^2 \right)-2c_{1}c_{2}a/b^2+c_{2}^{2} / b^2  \right]. 
\end{equation}
For example, the first eigenvalue of the $(\frac{1}{2},\frac{\sqrt{3}}{2})$-torus,
$\lambda_{1}=\frac{16\pi^{2}}{3}$ (multiplicity 6), is obtained when $(c_1,c_2)=(\pm1,0), \ \pm(1,1)$, or $(0,\pm1)$ 
implying 
$\Lambda_1 = \lambda_{1}b=\frac{8\pi^{2}}{\sqrt{3}} \approx 45.58$.
Numerical values of volume-normalized  Laplace-Beltrami  eigenvalues, $\Lambda_k(a,b) := \lambda_k(a,b) \cdot b$  for the square flat torus, $(a,b)=(0,1)$, and equilateral flat torus, $(a,b)= (\frac12,\frac{\sqrt{3}}{2})$  are listed in Table \ref{tab:ConfofmalSpectrum} for comparison.

\begin{figure}[t]
\begin{center}
 \includegraphics[width=9cm]{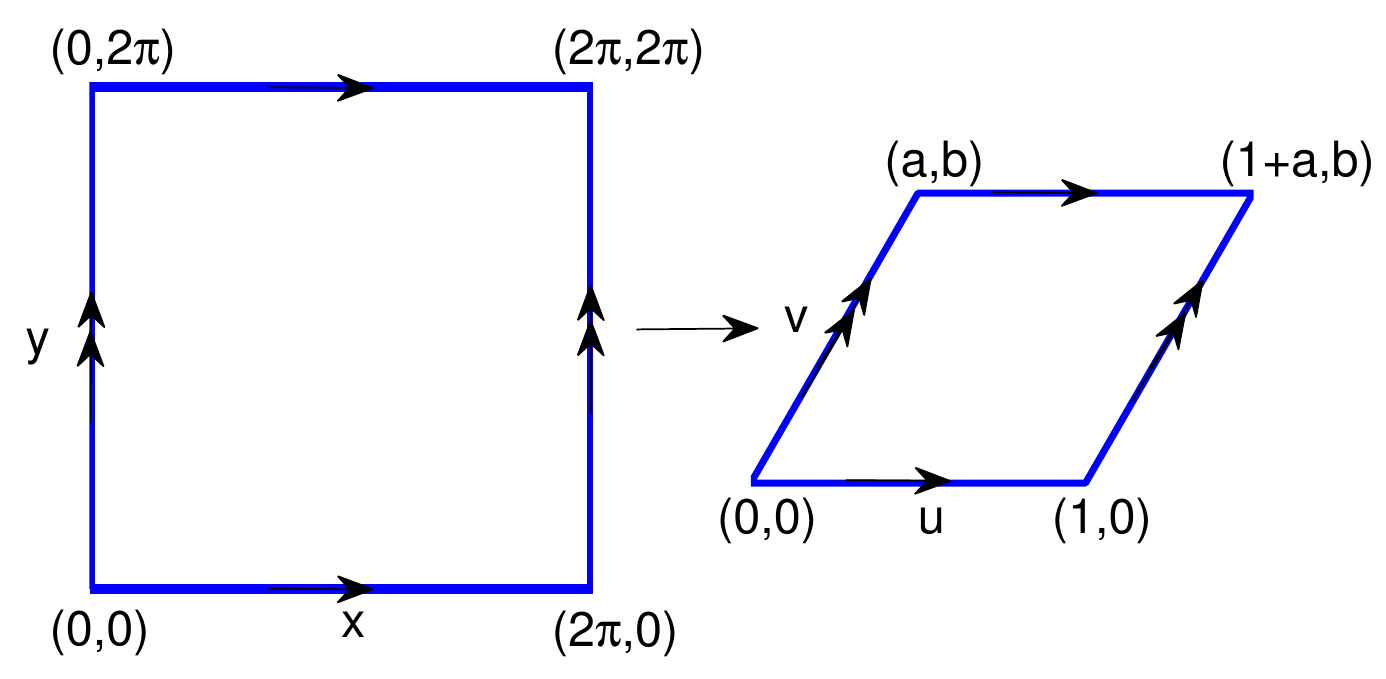}  \quad
 \includegraphics[width=7cm]{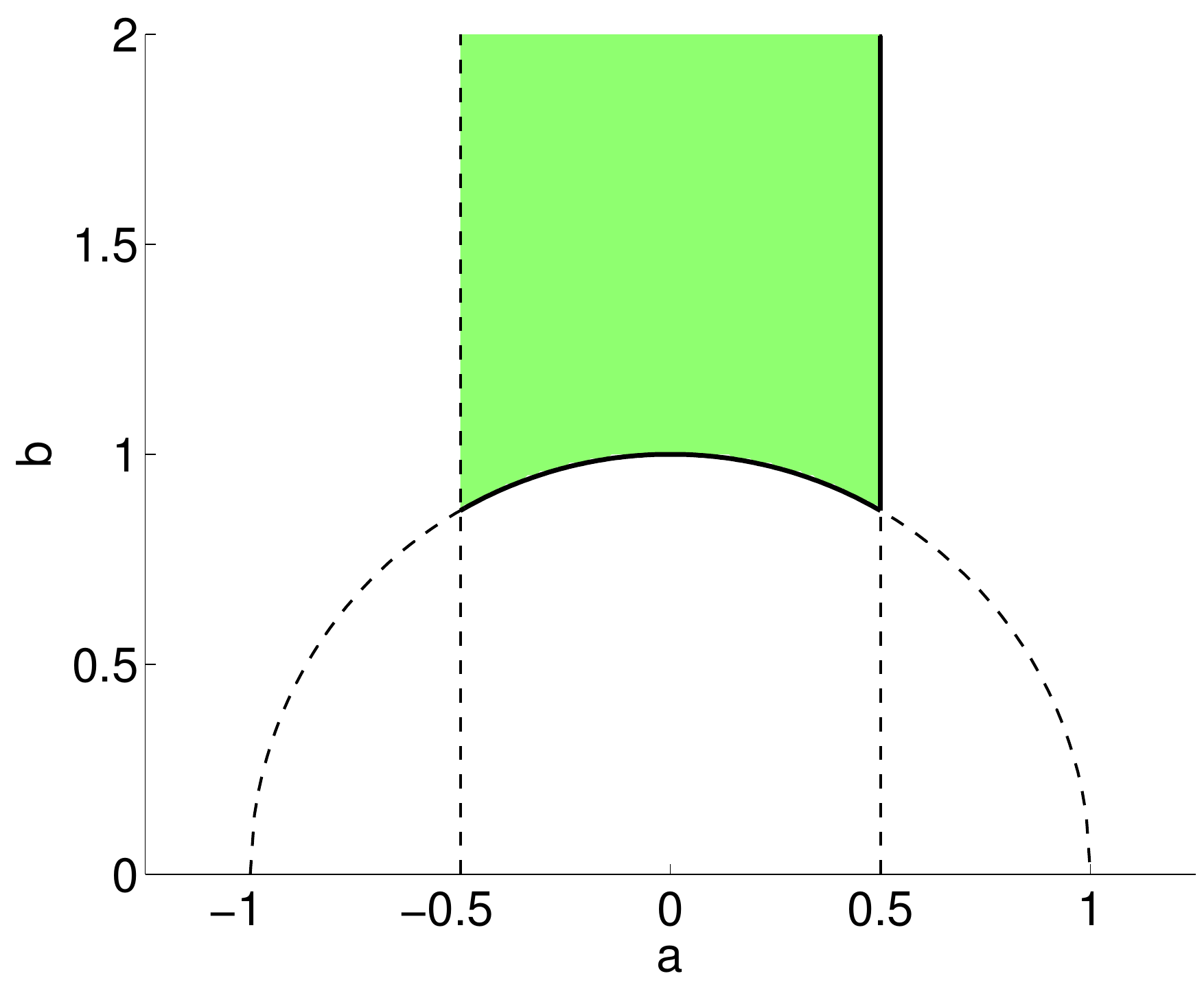}
\caption{{\bf (left)} Coordinates used in the construction of a flat torus. {\bf (right)} The fundamental domain for the moduli space of genus $\gamma=1$ Riemannian surfaces. See  \S\ref{sec:moduliSp} and \S\ref{sec:FlatTori}.}
\label{fig:coordDefs}
\end{center}
\end{figure}

It is useful to consider the linear transformation  from the $[0,2\pi]^2$ square to the $(a,b)$-flat torus,
\begin{align}
\label{eq:LinToriTransf}
\begin{pmatrix} u\\ v \end{pmatrix}
=  \frac{1}{2\pi} 
\begin{pmatrix}
1 & a \\
0 & b
\end{pmatrix}
\begin{pmatrix} x\\ y \end{pmatrix}
\quad \textrm{and} \quad 
\begin{pmatrix} x\\ y \end{pmatrix}
=   \frac{2 \pi}{b} 
\begin{pmatrix}
b & -a \\
0 & 1
\end{pmatrix}
\begin{pmatrix} u\\ v \end{pmatrix}. 
\end{align}
See Figure \ref{fig:coordDefs}. 
The pullback metric on the square is then given by 
$$
\frac{1}{4\pi^2} 
\begin{pmatrix}
1 & a \\
a & a^2+b^2
\end{pmatrix}. 
$$
Using  \eqref{eq:LapBelomega}, we obtain the Laplace-Beltrami operator on the square
\begin{equation}
\label{eq:Delab}
\Delta_{a,b} = \frac{4 \pi^2}{b^2} \left[ (a^2 + b^2) \partial_x^2 - 2a \partial_x \partial_y + \partial_y^2 \right]. 
\end{equation}
By construction, this mapping is an isometry and hence the eigenvalues of the flat Laplacian on the $(a,b)$-flat torus  are precisely the same as the eigenvalues of $\Delta_{a,b}$  on $[0,2\pi]^2$ (with periodic boundary conditions).  

The volume of the flat torus is simply $ b$. In this section, we consider the optimization problem 
\begin{equation}
\label{eq:maxEigsFlatTori}
\sup_{(a,b)\in \mathbb R^2} \ \Lambda_k (a,b), \quad \text{where} \quad \Lambda_k(a,b) := b\cdot  \lambda_k (a,b).
\end{equation}
%The conformal equivalence of two flat tori does \emph{not} imply the equivalence of the Laplace-Beltrami spectra (as the Laplace-Beltrami spectrum is generally not invariant under conformal mapping). 
Up to isometry and homothety (dilation), there is a one-to-one correspondence between the moduli space of flat tori and the fundamental region, 
 \begin{equation}
 \label{eq:Mad}
 F := \{ (a,b) \in \mathbb R^2 \colon \ a \in (-1/2,1/2]  \text{ and } a^2 + b^2 \geq 1 \},
\end{equation}
 as illustrated in  Figure  \ref{fig:coordDefs}(right). 
It follows that the admissible set in \eqref{eq:maxEigsFlatTori} can be reduced to $F$. To see this more explicitly, we prove in the following proposition that there exist three transformations of the parameters $(a,b)$ which preserve the value of $\Lambda_k(a,b)$. The first two are isometries and the third corresponds to a rotation and homothety. The last two are  due to the $SL(2,\mathbb Z)$ invariance of $\mathbb Z^2$ \cite{Imayoshi1992}.  Each transformation is  illustrated in Figure  \ref{fig:ToriTransform}. By composing these transformations, the fundamental domain can be restricted to $F$ and furthermore, on $F$, eigenvalues are symmetric with respect to the $b$-axis.

%It does not directly follow from moduli theory  for genus $\gamma=1$ Riemannian surfaces (see \S\ref{sec:moduliSp}) that the admissible set in \eqref{eq:maxEigsFlatTori} should be reduced to, say, the fundamental domain. However, the following proposition gives three transformations of the parameters $(a,b)$ which preserve the value of $\Lambda_k(a,b)$. The first is an (orientation reversing) isometry.  \com{The second was known to Conway and Sloan (see, for example,  \cite{}) and is not an isometry. } 
 %The third mapping corresponds to a rotation and dilation. A consequence of the proposition is that the admissible set in \eqref{eq:maxEigsFlatTori}  can be restricted to the fundamental domain, 
% \begin{equation}
 %\label{eq:Mad}
 %F := \{ (a,b) \in \mathbb R^2 \colon \ a \in (-1/2,1/2]  \text{ and } a^2 + b^2 \geq 1 \},
%\end{equation}
  
 \begin{figure}[t]
\begin{center}
\hspace{-1cm}
\includegraphics[height=4.4cm]{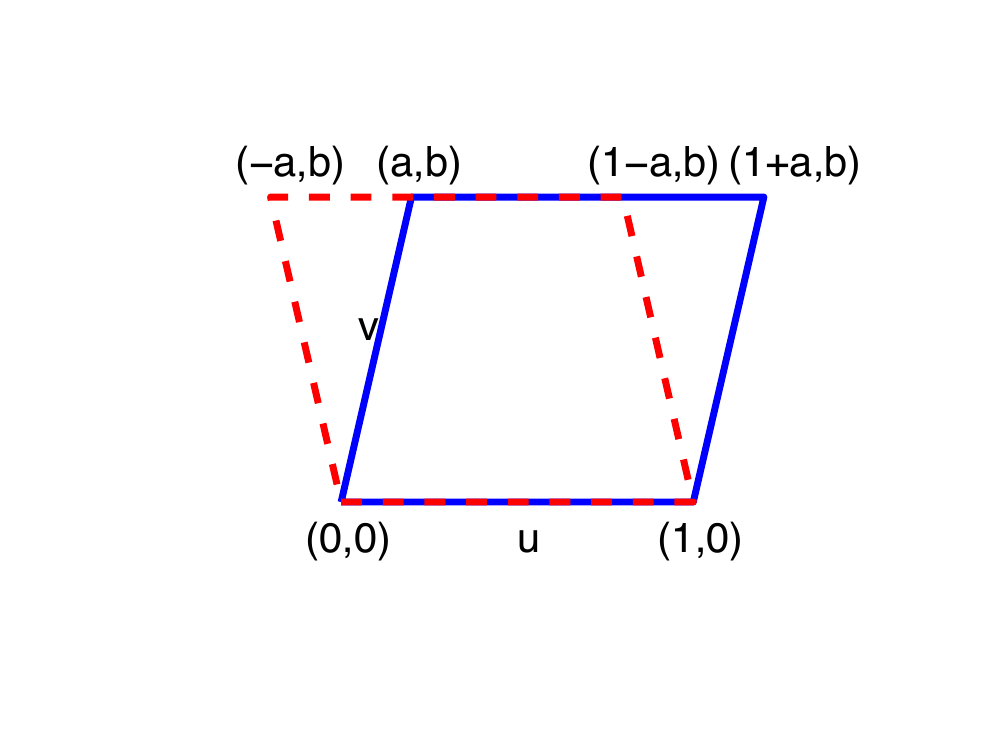}\hspace{-0.9cm}
\includegraphics[height=4.4cm]{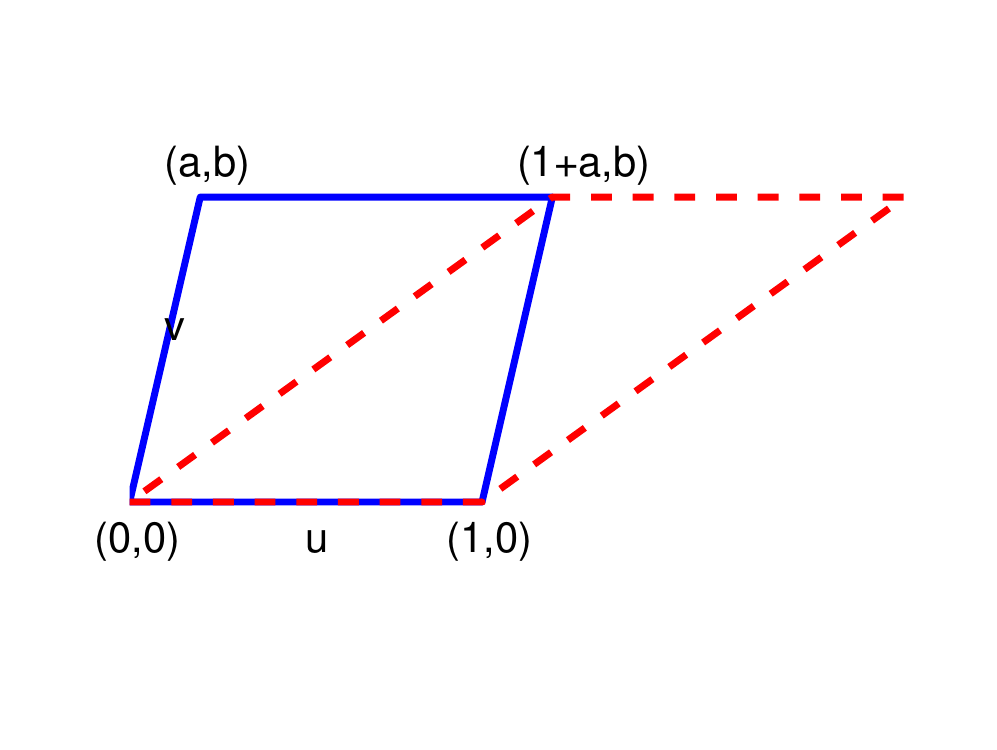}\hspace{-0.7cm}
\includegraphics[height= 3.7cm]{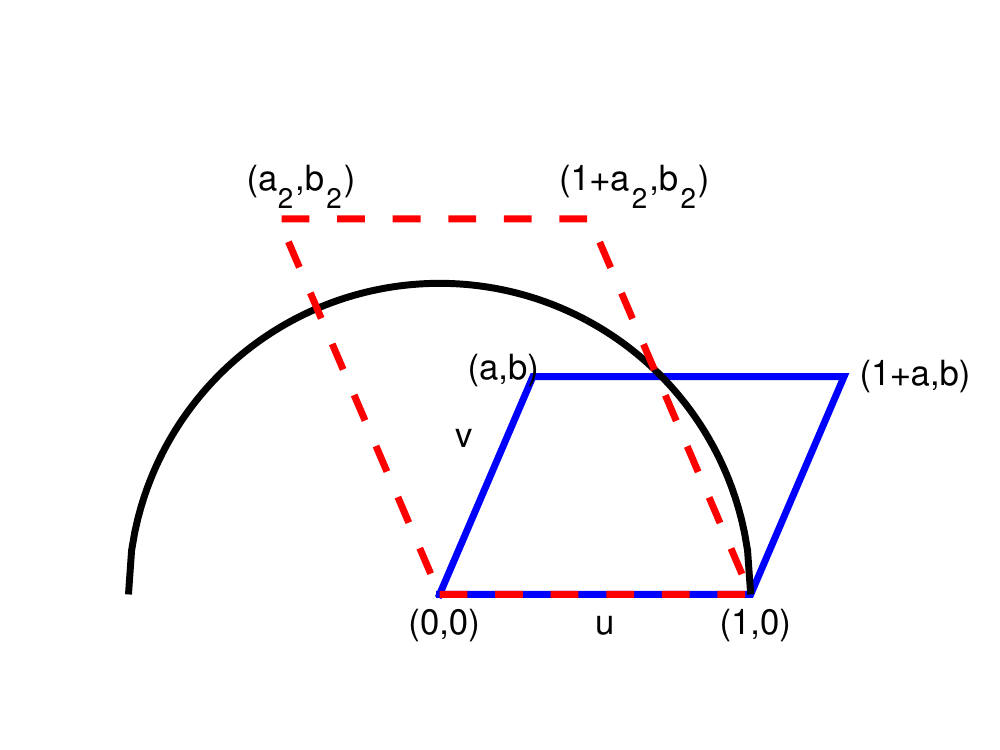}
\caption{An illustration of the transformations of flat tori in Proposition \ref{prop:FlatToriTransforms}. See \S\ref{sec:FlatTori}. }
\label{fig:ToriTransform}
\end{center}
\end{figure}

\begin{prpstn}
\label{prop:FlatToriTransforms}
The value of $\Lambda_k(a,b) := b\cdot  \lambda_k (a,b)$ is invariant under the transformations
\begin{align*}
(a,b) \mapsto (-a,b),  \qquad (a,b) \mapsto (a+1,b), \quad  \mathrm{and} \quad  (a,b) \mapsto \left( \frac{-a}{a^2 + b^2},  \frac{b}{a^2 + b^2} \right). 
\end{align*}
\end{prpstn}
\begin{proof}
The first transformation is an isometry of the flat torus and leaves the spectrum, and hence $\Lambda_k$, invariant. 

Suppose that $\psi_{a,b}(u,v)$ is an eigenfunction of the $(a,b)$-flat torus. Define the function on the $(a+1,b)$-flat torus, 
\[
\psi_{a+1,b}(u,v)=\left\{ \begin{array}{ccc}
u_{a,b}(u,v) & \mbox{if} & v>\frac{b}{a}(u-1)\\
\psi_{a,b}(u-1,v) & \mbox{if} & v\le\frac{b}{a}(u-1).
\end{array}\right.
\]
Since $\psi_{a,b}(u,v)$ is periodic, $\psi_{a+1,b}(u,v)$ is periodic too. The function
constructed is  an eigenfunction of the flat tori $(a+1,b)$
with the same eigenvalue.

To check invariance with respect to the third transformation, we consider the mapping
\[
(\tilde{x},\tilde{y})=(-y,x), \quad \text{and} \quad  
\left(\tilde{a},\tilde{b}\right)=\left(\frac{-a}{a^{2}+b^{2}},\frac{b}{a^{2}+b^{2}}\right).
\]
We then have
\[
\Delta_{(\tilde{a},\tilde{b})}^{\tilde{x},\tilde{y}} \tilde u
=\frac{4\pi^{2}}{\left(\frac{b}{a^{2}+b^{2}}\right)^{2}}\left[\left(\left(\frac{-a}{a^{2}+b^{2}}\right)^{2}+\left(\frac{b}{a^{2}+b^{2}}\right)^{2}\right)\tilde u_{\tilde{x}\tilde{x}}-2\left(\frac{a}{a^{2}+b^{2}}\right)\tilde u_{\tilde{x}\tilde{y}}+\tilde u_{\tilde{y}\tilde{y}}\right]=\lambda \tilde u
\]
\[
\implies \qquad \frac{4\pi^{2}}{b}\left[u_{yy}-2au_{xy}+(a^{2}+b^{2})u_{xx}\right]=\lambda\left(\frac{b}{a^{2}+b^{2}}\right)u=\lambda\tilde{b}u
\]
Thus, the spectrum scales by the factor $\frac{1}{a^2 + b^2}$, but $\Lambda_k$ is invariant. 
\end{proof}

%\begin{rmrk} Proposition \ref{prop:FlatToriTransforms} can also be proven using \eqref{eq:EigsFlatTori}. The first transformation can be seen from $(c_1, c_2) \mapsto(-c_1,c_2)$. The second transformation  is due to the $SL(2,\mathbb Z)$ invariance of $\mathbb Z^2$, take 
%$\begin{pmatrix} c_1 \\ c_2 \end{pmatrix} \mapsto 
%\begin{pmatrix} 1 &0 \\ -1 &1 \end{pmatrix}
%\begin{pmatrix} c_1 \\ c_2 \end{pmatrix}$.  
%For the third transformation, take $(c_1,c_2)\mapsto (c_2,c_1)$. Note that the composition of the first and third transformations is given by $\begin{pmatrix} c_1 \\ c_2 \end{pmatrix} \mapsto 
%\begin{pmatrix} 0 &-1 \\ 1 &0 \end{pmatrix}
%\begin{pmatrix} c_1 \\ c_2 \end{pmatrix}$. These two matrices are generators for $SL(2,\mathbb Z)$. 
%\end{rmrk}

Proposition \ref{prop:FlatToriTransforms} allows us to reduce the optimization problem \eqref{eq:maxEigsFlatTori} to  
\begin{equation} \label{eq:FlatF}
\Lambda^\star_k = \max \left\{ \Lambda_k(a,b) \colon (a,b)\in F\right\}. 
\end{equation}
The following proposition shows that \eqref{eq:FlatF} has a solution and gives a local maximum. 
We denote by $\left\lceil \cdot \right\rceil$ the ceiling function, $\left\lceil x \right\rceil$ for $x>0$  is the smallest integer not less than $x$.

\begin{prpstn}
\label{prop:lbtori} Fix $k\geq 1$. 
There exists a flat torus represented by a point $(a_k^\star, b_k^\star)\in F$ attaining the supremum  in \eqref{eq:FlatF}. 
Furthermore, the maximal value 
\begin{equation} \label{eq:reducedAdSet}
\tilde\Lambda_k = \max \left\{ \Lambda_k(a,b) \colon (a,b)\in F \text{ with } a^2 + b^2 \geq \left(  \left\lceil \frac{k}{2} \right\rceil -1\right)^2 \right\}
\end{equation}
has the following analytic solution 
\begin{equation}
\label{eq:flattori_exact}
\tilde\Lambda_{k} = \frac{4\pi^2 \left\lceil \frac{k}{2} \right\rceil^2}{\sqrt{\left\lceil \frac{k}{2} \right\rceil^2  - \frac{1}{4}} }, 
\end{equation}
which is attained by the $(a,b)$-flat torus with $(a,b) = \left(\frac{1}{2},\sqrt{\left\lceil \frac{k}{2} \right\rceil^2  - \frac{1}{4}}\right)$. 
The optimal value in \eqref{eq:flattori_exact} is obtained only for the integer lattice values 
$$(c_1,c_2) = (1,0), (-1,0), (1,1), (-1,-1), (0,\Big\lceil \frac{k}{2} \Big\rceil), \text{and }(0,-\Big\lceil \frac{k}{2} \Big\rceil)$$
and thus the maximal eigenvalue has multiplicity 6. 
\end{prpstn}
\begin{proof}
By Proposition \ref{prop:FlatToriTransforms}, we may restrict to the set $F$ as defined in \eqref{eq:Mad}. 
Since every eigenvalue of a flat torus has even multiplicity, without loss of generality,  we assume $k$ to be even, 
$$k=2m \quad \text{for} \quad m\in \mathbb Z.$$ 
We consider the Courant-Fischer type expression for the $k$-th eigenvalue 
\eqref{eq:EigsFlatTori} with a trial subspace of the form 
$$
E_k = \{ (0,0), \ (0,\pm1), \ldots (0,\pm m) \}. 
$$
(This is equivalent to using  \eqref{eq:CF} and a trial subspace of the form
$E_{k} = \text{span} \left\{ 1, e^{\pm \imath \ell y} \right\}_{\ell=1}^{ m } $ on the square.)
For each $k$, we obtain
$$
\Lambda_k(a,b) = b \ \lambda_k (a,b) \leq \frac{4 \pi^2 m^2}{b}.
$$
Let $\lambda_k^\square$ denote the  eigenvalues of the flat tori with $(a,b)=(0,1)$. For each $k$, define $\displaystyle \tilde b_k := \frac{C_k}{\lambda_k^\square}$. Thus, for $b>\tilde b_k$, 
$$
b \ \lambda_k (a,b) \leq 1 \ \lambda_k^\square.
$$ 
This implies that for each $k$ we can further restrict the admissible set to $\text{cl}(F) \cap \{(a,b): b\leq \tilde b_k\}$, where $\text{cl}(\cdot)$ denotes closure. Since this is a compact set, the supremum is attained.  

To show \eqref{eq:flattori_exact}, we rewrite the optimization problem using the  expression for Laplace-Beltrami eigenvalues of flat tori in \eqref{eq:EigsFlatTori},  
\begin{equation} \label{eq:JFT}
 \max_{(a,b)}   \
\min_{\substack{E \subset \mathbb Z^2\\ |E|=k+1}} \ 
\max_{(c_1,c_2)\in E}  \Lambda(a,b; c_1,c_2) 
\quad \text{where} \quad
\Lambda(a,b; c_1,c_2)  := 4\pi^2\left[\frac{(c_1 a - c_2)^2}{b} + c_1^2 b\right].
\end{equation}
In \eqref{eq:JFT}, we can rewrite 
$$
\Lambda(a,b; c_1,c_2) =  c^t A(a,b) c  
\quad \text{where}\quad 
A(a,b) = \frac{4 \pi^2}{b} \begin{pmatrix} a^2+ b^2 & -a \\ -a &1 \end{pmatrix} 
\quad \text{and} \quad
c = \begin{pmatrix} c_1\\c_2 \end{pmatrix}. 
$$
Furthermore, for every  $(a,b)\in F$,  we compute 
\begin{align*}
(\text{tr} A )^2 - 4 \text{det}(A) = \frac{16 \pi^4}{b^2} \left[ (a^2 + b^2 +1)^2 - 4 b^2\right] 
 = \frac{16 \pi^4}{b^2} \left[ a^2 +  (b-1)^2 \right] \left[ a^2 +  (b+1)^2 \right]  
 \geq 0, 
\end{align*}
which shows that each sublevel set of the  quadratic form can be viewed as an ellipse, circular for $(a,b)=(0,1)$. 
Thus, the equation for the eigenvalues of the $(a,b)$-flat torus \eqref{eq:EigsFlatTori} can be interpreted as follows. We consider increasingly large sub-level-sets of the $(a,b)$-ellipse,  {\it i.e.}, $\{ (x,y)\colon \Lambda(a,b;x,y) \leq \gamma\}$ for  increasing $\gamma$. Eigenvalues occur every time the sub-level-sets of the ellipse enclose a new integer lattice point. 
We thus interpret \eqref{eq:JFT} as finding the $(a,b)$-parameterized ellipse for $(a,b)\in F$ whose $k$-th smallest enclosed value on the integer lattice is maximal. 

When $k=1$ (or equivalently, $k=2$), we have from  \eqref{eq:JFT} that 
\begin{align*}
\tilde \Lambda_1 \ = \  \max_{(a,b)\in F}   \ \left\{  \min_{c \in E\setminus (0,0)}  \ c^t A(a,b) c \right\}
 \ \leq \  \max_{(a,b)\in F}   \  ( 0,  1)  A(a,b) ( 0,  1 )^t 
\ = \ \max_{(a,b)\in F}   \  \frac{4 \pi^2}{b}  
 \  = \   \frac{8 \pi^2}{\sqrt{3}} . 
\end{align*}
However, if we choose $(a,b) = \left( \frac{1}{2}, \frac{\sqrt 3}{2} \right)$, and solve the inner optimization problem in \eqref{eq:JFT}  to find the normalized eigenvalue, we obtain $\Lambda_1(a,b) = \frac{8 \pi^2}{\sqrt{3}}$. This implies that $\tilde \Lambda_1 =  \frac{8 \pi^2}{\sqrt{3}}$. 

Thus we can assume $k>2$. Let $m>1$ and $k=2m$.  Observe  that for $b>m$, the first $k$ nontrivial eigenvalues are obtained from  \eqref{eq:JFT}  by choosing $c_1=0$ and $c_2= \pm1, \pm2, \ldots,\pm m$. In this case, we find that $\Lambda_k = 4 \pi^2 m^2/b \leq 4 \pi^2 m$, attained in the case where $(a,b)=\left(\frac{1}{2}, m\right)$. We conclude that $\tilde \Lambda_{2m} \geq 4 \pi^2 m$ and that we can restrict the admissible set  to $b\leq m$.

We consider candidate subsets $E^j_k \subset \mathbb Z^2$, $j=1,2$ of the form 
\begin{align*}
E_k^1 &=   \{ (0,0), \ (0,\pm1), \ldots \left(0,\pm m \right) \} \\
E_k^2 &=   \{ (0,0), \ (0,\pm1), \ldots \left(0,\pm (m-1) \right), (\pm1,0) \} \\
\end{align*}
From    \eqref{eq:JFT}, we have that 
\begin{align*}
\tilde \Lambda_k & \leq  \max_{\substack{(a,b)\in F \\ b\leq m}}   \  \min_{j=1,2} \ \max_{c \in E^j_k} \ \Lambda(a,b,c_1,c_2) 
\end{align*}

We see that for $n< m$, 
$$
\Lambda(a,b,0,m)  = 4 \pi^2 \frac{m^2}{b} \geq 4 \pi^2  \frac{n^2}{b} = \Lambda(a,b,0,\pm n)
$$
and so  the elements in $E_k^1$ are dominated by  $(c_1,c_2) = (0,m)$. Thus,  
$$
\max_{c \in E^1_k} \ \Lambda(a,b,c_1,c_2)  =  \frac{4 \pi^2 m^2}{b} . 
$$

Looking at $E_k^2$, we have to compare the functions  $\Lambda(a,b,1,0) = 4 \pi^2 \left( \frac{a^2 + b^2}{b} \right) $ and $\Lambda(a,b,0,m-1) = \frac{4 \pi^2 (m-1)^2}{b} $. If $ a^2 + b^2 \ge (m-1)^2$ then the first term dominates. 
Thus, we have shown that if $ a^2 + b^2 \geq (m-1)^2$ then 
$$
\Lambda_k(a,b)  \ \leq \ 4 \pi^2 \cdot \max_{\substack{(a,b)\in F \\ b\leq m}}  \
  \min \ \left\{  \frac{a^2}{b} + b , \ \frac{m^2}{b}  \right\} 
  \ \leq  \ 4 \pi^2 \cdot \max_{\frac{\sqrt 3}{2} \leq b \leq m}  \
  \min  \left\{  \frac{1/2}{b} + b , \ \frac{m^2}{b}  \right\} 
$$
The first term is increasing for $b \geq \frac{1}{\sqrt{2}}$. The second term is decreasing in $b$. The optimal value of $b$ is found by setting the two terms equal to each other. They are equal  at $b = \sqrt{m^2 - 1/4}$ with value $ \frac{4 \pi^2 m^2}{\sqrt{m^2 - 1/4}}$. 
Thus, for all $ (a,b) \in F$, with $a^2 + b^2 \geq (m-1)^2$ we have that 
$$
\Lambda_k(a,b) \leq  \frac{4 \pi^2 m^2}{\sqrt{m^2 - 1/4}}. 
$$
with equality for 
$(a,b) = \left(\frac{1}{2},\sqrt{m^2  - 1/4}\right)$.  Equation  \eqref{eq:flattori_exact} then follows from the substitution $m\mapsto  \left\lceil \frac{k}{2} \right\rceil$.
\end{proof}

\begin{rmrk}
We note that the admissible sets in \eqref{eq:reducedAdSet} and \eqref{eq:FlatF} agree for $k=1,2,3,4$ and thus, the local maximum for \eqref{eq:reducedAdSet} given in Proposition \ref{prop:lbtori} is the global solution for \eqref{eq:FlatF}. 
 In particular, we recover  the result of  \cite{Berger1973} that $\frac{8\pi^2}{\sqrt{3}}$ is the largest first eigenvalue for any flat torus of volume one. 
 \end{rmrk}
 
 In Figure  \ref{fig:eig_FlatTori}, we plot $\Lambda_k(a,b)$ for $k=1\ldots 16$ and $(a,b)\in F$. Each eigenvalue has multiplicity two, so only odd values of $k$  are shown. 
Note that $\Lambda_k(a,b)$ has local maxima which are not globally maxima. We tabulate the values of the maximum of $\Lambda_k(a,b)$ in Table \ref{tab:ConfofmalSpectrum} for $k=1,\ldots,8$. 

\begin{rmrk}
We conjecture that the solutions to the optimization problems in  \eqref{eq:reducedAdSet} and \eqref{eq:FlatF} agree. 
 According to the proof of Proposition \ref{prop:lbtori}, this conjecture is equivalent to the statement: for  $a^2 + b^2 < (m-1)^2$ with $m\geq 3$, the ellipse 
$$
E(a,b) = \left\{ c\in \mathbb R^2 \colon c^t A(a,b) c \leq  \frac{4 \pi^2 m^2}{\sqrt{m^2 - 1/4}} \right\}
$$
contains at least $1+2m$ integer points. 
\end{rmrk}

\begin{figure}[h!]
\begin{center}
\vspace{-1cm} 
\includegraphics[width=14.2cm]{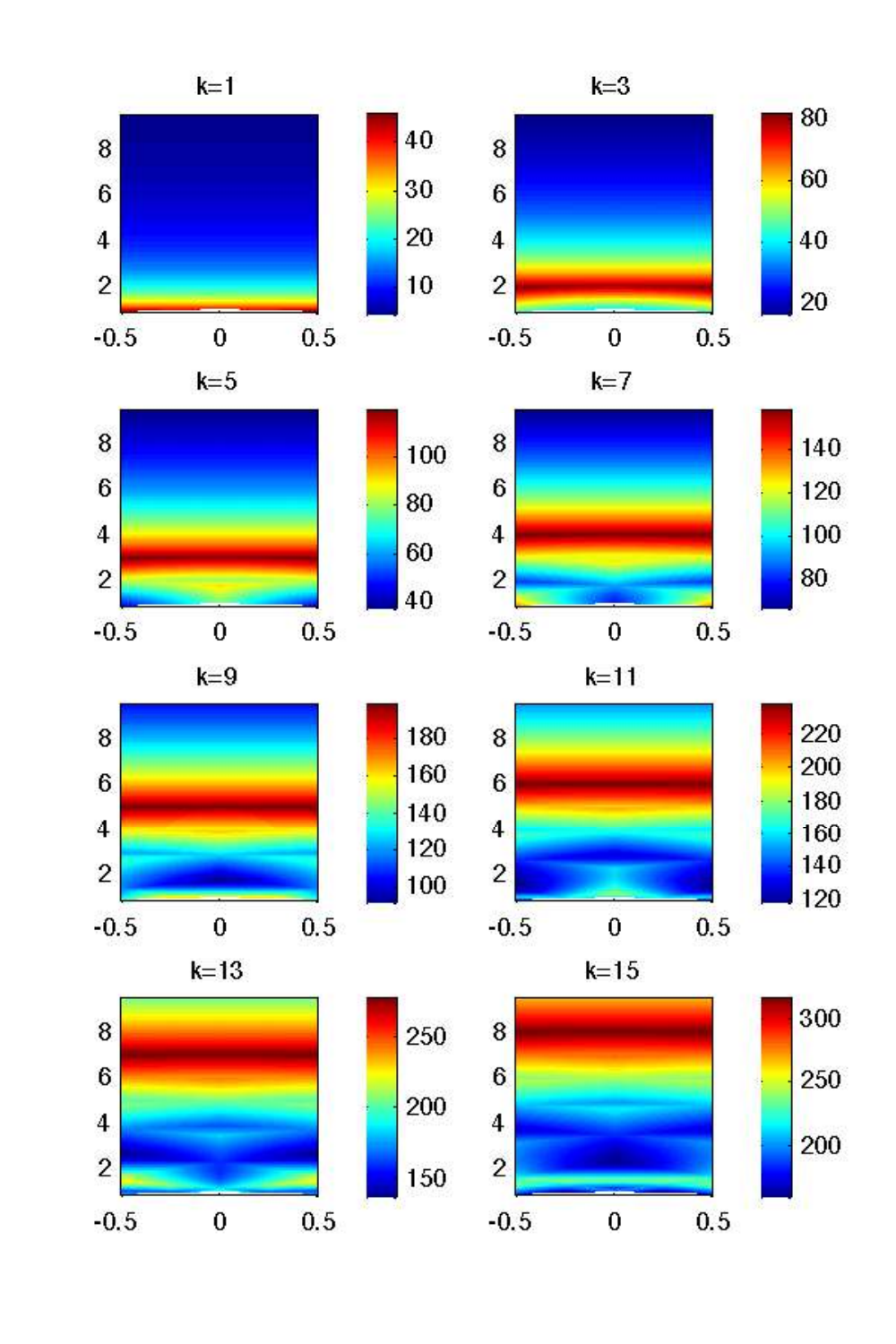} 
\vspace{-2cm} 
\caption{The first 16 volume-normalized eigenvalues, $\Lambda_k(a,b)$, of flat tori plotted as a function of the tori parameters $(a,b)$. Each eigenvalue has multiplicity two, so only odd eigenvalues are shown. 
See \S\ref{sec:FlatTori}.}
\label{fig:eig_FlatTori}
\end{center}
\end{figure}

The maximal value for $k=2$, $\Lambda_2^\star = 45.58$,  is less than the value for the $2$-kissing spheres, $\Lambda_2 = 50.26$. Generally, for all $k\neq1,3$, the maximum value for $\Lambda_k^\star$ is less than the value for $k$ kissing spheres. Since the topological spectrum is a non-decreasing function of the genus \cite{Colbois2003}, this implies that flat tori do not attain the genus $\gamma=1$ topological spectrum for $k\neq1,3$. Since, by \eqref{eq:SpectralGap},  $\Lambda^t_3(1) \geq \Lambda^t_1(1) + 4\pi  \approx 95.85$,  a flat tori also does not attain the genus $\gamma=1$ topological spectrum for $k=3$. Thus, for $k\geq 2$, to study the topological spectrum, we require an  inhomogeneous conformal factor.

\subsection{Spectrum of embedded tori} \label{sec:tori}
To provide another comparison, we consider the torus embedded in $\mathbb R^3$ with parameterization,
$$
x(u,v) = \left( (r \cos u  + R) \cos v , (r \cos u  + R) \sin v, r \sin u \right),  \qquad u,v \in [0,2\pi] . 
$$
Here $r>0$ is the minor radius, $R>r$ is the major radius, $u$ is the poloidal coordinate, and $v$ is the toroidal coordinate. See Figure  \ref{fig:eigsTori}. We consider the metric induced from $\mathbb R^3$, 
$$g(u,v) = 
\begin{pmatrix}
r^2 & 0 \\ 0 & (r \cos u + R)^2 
\end{pmatrix}. 
$$
From \eqref{eq:LapBelomega}, we obtain the Laplace-Beltrami operator
$$
\Delta f = r^{-2 } \left( r \cos u + R \right)^{-1} \partial_u  \left( r \cos u + R \right) \partial_u f + \left( r \cos u + R \right)^{-2} \partial_v^2 f.  
$$
Noting that the Laplace-Beltrami eigenvalue problem $-\Delta \psi = \lambda \psi$ is separable, we take $\psi(u,v) = \phi(u) e^{\imath m v} $ for $m\in \mathbb N$ to obtain the periodic eigenvalue problem on the interval $[0,2\pi]$, 
\begin{equation}
\label{eq:seperatedEigProb}
-r^{-2} \partial_u^2 \phi + r^{-1} \sin u \left( r \cos u + R \right)^{-1} \partial_u \phi + m^2 \left( r \cos u + R \right)^{-2} \phi = \lambda \phi . 
\end{equation}
Note that the eigenvalues for $m>0$ have multiplicity  at least two.  
We obtain spectrally accurate solutions to \eqref{eq:seperatedEigProb} using the \verb+Chebfun+ Matlab toolbox \cite{Driscoll2014}. 
Let $\mathbb T_a^2$ denote the torus with volume $(2 \pi)^2 R r = 1$ and (squared) aspect ratio 
$a^2 = R/r>1$. In Figure  \ref{fig:eigsTori}, we plot the volume-normalized Laplace-Beltrami eigenvalues, $\Lambda_k(a) := \lambda_k(\mathbb T_a^2,g) \cdot \vol(\mathbb T_a^2,g)$,  as a function of the aspect ratio, $a$. 
We remark that a similar figure appears in \cite{Glowinski2008}, where the eigenvalues are computed using a finite difference method. Numerical values of the eigenvalues for the horn torus ($a=1$) are listed in Table \ref{tab:ConfofmalSpectrum} for comparison. 

\begin{figure}[t]
\begin{center}
 \includegraphics[width=6cm]{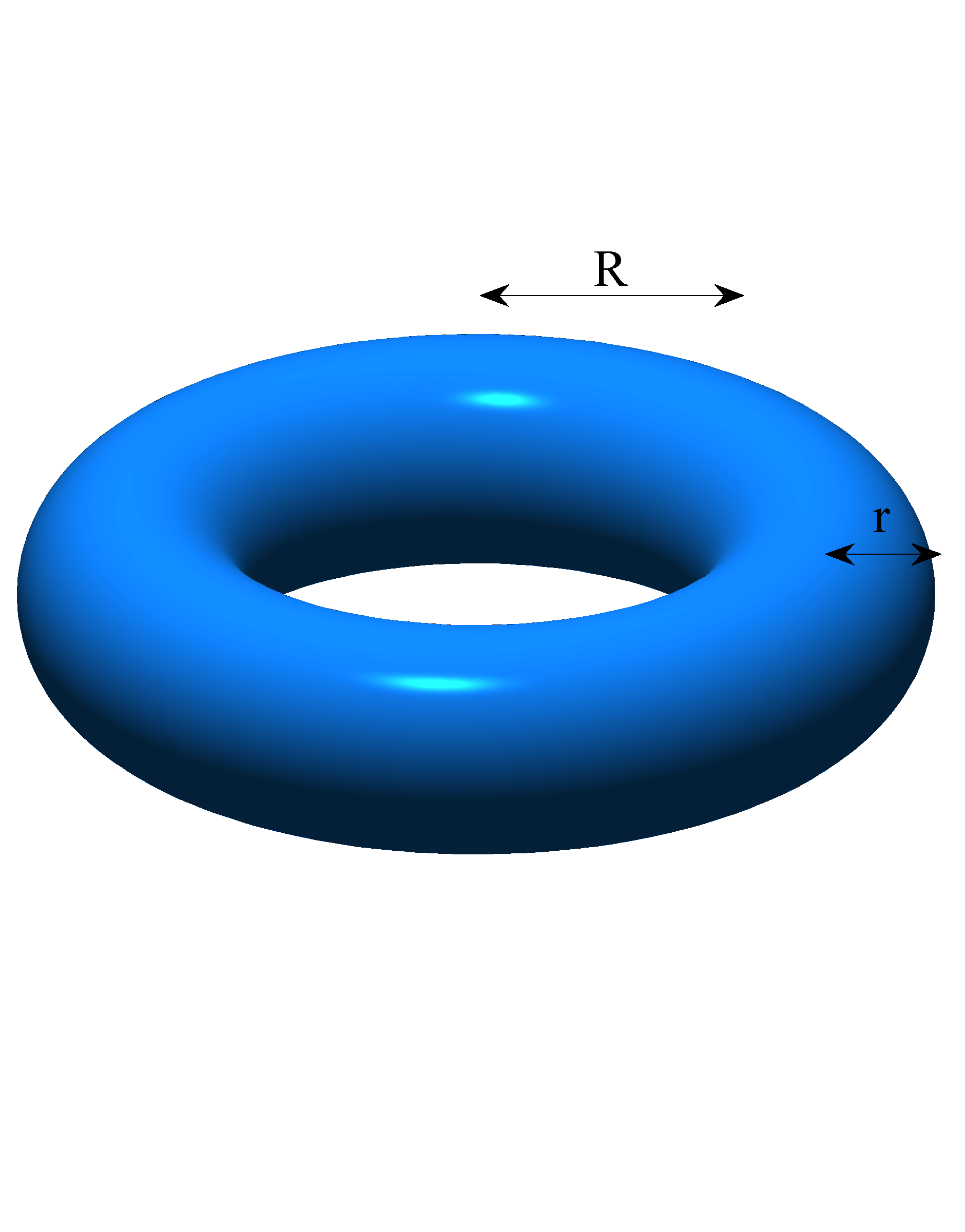}
 \includegraphics[width=10cm]{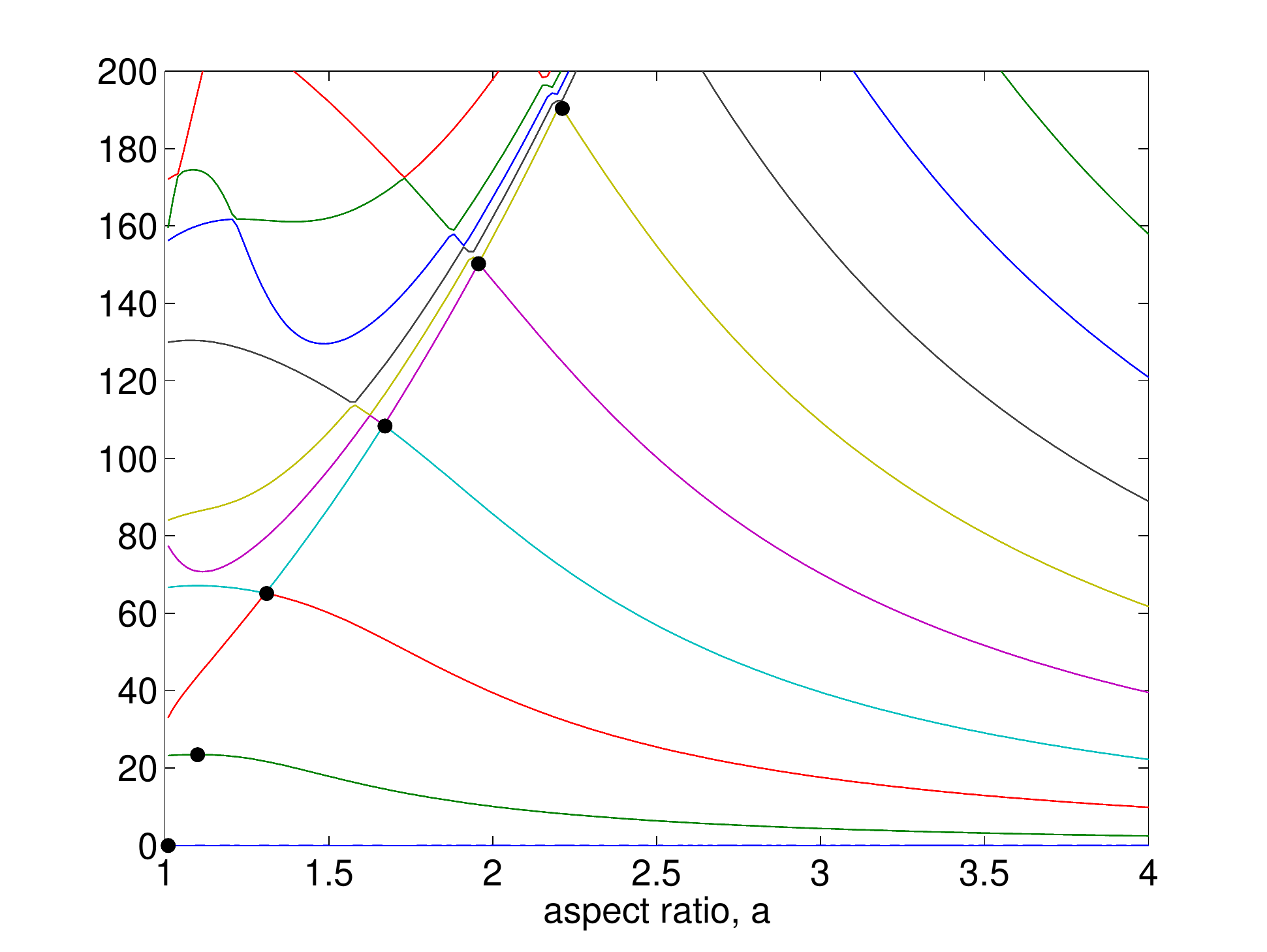}
\caption{{\bf (left)} A diagram of the coordinates used for the embedded tori.  {\bf (right)} The eigenvalues of an embedded  torus with volume one as the aspect ratio is varied. See \S\ref{sec:tori}. }
\label{fig:eigsTori}
\end{center}
\end{figure}

Now, consider the problem of maximizing the $k$-th  Laplace-Beltrami eigenvalue over the aspect ratio, $a$, 
\begin{align}
\label{eq:supEmbTori}
\sup_{a\in[1,\infty)} \ \Lambda_k(a). 
\end{align}
As $a \to \infty$, for fixed $k$, it is straightforward to show using the 
Courant-Fischer formula that $\Lambda_k(a) \to 0$, so 
%Using the same method used to show Proposition \ref{prop:lbtori}, we can prove that 
there exists an $a^\star_k$ which attains the supremum in \eqref{eq:supEmbTori}.
 From Figure  \ref{fig:eigsTori}, we observe that $a^\star_k$ is an increasing sequence, corresponding to a sequence of tori with increasingly large aspect ratio. 
The numerical values of the optimal eigenvalues are listed in Table \ref{tab:ConfofmalSpectrum}. 
 The maximal eigenvalues have multiplicity greater than one. Each of the corresponding optimal eigenspaces contain an eigenfunction which is non-oscillatory in the poloidal coordinate and increasingly oscillatory in the toroidal coordinate ({\it i.e.}, the first eigenfunction of \eqref{eq:seperatedEigProb} for an increasing sequence in $m$). Compared to, {\it e.g.}, the flat tori studied in \S\ref{sec:FlatTori}, these maximal eigenvalues are relatively small and will not be further discussed. 

\section{Computational methods} \label{sec:compMeth}
In this section, we introduce a numerical method  for approximating the conformal  and topological spectra of a Riemannian surface $(M,g)$, as given in  \eqref{eq:ConEig} and \eqref{eq:TopEig}. 
Our method is an adaption of the  methods found in \cite{Oudet,Osting:2010,AntunesFreitas,OK12,OK12b} for shape optimization problems involving extremal eigenvalues of the Laplacian to the  setting of Laplace-Beltrami eigenvalues of Riemannian surfaces using the computational tools developed  in \cite{Lai2011,Shi2011conformal}. 
Our approach is to approximate \eqref{eq:ConEig} and \eqref{eq:TopEig} using \eqref{eq:opt:Linfty} and \eqref{eq:opt:LinftyTop} respectively, as justified by Proposition 
 \ref{prop:Existence}. 
 
For the computation of Laplace-Beltrami eigenpairs, we use both finite element and spectral methods, which we describe in  \S\ref{sec:homer}. Generally spectral methods are more accurate than finite element methods, but are difficult to implement for general surfaces. Therefore, we use spectral methods for computations on the torus and finite element methods for computations on other surfaces.

We numerically solve the optimization problem in \eqref{eq:opt:Linfty} as follows. 
 For a fixed surface, $(M,g_0)$, we evolve $\omega$ within $\mathcal A(M,g_0,\omega_-, \omega_+)$ to increase $\Lambda_k(M,g_0,\omega)$. 
At each iteration, the variation of $\Lambda_k(M,g,\omega)$ with respect to the conformal factor is computed using Proposition \ref{prop:Hadamard}, as described below in \S\ref{sec:gradientFlow}. 
This can be viewed as an ``optimize-then-discretize'' approach to the problem, where the analytically computed gradient is evaluated   using discretized quantities. This is in contrast to the ``discretize-then-optimize'' approach in which a finite dimensional version of the problem would be formulated and the gradient of the discretized objective function would be used. 
The BFGS quasi-Newton method is then used to determine a direction of ascent, in which the metric is evolved for a step-length determined by an Armijo-Wolfe line search. A log-barrier interior-point method is used to enforce $L^\infty(M)$ constraints. The process is iterated until a metric $g$ satisfying convergence criteria is obtained. Metrics obtained by this approach are (approximately) local maxima of $\Lambda_k(M,g)$, not necessarily global maxima. We repeat this evolution for many different choices of initial metric and choose the conformal factor which yields the largest value of $\Lambda_k(M,g)$. 

For the solution of the optimization problem in \eqref{eq:opt:LinftyTop}, we additionally must consider a parameterization of the conformal classes. For genus $\gamma=1$, this parameterization $(a,b) \in F$ is described in \S\ref{sec:moduliSp} and illustrated in Figure \ref{fig:coordDefs}(right). We use the same strategy as for  \eqref{eq:opt:Linfty}, except we also evolve the parameters $a$ and $b$ to increase $\Lambda_k(M,g,\omega)$. The derivatives of $\lambda_k(M,g,\omega)$ with respect to the parameters $a$ and $b$ are computed  in \S\ref{sec:gradientFlow}.

The reader may have noticed that we use Hadamard's formula (Proposition \ref{prop:Hadamard}) to compute the variation of $\lambda_k(M,\omega g_0)$ with respect to the conformal factor, $\omega$, and this formula is only valid for simple eigenvalues. It is well-known that eigenvalues $\lambda_k(M,g)$ vary continuous with the metric $g$, but are not differentiable when they have multiplicity greater than one.  In principle, for an analytic deformation $g_t$, left- and right-derivatives of  $\lambda_k(M,g_t)$ with respect to $t$ exist \cite{Soufi2008,penskoi2013c} and could be computed numerically. However, in practice, 
eigenvalues computed numerically that approximate the Laplace-Beltrami eigenvalues of a surface are always simple. This is due to discretization error and finite precision. Thus, we are faced with the problem of maximizing  a function that we know to be non-smooth, but whose gradient  is well-defined at points in which we sample. For a variety of such non-smooth problems, the BFGS quasi-Newton method with an inexact line search has proven to be very effective \cite{LewisOverton2013}, but the convergence theory remains sparse. In particular, for this problem, a gradient ascent algorithm will generate a sequence of conformal factors where the $k$-th and $(k+1)$-th eigenvalues will converge towards each other. The sequence will become ``stuck'' at this point and the objective function values will be relatively small compared to the optimal value.  As reported in other computational studies of extremal eigenfunctions  \cite{Osting:2010,AntunesFreitas,OK12b}, for this problem we observe that a BFGS approximation to the Hessian avoids this phenomena. 

Finally, in Proposition  \ref{prop:Existence}, we introduced two constants $\omega_+$ and $\omega_-$ which provide point-wise bounds on the conformal factor $\omega(x)$ for $x\in M$. An approximate  solution to \eqref{eq:ConEig} can be obtained by  computing the solution to  \eqref{eq:opt:Linfty}  for a sequence of values $\omega_+$ and $\omega_-$ such that $\omega_+ \uparrow \infty$ and $\omega_- \downarrow 0$. In practice, we fix $\omega_+$ and $\omega_-$ to be large and small constants respectively. Taking sequences tending to $\pm \infty$ would be a poor idea as conformal factors with very large or small values reduce computational  accuracy. 

In the following subsections, we describe the methods used for the computation of the Laplace-Beltrami eigenpairs, as well as compute the variation of Laplace-Beltrami eignenvalues with respect to the conformal factor  and moduli space parameters.

\subsection{Eigenvalue computation}  \label{sec:homer}
In this section, we describe the finite element and spectral methods for computing Laplace-Beltrami eigenpairs. 

\subsubsection*{Finite Element Method}
For some of our eigenpair computations, we use the finite element method (FEM) \cite{reuter2006,qiu2006,LarssonThomee,Boffi}, which we briefly describe here. The finite element method is based on the  weak formulation of  \eqref{eq:eig}, given by
\begin{equation}
\label{eqn:WeakLBeigs}
\int_{M}\nabla_{M}\psi \cdot \nabla_{M}\eta = \lambda \int_{M} \psi \eta,\quad\quad \forall~ \eta\in C^{\infty}(M).
\end{equation}
Numerically, we represent $M \subset \mathbb{R}^3$ as a triangular mesh
$\{V=\{v_i\}_{i=1}^{N},T=\{T_l\}_{l=1}^L\}$, where $v_i\in\mathbb{R}^3$ is the i-th
vertex and $T_l$ is the l-th triangle. We use piecewise linear elements to discretize the surface, so that the triangular mesh approaches the smooth surface in the $L^2$-sense as the mesh is refined. We choose linear conforming elements $\{e_i\}_{i=1}^{N}$ satisfying
$e_i(v_j)=\delta_{i,j}$, where $\delta_{i,j}$ is the Kronecker delta symbol, and write
$S=\text{span} \{e_i\}_{i=1}^{N}$. The discrete  Galerkin version of \eqref{eqn:WeakLBeigs}  is to find a
$\phi\in S$, such that  
$$
% \label{formula:discrete}
\displaystyle  \sum_l\int_{T_l}\nabla_{M}\phi \cdot \nabla_{M}
 \eta=\lambda\sum_l\int_{T_l}\phi~\eta,~\quad \forall\eta\in S. 
$$
We define
\begin{align*} 
\phi &=\sum_i^N x_ie_i\\
A_{ij} &=
\sum_l\int_{T_l}\nabla_{M}e_i\nabla_{M}
e_j\\
B_{ij} &=\sum_l\int_{T_l}e_ie_j,
\end{align*}   
where the \emph{stiffness matrix}, $A$, is symmetric and the \emph{mass matrix}, $B$, is symmetric and positive definite. Both $A$ and $B$  are sparse $N\times N$  matrices. The finite element method approximates solutions to \eqref{eqn:WeakLBeigs} by solving the generalized matrix eigenproblem, 
\begin{equation}
\label{eq:genMatEig}
Ax =\lambda Bx, \qquad \phi =\sum_i^N x_ie_i. 
\end{equation}
There are a variety  of numerical packages to solve \eqref{eq:genMatEig}. We use Matlab's built-in function \verb+eigs+ with default convergence criteria. This eigenvalue solver is based on Arnoldi's method \cite{sorensen1992implicit,lehoucq1996deflation}.
Figure \ref{fig:RateofConv} demonstrates the 2nd order of convergence in the mesh size $h ~(\sim \sqrt{N^{-1}} )$ for the Laplace-Beltrami eigenvalues of the unit sphere; see \S\ref{sec:Sphere} for explicit analytic values.  Higher eigenvalues generally have larger  error than lower eigenvalues;  higher order elements could be used for improved accuracy.

To further demonstrate the flexibility of the finite element method for computing eigenvalues of surfaces and to provide a comparison of eigenvalues for a ``typical'' embedded mesh, we also consider a surface in the shape of Homer Simpson embedded in $\mathbb R^3$, equipped with the induced metric.  This mesh has 21,161 vertices. In Figure \ref{fig:Homer}, we plot the first  8 nontrivial eigenfunctions. 
 Note that in Figure \ref{fig:Homer} and later three-dimensional  plots 
(Figures \ref{fig:KissingSpheres}, \ref{fig:KissingSpheres2}, \ref{fig:homerGlobalParam}, \ref{fig:KissingSphereFlatTorus1},
and \ref{fig:KissingSphereFlatTorus2}), 
we use a Matlab  visualization effect, achieved by the command, \verb+lighting phong+. 
Although the reflection makes it easier to see the three-dimensional structure, it also slightly distorts the color. 
Numerical values of the corresponding volume-normalized eigenvalues are listed in Table \ref{tab:ConfofmalSpectrum} for comparison. 
 We use this mesh again in \S\ref{sec:genus0} to illustrate a solution for the topological eigenvalue problem.

\begin{figure}[t]
\begin{center}
\includegraphics[width=.6\linewidth]{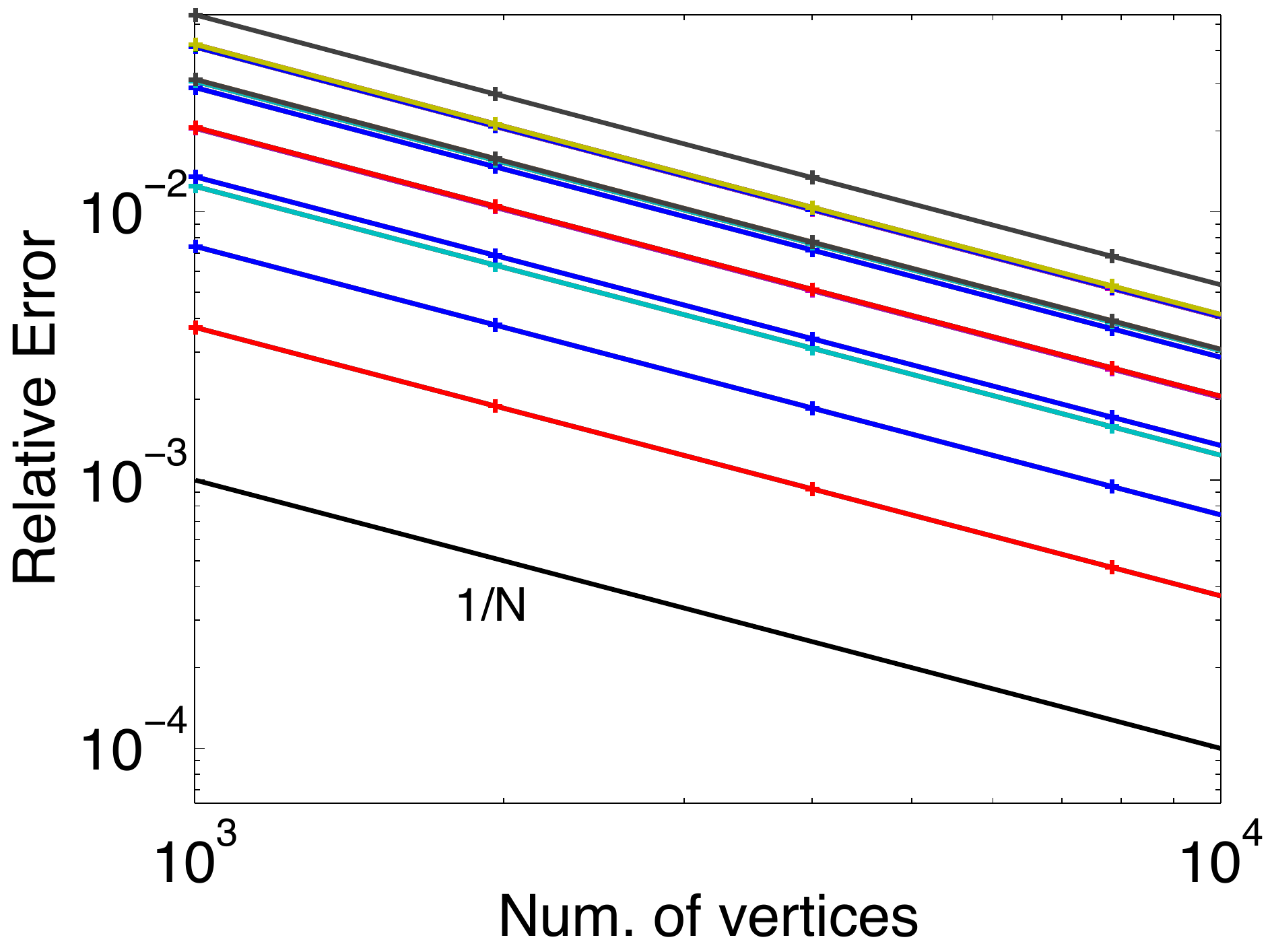}
\caption{Relative error of the finite element method for computing the first 50 Laplace-Beltrami eigenvalues on the unit sphere. Each curve in this figure represents one eigenvalue. (Lower eigenvalues are more accurate.) See \S\ref{sec:homer}. }
\label{fig:RateofConv}
\end{center}
\end{figure}

Since the finite element method approximates the variational problem \eqref{eq:CF} by a variational problem where the trial functions are taken to be a linear combination of basis functions, it overestimates the eigenvalues. This is undesirable since \eqref{eq:opt:Linfty} and \eqref{eq:opt:LinftyTop} are  \emph{maximization} problems.  Lower bounds on the eigenvalues could also be obtained numerically using non-conforming elements \cite{armentano2004}, however this is beyond the scope of this paper. 

\begin{figure}[t]
\begin{center}
\begin{minipage}{0.245\linewidth}
\centering $\lambda_1 = 7.74$
\includegraphics[width=1\linewidth]{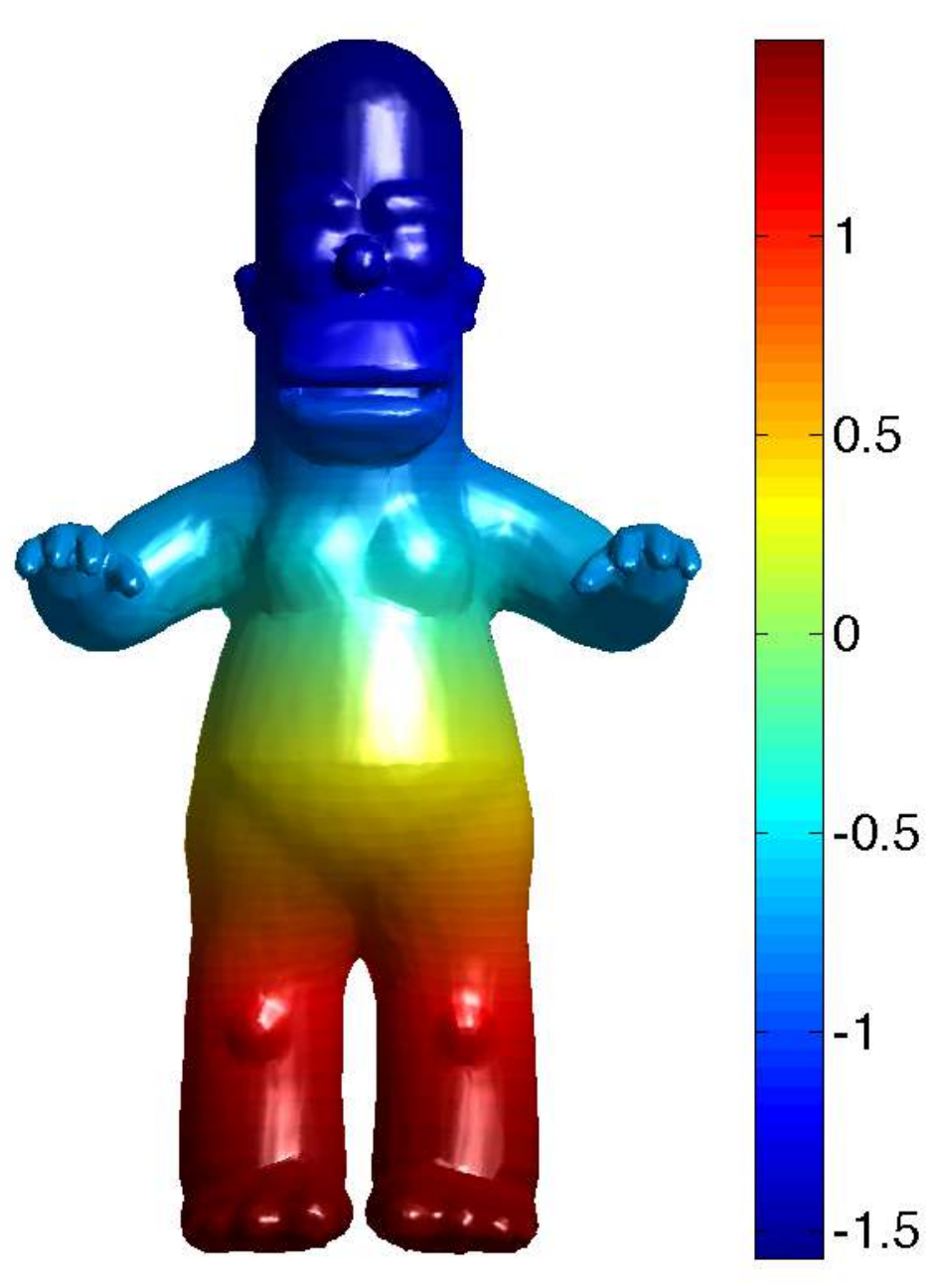}
\end{minipage}\hfill
\begin{minipage}{0.245\linewidth}
\centering $\lambda_2 = 16.98$
\includegraphics[width=1\linewidth]{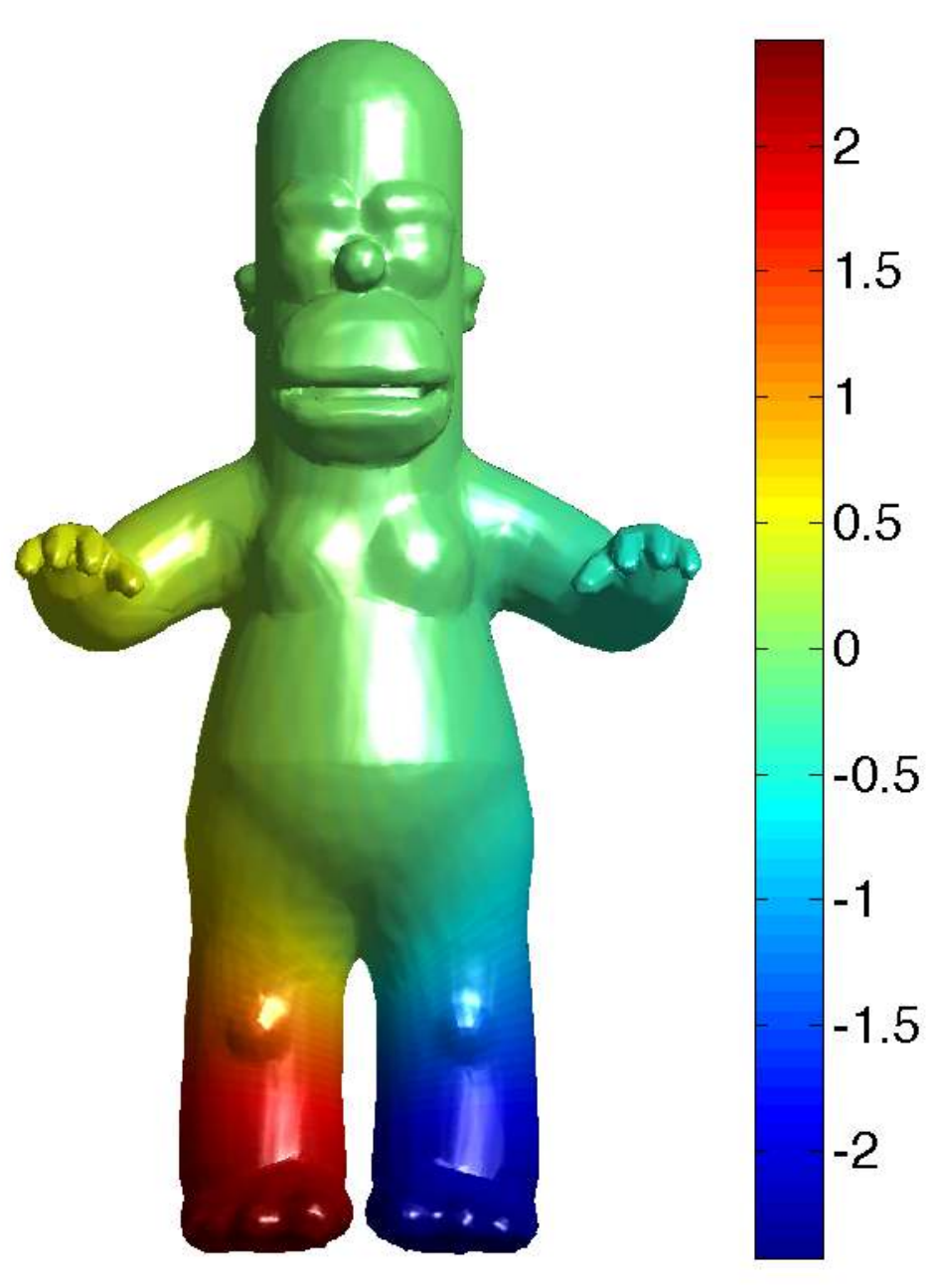}
\end{minipage}\hfill
\begin{minipage}{0.245\linewidth}
\centering $\lambda_3 = 20.58$
\includegraphics[width=1\linewidth]{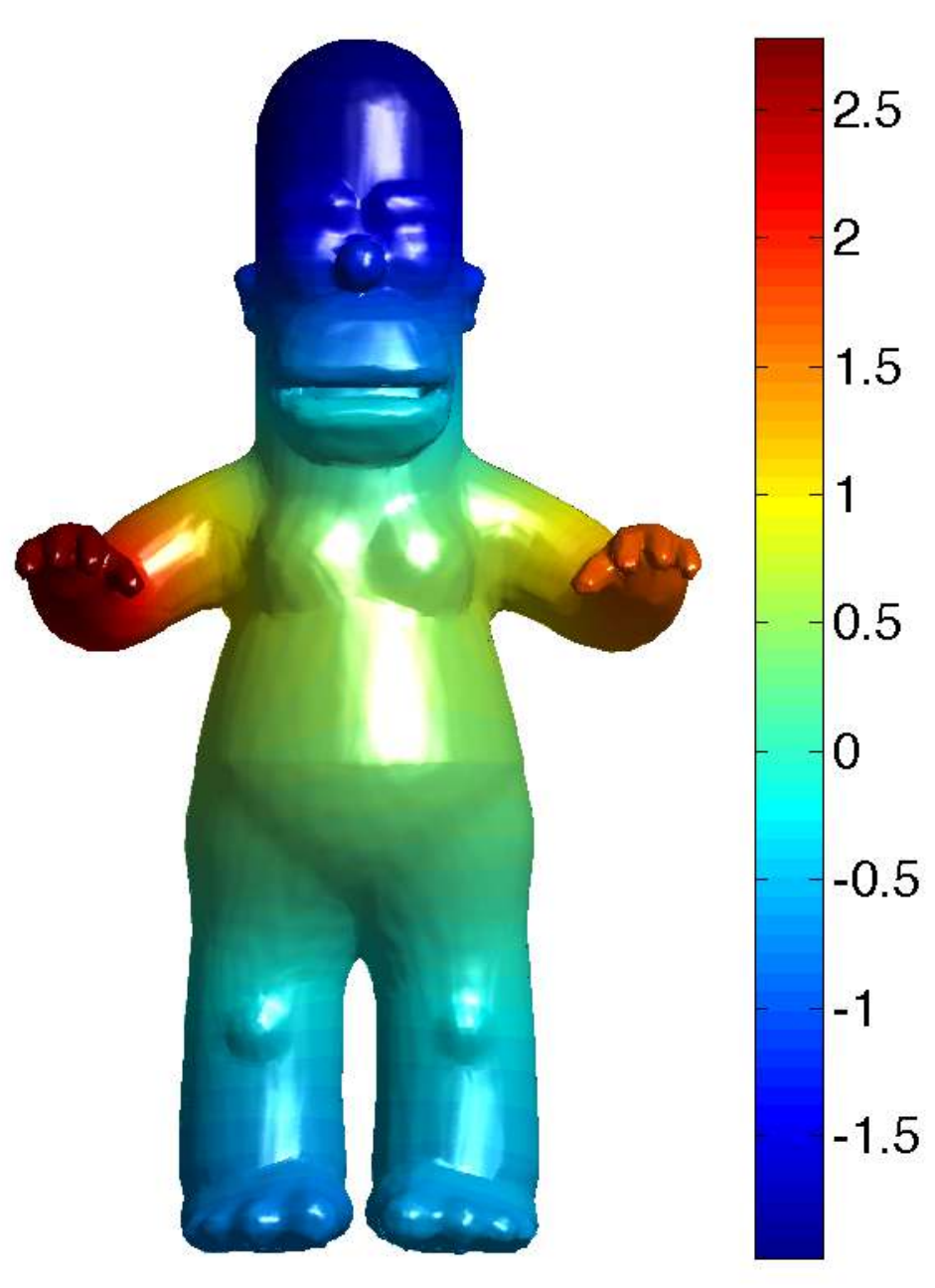}
\end{minipage}\hfill
\begin{minipage}{0.245\linewidth}
\centering $\lambda_4 = 21.53$
\includegraphics[width=1\linewidth]{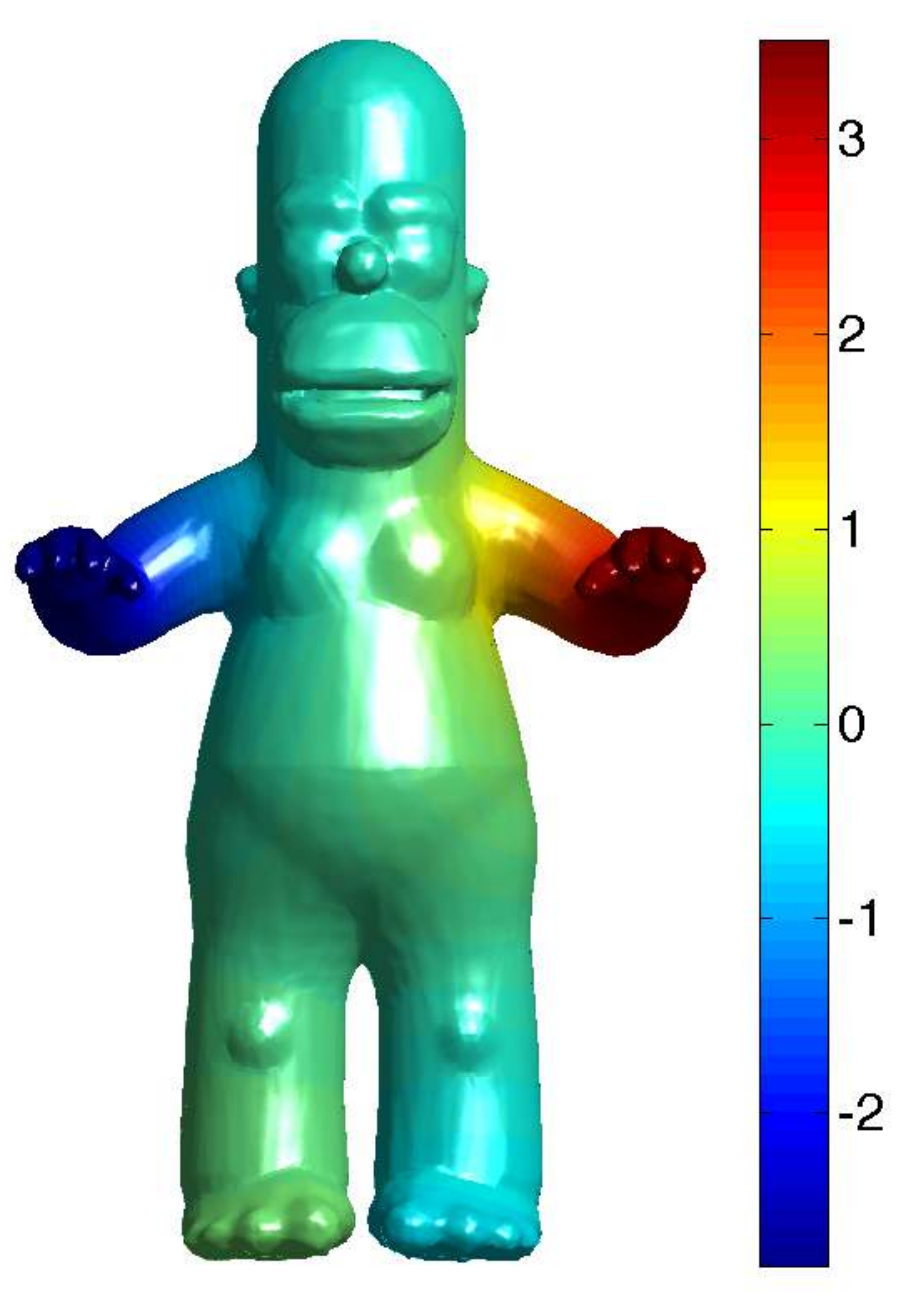}
\end{minipage}\hfill\\
\begin{minipage}{0.245\linewidth}
\centering $\lambda_5 = 42.58$
\includegraphics[width=1\linewidth]{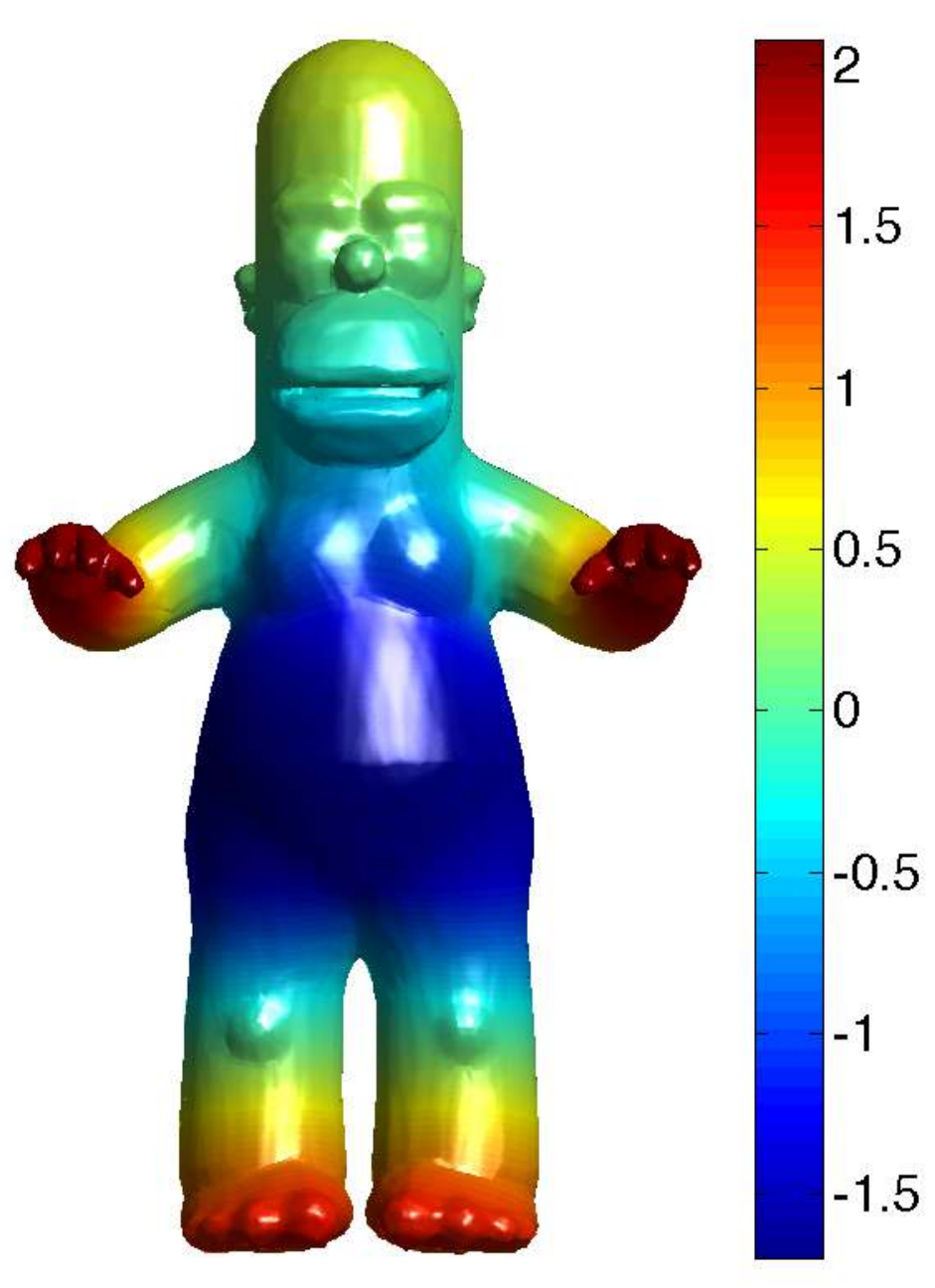}
\end{minipage}\hfill
\begin{minipage}{0.245\linewidth}
\centering $\lambda_6= 71.36$
\includegraphics[width=1\linewidth]{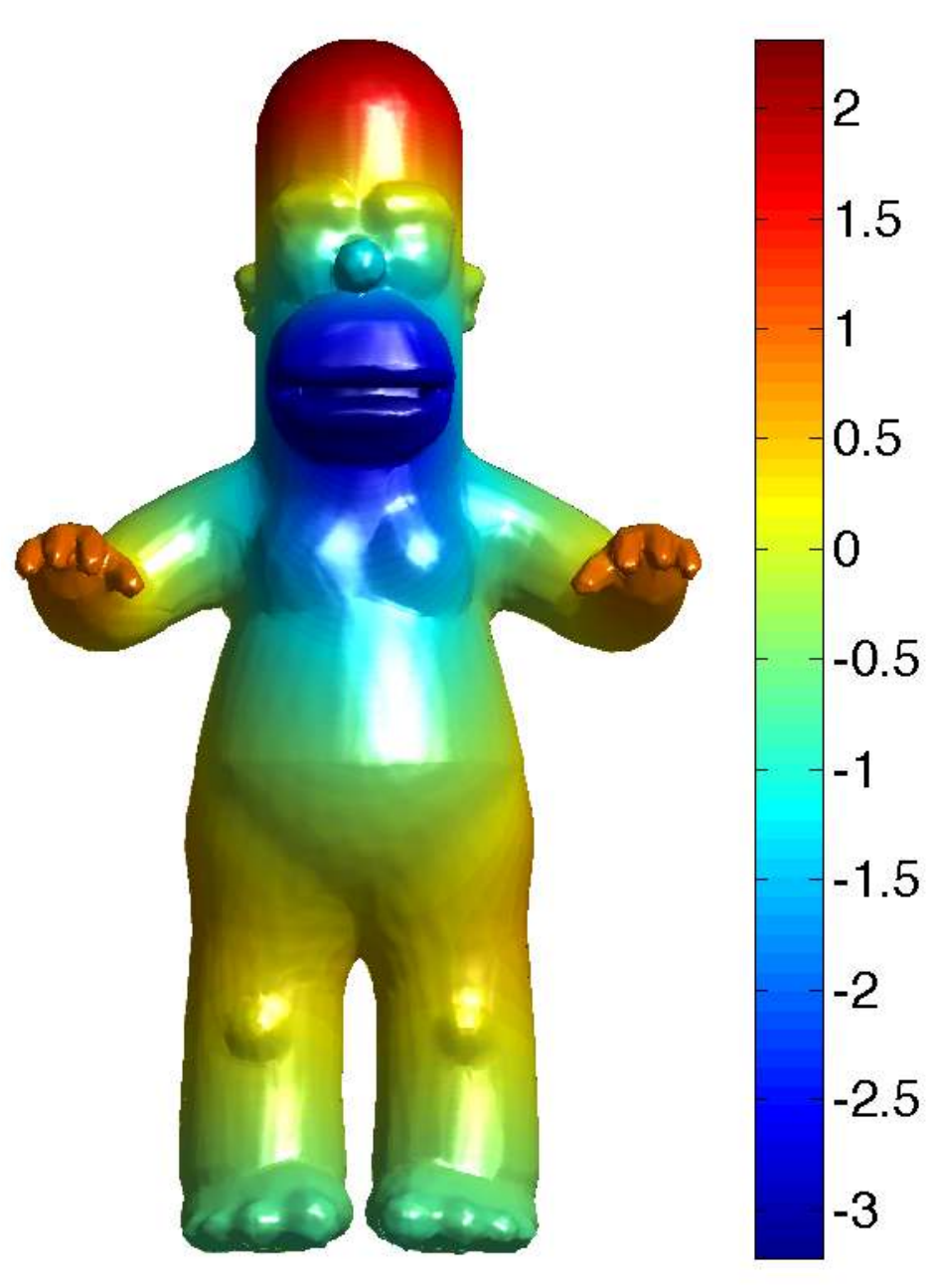}
\end{minipage}\hfill
\begin{minipage}{0.245\linewidth}
\centering $\lambda_7 = 87.92$
\includegraphics[width=1\linewidth]{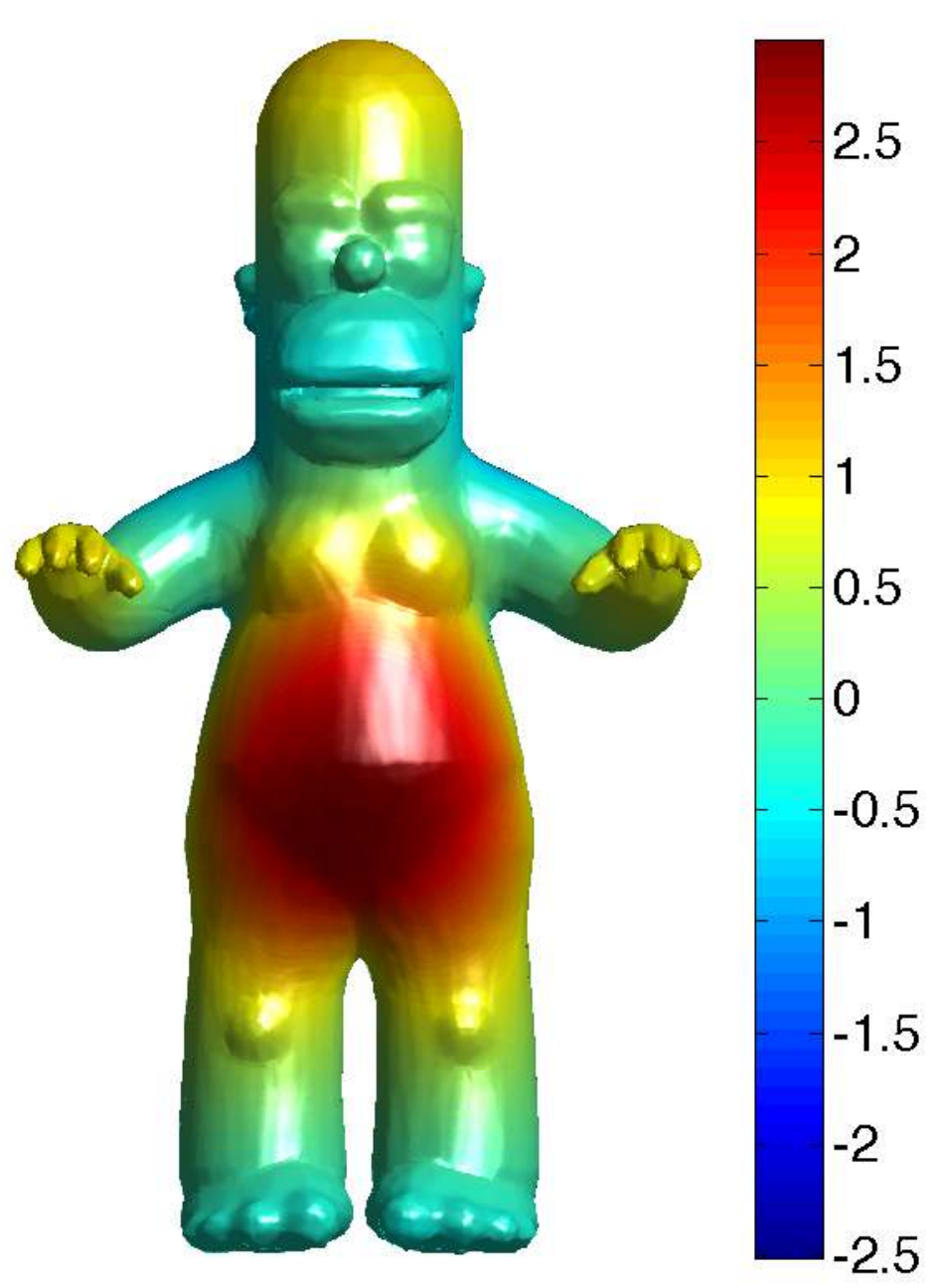}
\end{minipage}\hfill
\begin{minipage}{0.245\linewidth}
\centering $\lambda_8 = 95.38$
\includegraphics[width=1\linewidth]{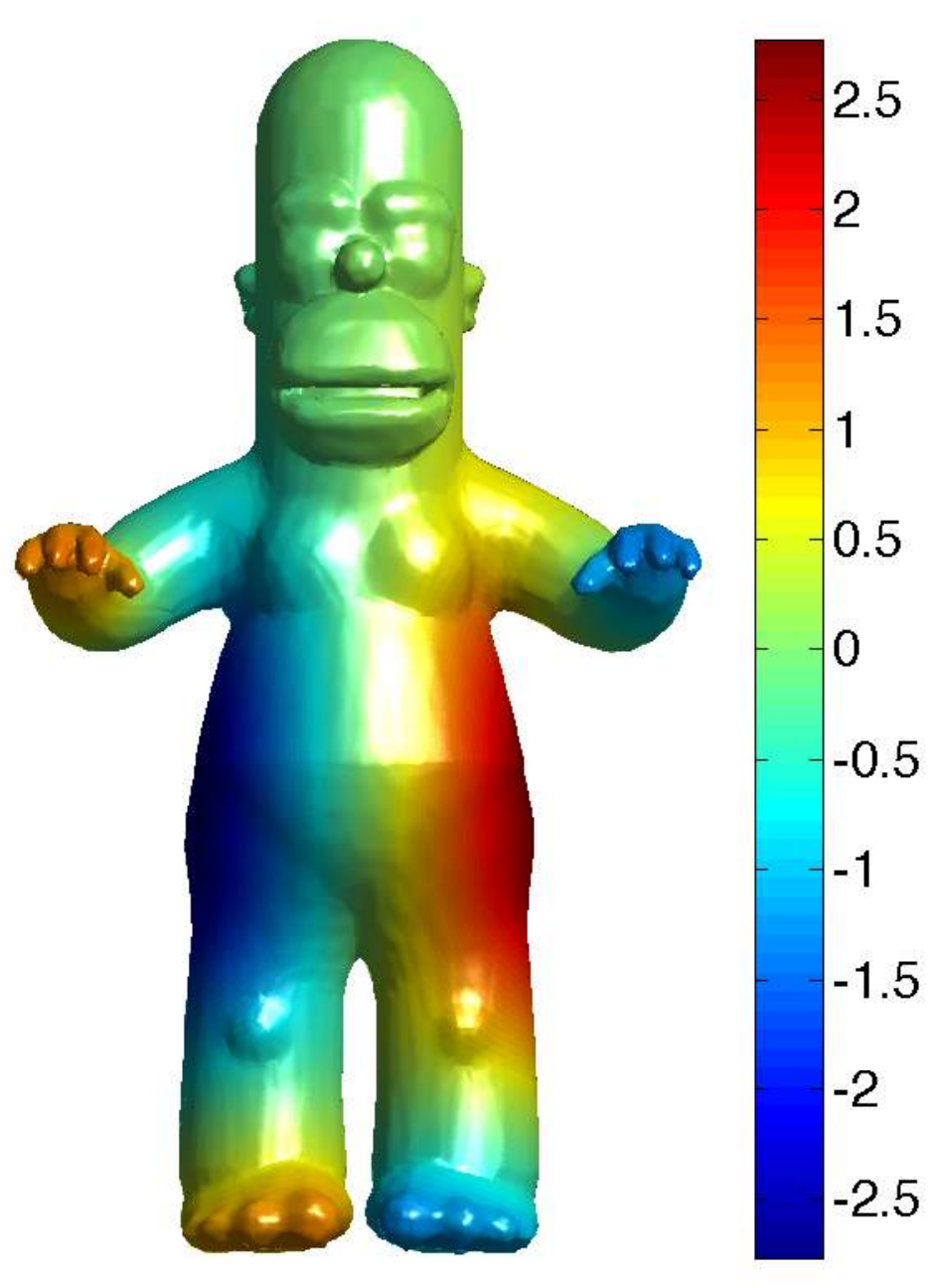}
\end{minipage}\hfill
\caption{The first eight nontrivial Laplace-Beltrami eigenfunctions on the ``Homer Simpson'' mesh. 
 See \S\ref{sec:homer}. }
\label{fig:Homer}
\end{center}
\end{figure}

\subsubsection*{Spectral Method}
For eigenvalue computations on the torus,  we use a spectral method \cite{trefethen2000}, which we briefly discuss here. 
 We use the transformation, given in \eqref{eq:LinToriTransf} and illustrated in Figure \ref{fig:coordDefs}, that takes the $(a,b)$-flat torus to the $[0,2\pi]^2$ square. The Laplace-Beltrami operator, $\Delta_{a,b}$,  on the square is defined in \eqref{eq:Delab}. 
Thus, we seek solutions to the eigenvalue problem 
\begin{equation}
\label{eq:specMethEigEq}
\Delta_{a,b} \psi = \lambda \omega \psi
\end{equation}
defined on the $[0,2\pi]^2$ square with periodic boundary conditions. The discrete operators obtained by spectral collocation   for the first and second derivatives on a one-dimensional periodic grid on $[0,2\pi]$ with (even) $N$  points are represented by the Toeplitz matrices
$$
D = \begin{pmatrix} 
0  & &&&& - \frac{1}{2} \cot\frac{1 h}{2} \\
-\frac{1}{2}\cot \frac{1 h}{2}  & \ddots && \ddots & & + \frac{1}{2} \cot\frac{2 h}{2} \\
+\frac{1}{2}\cot \frac{2 h}{2}  &  & \ddots && & - \frac{1}{2} \cot\frac{3 h}{2} \\
-\frac{1}{2}\cot \frac{3 h}{2}  &  && \ddots& & \vdots \\
\vdots  &  & \ddots && \ddots& + \frac{1}{2} \cot\frac{1 h}{2} \\
+\frac{1}{2}\cot \frac{1 h}{2}  &  &&& & 0 
 \end{pmatrix} 
 $$
 $$
D^{(2)} =  \begin{pmatrix} 
\ddots & \vdots & \\
\ddots & - \frac{1}{2} \csc^2\left( \frac{2 h }{2} \right)& \\
\ddots & + \frac{1}{2} \csc^2\left( \frac{1 h }{2} \right)& \\
& - \frac{\pi^2}{3 h^2} - \frac{1}{6}& \\
 & + \frac{1}{2} \csc^2\left( \frac{1 h }{2}\right) &\ddots \\
 & - \frac{1}{2} \csc^2\left( \frac{2 h }{2} \right)&\ddots \\
 & \vdots & \ddots
\end{pmatrix}.
$$
Here, $h = \frac{2 \pi}{N}$. See, for example, \cite[Ch. 3]{trefethen2000}.
The two-dimensional operators are then easily obtained from $D$ and $D^{(2)}$ using the Kronecker product, $\otimes$. That is, if $I$ represents the $N\times N$ identity matrix, then 
$$
D^{(2)}_{x,x} = D^{(2)}  \otimes I, 
\qquad 
D^{(2)}_{x,y} = \frac{1}{2} \left( (I \otimes D)*(D\otimes I) + (D\otimes I)*(I\otimes D)  \right), 
\quad  \text{and} \quad 
D^{(2)}_{y,y} = I \otimes D^{(2)},
$$ 
are $N^2 \times N^2$  discrete approximations to $\partial_x^2$, $\partial_{x,y}^2$, and $\partial_y^2$ respectively. 
A discrete approximation to \eqref{eq:specMethEigEq} is then given by 
$$
 \frac{4 \pi^2}{b^2} \left[ (a^2+b^2) D^{(2)}_{x,x} -2 a D^{(2)}_{x,y}+ D^{(2)}_{y,y} \right] v = \lambda \Omega v , \quad v \in \mathbb R^{N^2}. 
$$
Here $\Omega$ is a diagonal matrix with entries given by the values of $\omega$. 
This generalized eigenvalue problem is then solved using Matlab's built-in function \verb+eigs+ with default convergence criteria. In Figure \ref{fig:SpecConv}, we give a log-log plot of the relative error of the first 16 eigenvalues for the conformal factor given by  $\omega(x,y) = e^{\cos x+\cos y} $ on the equilateral torus. The method is seen to be spectrally convergent. 

\begin{figure}[t]
\begin{center}
\includegraphics[width=.7\linewidth]{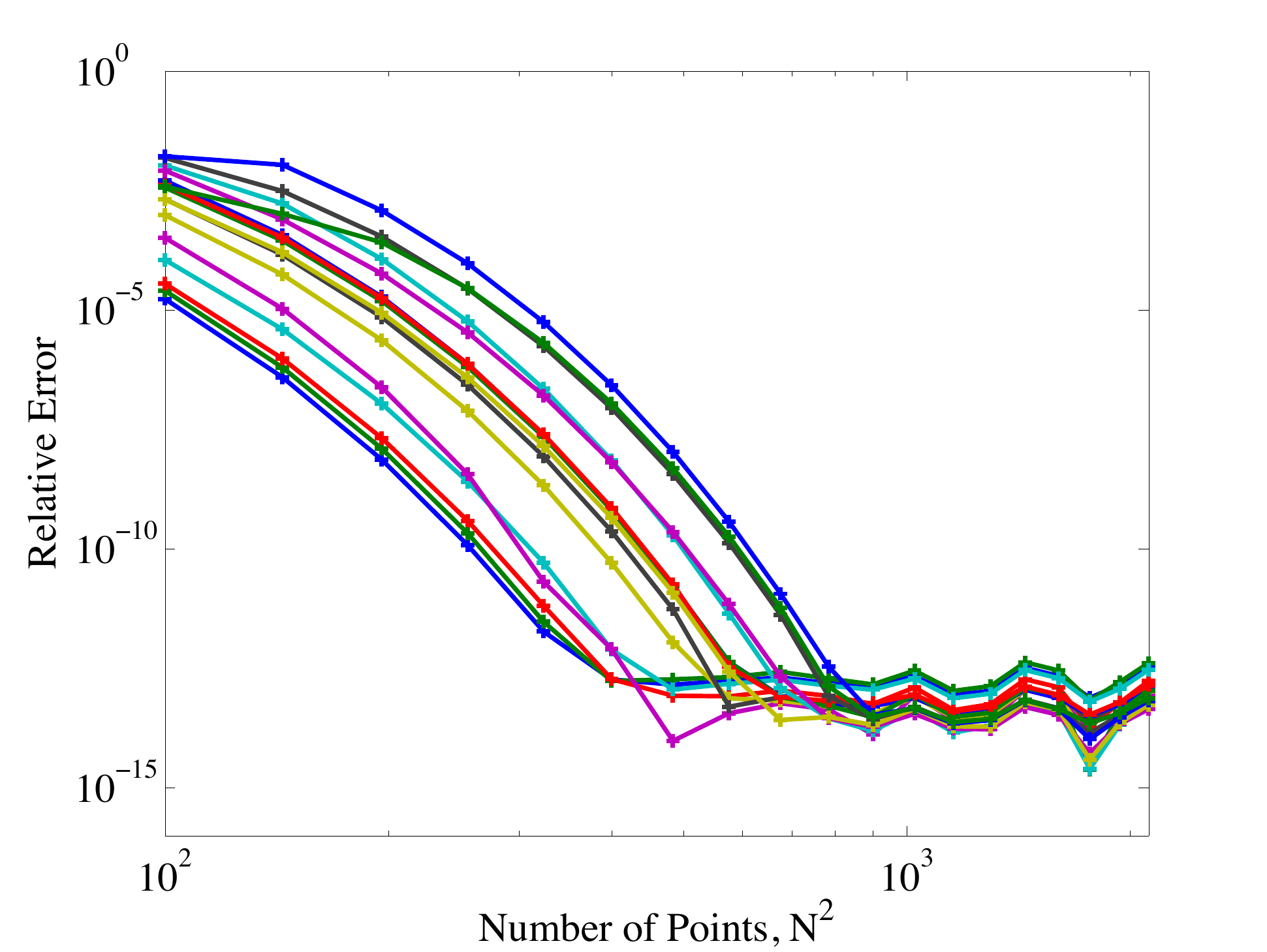}
\caption{Relative error  of the spectral method for computing  Laplace-Beltrami eigenvalues on a torus. Similar to Figure  \ref{fig:RateofConv}, each of the 16 curves in this figure represents one eigenvalue. (Lower eigenvalues are more accurate.) See \S\ref{sec:homer}.}
\label{fig:SpecConv}
\end{center}
\end{figure}

\subsection{Gradient flow of conformal factor and moduli space parameterization} \label{sec:gradientFlow}
Here, we apply Proposition \ref{prop:Hadamard} to the eigenvalues of the sphere and $(a,b)$-flat torus. The results are stated as propositions  for reference. First, consider 
 the mapping $\omega \mapsto \lambda_k(\omega)$ satisfying 
$$
- \Delta \ \psi = \omega \ \lambda(\omega) \  \psi \qquad \text{on } \   \mathbb S^2.
$$

\begin{prpstn}
\label{prop:S2derivs}
Let $\lambda(\omega)$ be a simple eigenvalue of $(\mathbb S^2, \omega g_0)$ and corresponding eigenfunction $\psi$ normalized such that $\langle \psi,\psi \rangle_{\omega g_0} =1$. Then,
$$
\frac{\delta \lambda} {\delta \omega} \cdot \delta \omega = - \lambda  \langle \psi^2 \omega^{-1}, \delta \omega \rangle_{\omega g_0},
$$
where $\langle f,h\rangle_{\omega g_0} =  \int_{\mathbb S^2} f h \omega d\mu_{g_0} $.
\end{prpstn}

We next compute the gradient of a Laplace-Beltrami eigenvalue on the $(a,b)$-flat tori with respect to both the conformal factor $\omega$ and the parameters $a$ and $b$. Recall the linear transformation introduced in \S\ref{sec:FlatTori} which takes the $[0,2\pi]^2$ square to the $(a,b)$-flat torus (see Figure  \ref{fig:coordDefs}).  
Consider the mapping $(a,b,\omega) \mapsto \lambda_k(a,b,\omega)$ satisfying 
\begin{equation}
\label{eq:FTeig}
- \Delta_{a,b} \ \psi = \omega \ \lambda(a,b,\omega) \  \psi \qquad \text{on } \   [0,2\pi]^2.
\end{equation}
where $ \Delta_{a,b}$ is defined in  \eqref{eq:Delab}. 

\begin{prpstn}
\label{prop:FTderivs}
Let $\lambda(a,b,\omega)$ be a simple eigenvalue of an $(a,b)$-flat torus with conformal factor $\omega$ and corresponding eigenfunction $\psi$ normalized such that $\langle \psi ,   \psi  \rangle_{\omega g_0}=1$. Then,
\begin{align*}
\frac{\partial \lambda} {\partial  a} &= -  \langle \psi, \omega^{-1} \Delta_a \psi \rangle_{\omega g_0} , \qquad \Delta_a :=  \frac{4 \pi^2  }{b^2} \left[ 2 a \partial_x^2 - 2 \partial_x \partial_y \right] \\
\frac{\partial \lambda} {\partial  b} &= -  \langle \psi, \omega^{-1} \Delta_b \psi \rangle_{\omega g_0}, \qquad \Delta_b := \frac{2  \lambda \omega(x,y) }{b} +  \frac{8 \pi^2}{b}   \partial_x^2 \\
\frac{\delta \lambda} {\delta \omega} \cdot \delta \omega &= - \lambda  \langle \psi^2 \omega^{-1}, \delta \omega \rangle_{\omega g_0},
\end{align*}
where $\langle \cdot, \cdot \rangle_{\omega g_0}$ is the  inner product  induced by the metric, 
$$
\langle f,h\rangle_{\omega g_0} = \int_M f h d\mu_g = \int_{[0,2\pi]^2} f h \sqrt{|g|} dx dy = \frac{b}{4\pi^2} \int_{[0,2\pi]^2} f(x,y) h(x,y) \omega(x,y) dx dy.
$$
\end{prpstn}
All computations  for the flat torus using the spectral method are  done on the domain $[0,2\pi]^2$ (with periodic boundary conditions). Eigenvalue derivatives are computed numerally using the formulae in  Proposition \ref{prop:FTderivs}. The operators $\Delta_a$ and $\Delta_b$ are implemented using the Toeplitz matrices, $D$ and $D^{(2)}$, and the Kronecker product as discussed in \S\ref{sec:homer}. 

Finally, we can use Proposition   \ref{prop:FTderivs} and  the relationships 
\begin{align*}
\partial_x  = \frac{1}{2 \pi} \partial_u \quad \text{and} \quad 
\partial_y  =  \frac{a}{2 \pi} \partial_u + \frac{b}{2\pi} \partial_v, 
\end{align*}
to push these derivatives forward from the square to the flat torus (see Figure  \ref{fig:coordDefs}).  
We obtain the following result, which is used in the finite element computations on flat tori. 
\begin{prpstn} \label{prop:FTderivs2}
Let $\lambda(a,b,\omega)$ be a simple eigenvalue of an $(a,b)$-flat torus with conformal factor $\omega$ and corresponding eigenfunction $\psi$ normalized such that $\langle \psi ,  \psi  \rangle_{\omega g_0} =1$. Then,
\begin{align*}
\frac{\partial \lambda} {\partial  a} &= -   \langle \psi, \omega^{-1} \Delta_a \psi \rangle_{\omega g_0} , \qquad \Delta_a := -  \frac{2  }{b}  \partial_u \partial_v  \\
\frac{\partial \lambda} {\partial  b} &= -  \langle \psi, \omega^{-1} \Delta_b \psi \rangle_{\omega g_0}, \qquad \Delta_b := \frac{2  \lambda \omega}{b} + \frac{2 }{b} \partial_u^2 \\
\frac{\delta \lambda} {\delta \omega} \cdot \delta \omega &= - \lambda  \langle \psi^2 \omega^{-1}, \delta \omega \rangle_{\omega g_0}, 
\end{align*}
where  $\langle f,h\rangle_{\omega g_0} =  \int_{\mathbb T^2} f h \omega d\mu_{g_0} $ is the  inner product on the flat torus. 
\end{prpstn}

\section{Computations of conformal and topological spectra} \label{sec:compRes}
In this section, we compute the conformal spectrum for several manifolds and  topological spectrum for genus $\gamma=0,1$.  Numerical values of volume-normalized eigenvalues, $\Lambda_k$, are given in Table  \ref{tab:ConfofmalSpectrum} for comparison. 

\subsection{The topological spectrum of genus zero Riemannian surfaces } \label{sec:genus0}

\begin{figure}[htbp]
\begin{center}
\includegraphics[width=8cm]{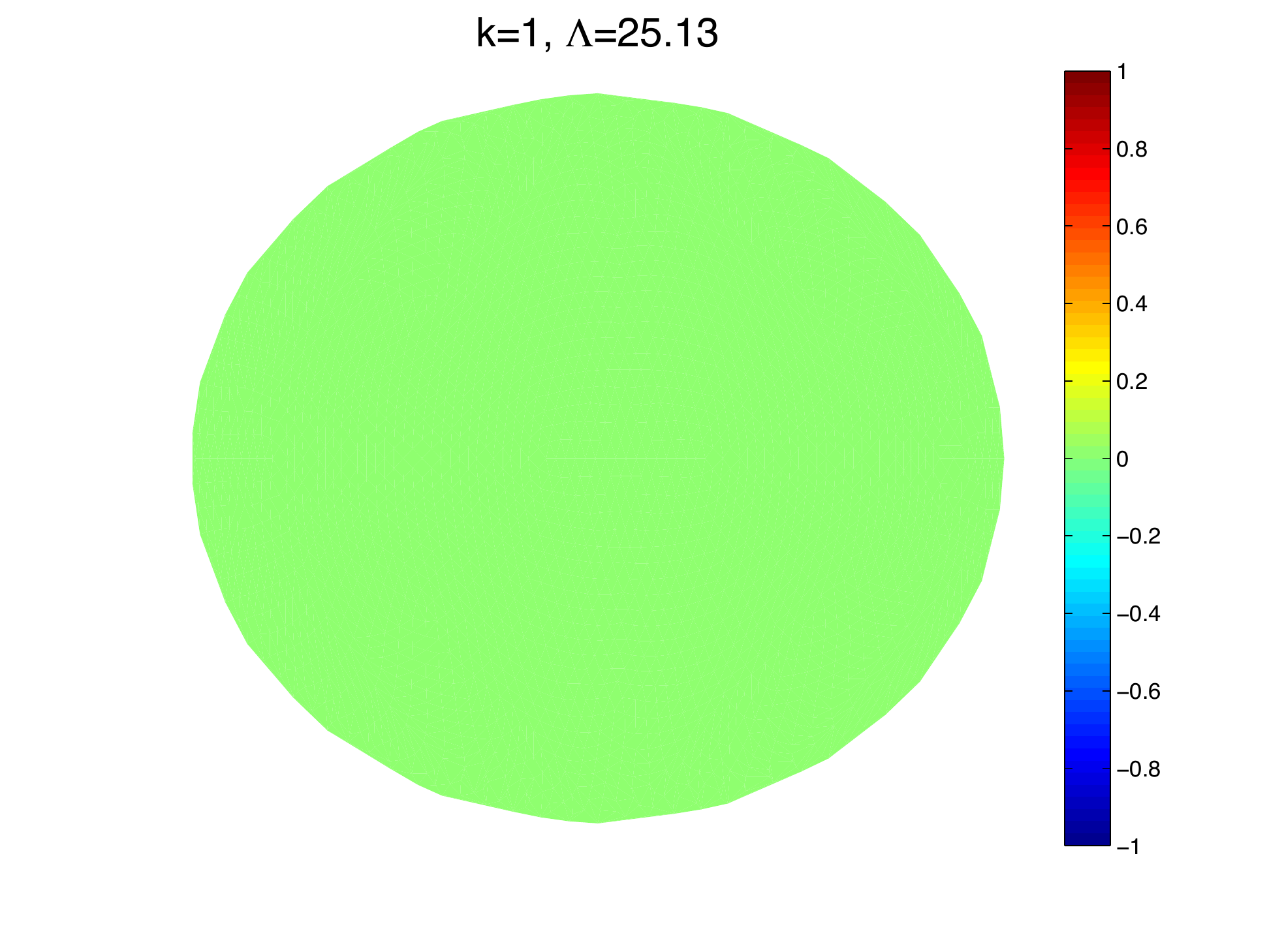} 
\includegraphics[width=8cm]{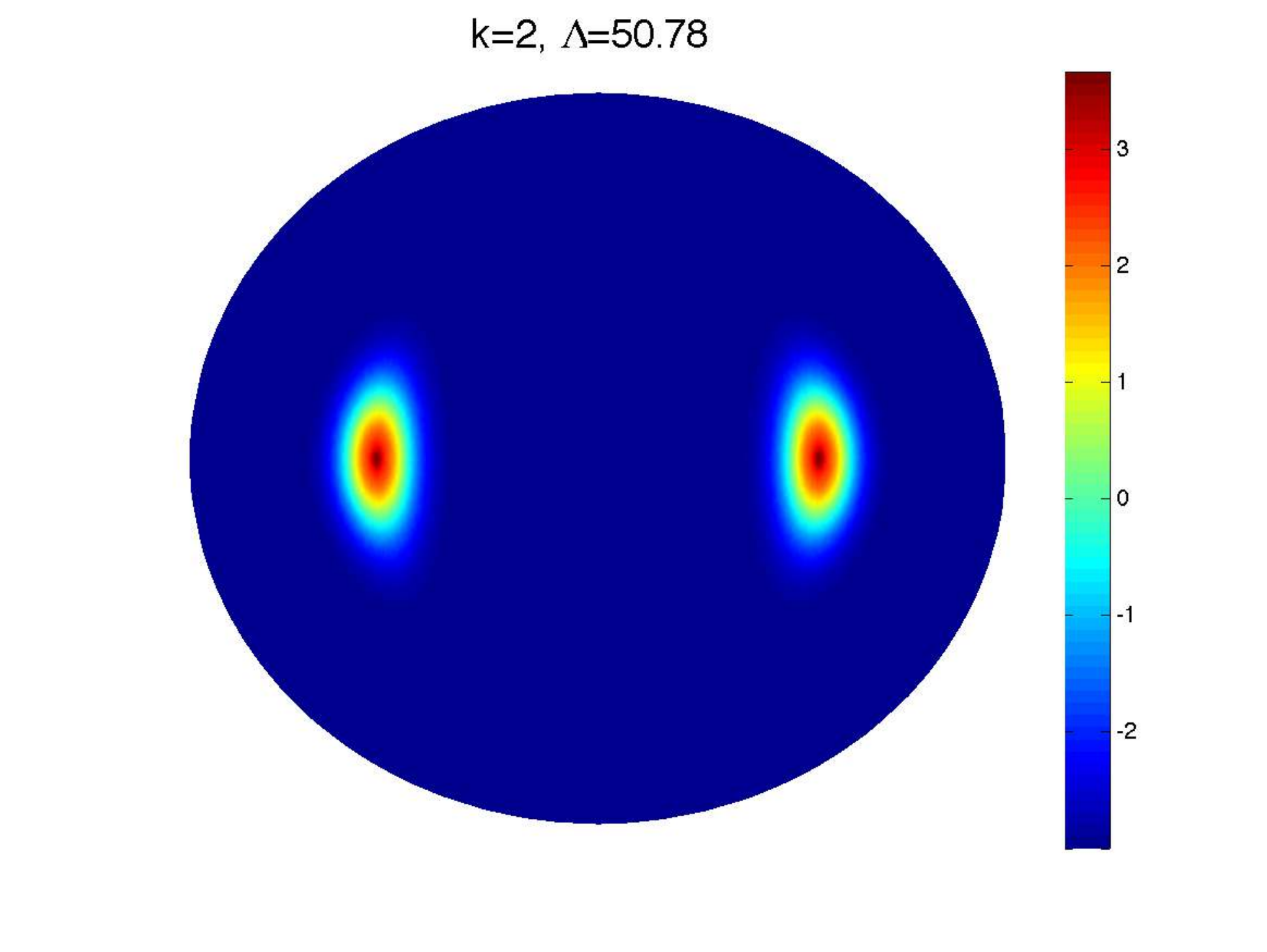} 
\includegraphics[width=8cm]{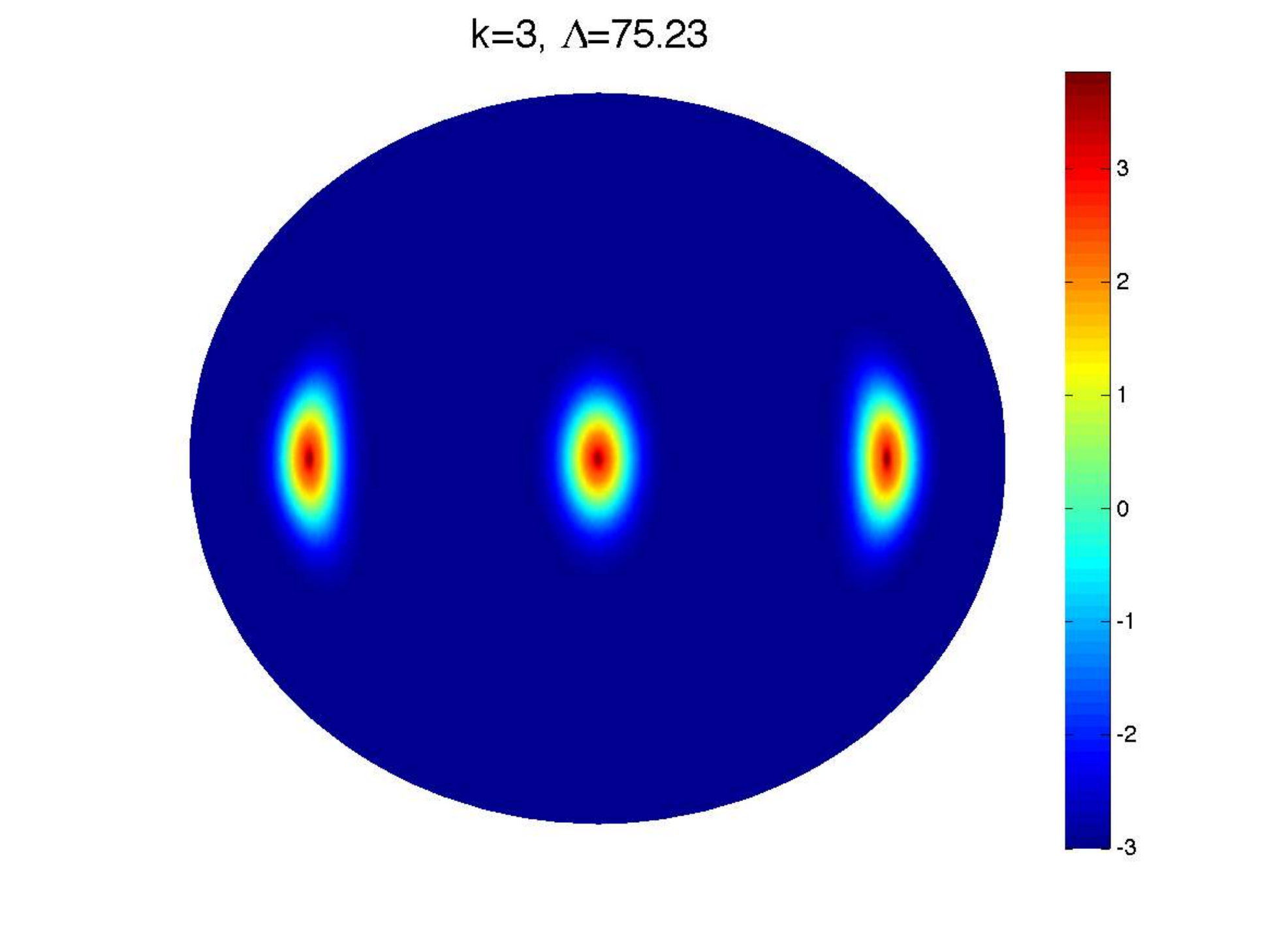} 
\includegraphics[width=8cm]{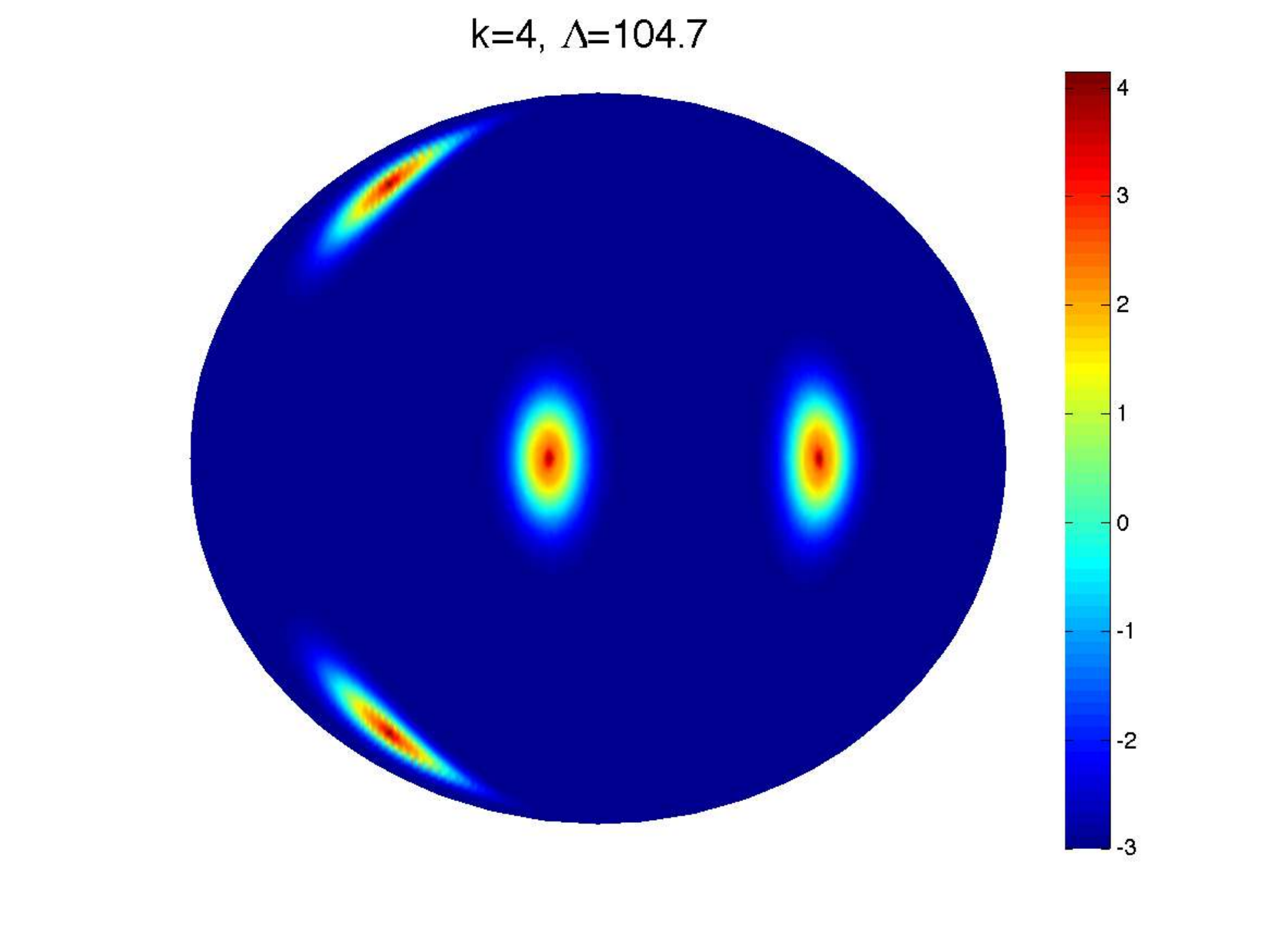} 
\includegraphics[width=8cm]{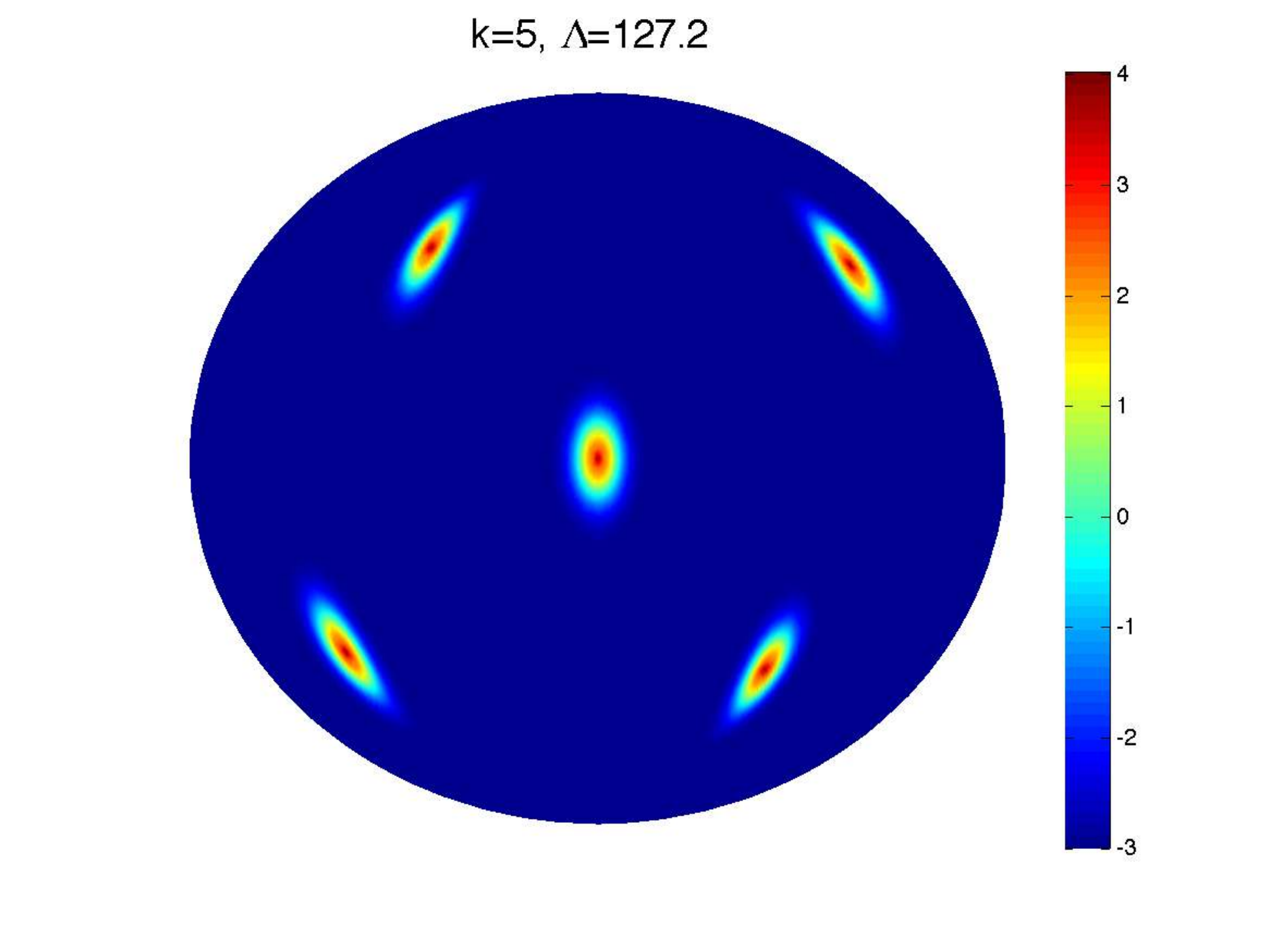} 
\includegraphics[width=8cm]{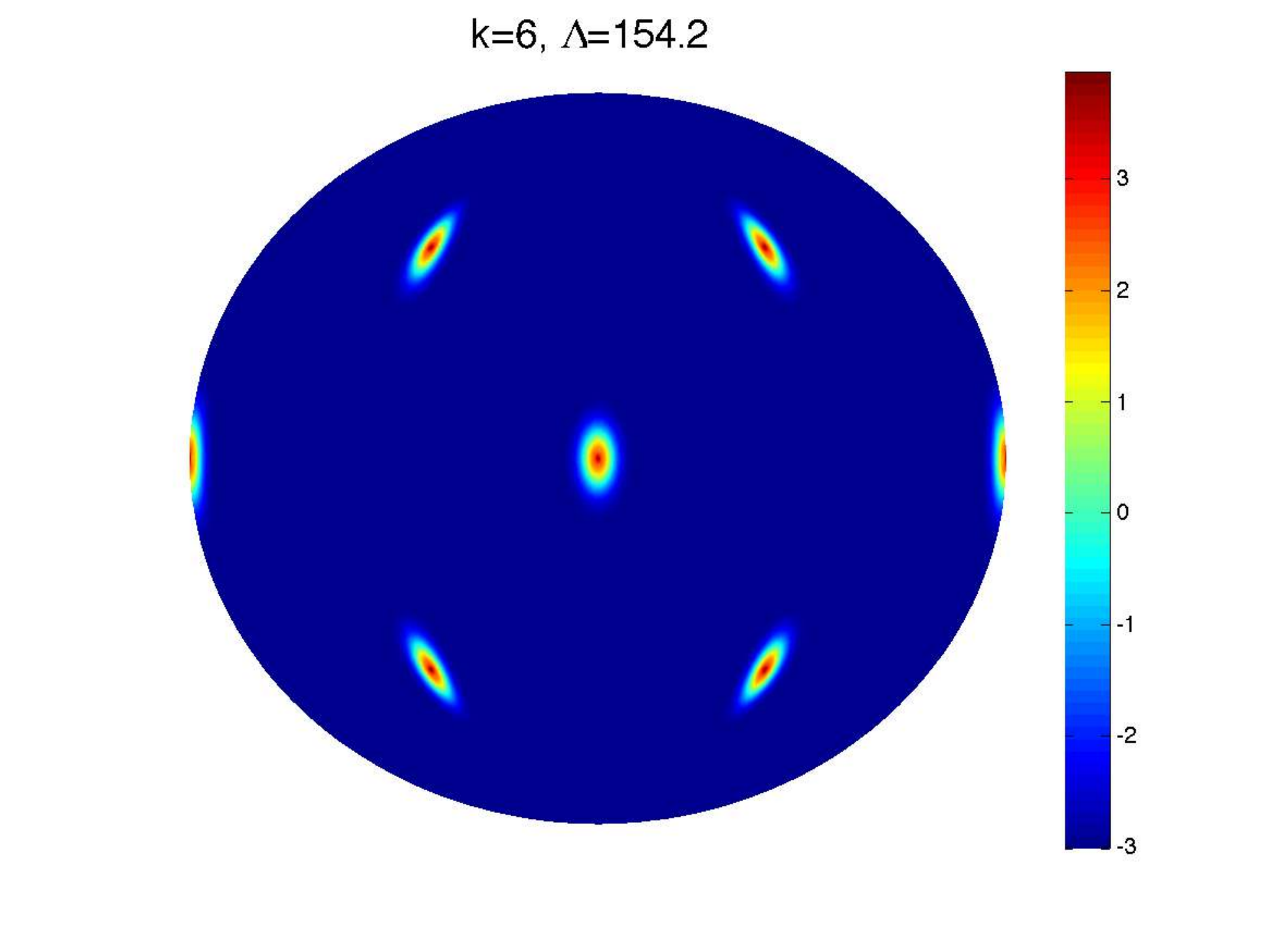} 
\caption{A Hammer projection of the best conformal factors found for $\Lambda_k$, $k=1,\ldots,6$ on the sphere. See \S\ref{sec:genus0}. }
\label{fig:optSpheres}
\end{center}
\end{figure}

\begin{figure}[h!]
\begin{center}
\begin{minipage}{0.49\linewidth}
\includegraphics[width=1\linewidth]{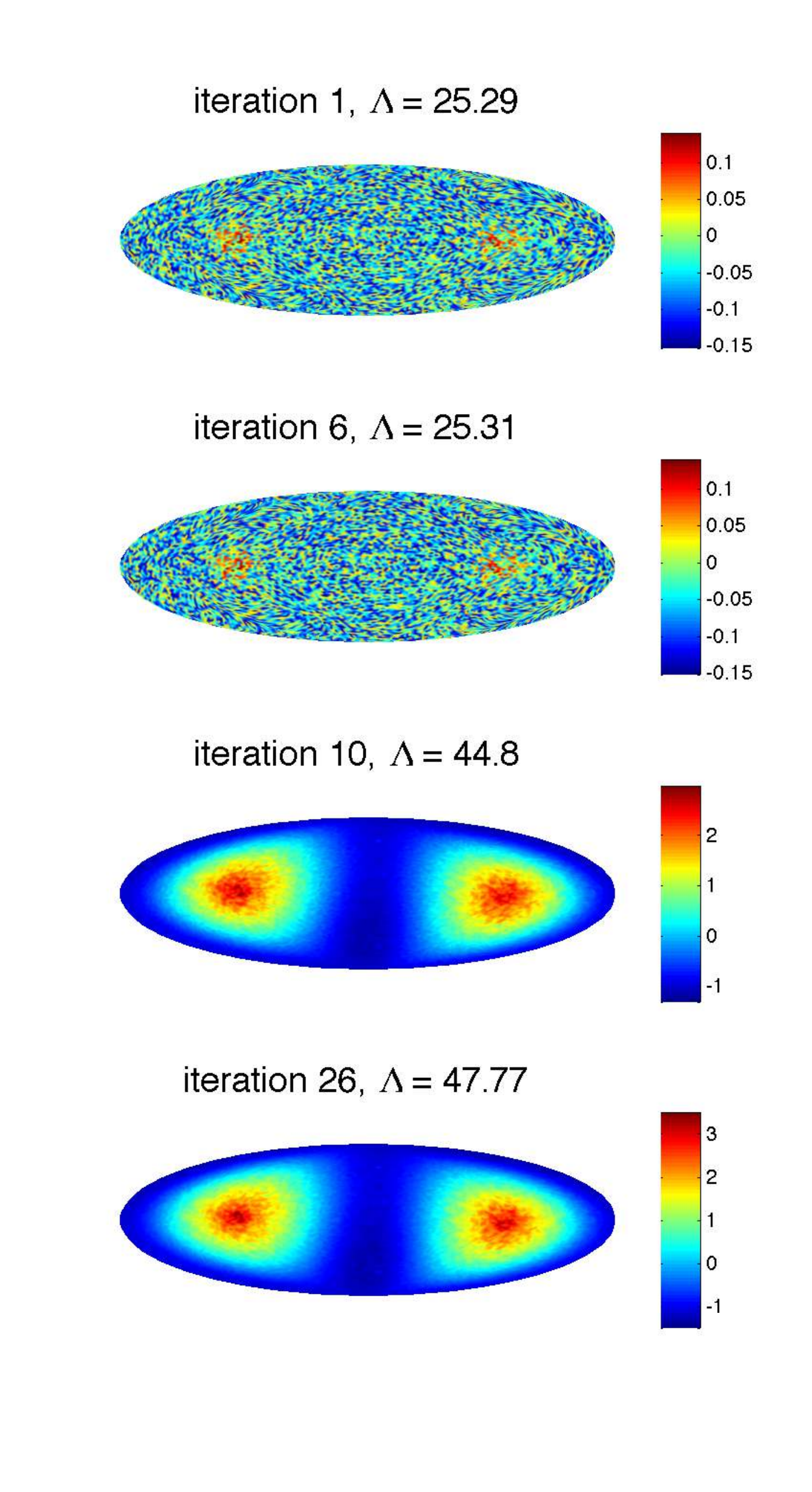}
\end{minipage}\hfill
\begin{minipage}{0.49\linewidth}
\vspace{-0.7cm}
\includegraphics[width=1\linewidth]{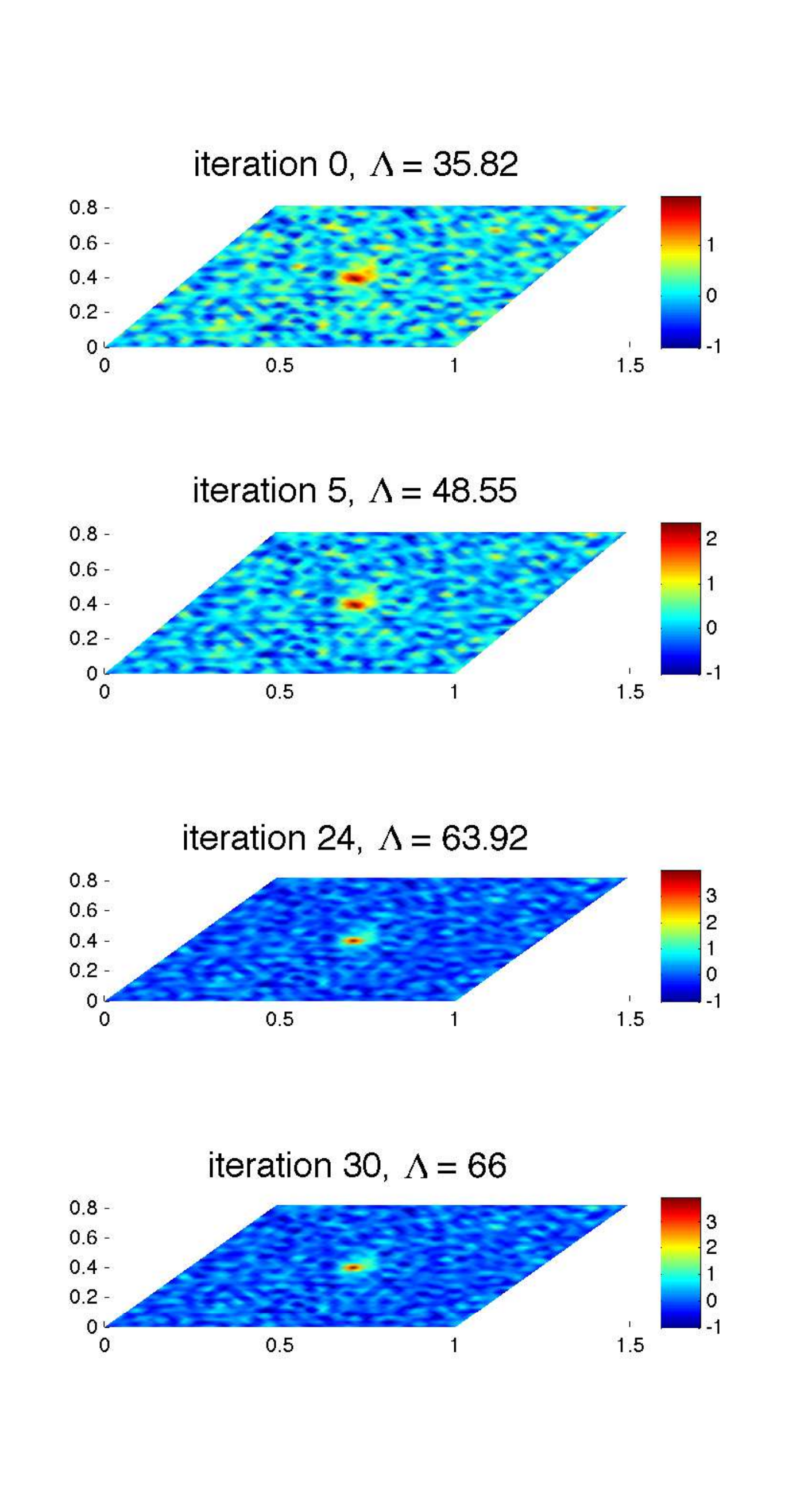}
\end{minipage}\\
\caption{Iterates of the proposed computational method. 
{\bf (left) } A sequence of conformal factors on the sphere to maximize $\Lambda_2$.  See \S\ref{sec:genus0}.  
{\bf (right) } A sequence of tori  to maximize  $\Lambda_1$. See \S\ref{sec:TopSpecOne}. }
\label{fig:sphereSequence}
\end{center}
\end{figure}

By the Uniformization Theorem, any closed genus-0 Riemannian surface $(M,g)$ is conformal to $\mathbb \mathbb S^2$ with the canonical metric of constant sectional curvature, $g_0$ \cite{Imayoshi1992}. In other words, the moduli space of closed Riemannian surfaces consists of one point and the conformal spectrum for any genus $\gamma=0$ Riemannian surface is identical. In particular, for any genus zero  Riemannian surface, $(M,g)$, 
$$
\Lambda^c_k(M,[g]) = \Lambda^c_k(\mathbb S^2, [g_0]) = \Lambda^t_k(0).
$$

In this section, we approximate $\Lambda^c_k(\mathbb S^2, g_0)$ using the computational methods described in \S\ref{sec:compMeth}. 
To compute the Laplace-Beltrami eigenvalues, we use the finite element method on a mesh of the sphere with $40,962$ vertices. The optimization problem is solved using a quasi-Newton  method, where the gradient of the eigenvalues is computed via Proposition  \ref{prop:S2derivs}. 

The best conformal factors found for $k=1, 2, \cdots,6$ are presented in Figure  \ref{fig:optSpheres} and the corresponding numerical values given in  Table \ref{tab:ConfofmalSpectrum}. For this computational experiment, we have chosen many different initializations for the conformal factors. The initial conditions used for Figure  \ref{fig:optSpheres} were the sum of localized Gaussians located at  points  equidistributed on the sphere.  To further illustrate our computational method, we consider a randomly initialized conformal factor.  In Figure  \ref{fig:sphereSequence}(left), we plot for $k=2$, the 1st, 6th, 10th, and 26th iterates of the conformal factor. The mesh of the sphere used here has 10,242 vertices. The optimization code is only able to achieve a value of $\Lambda=47.77$ for this grid size and initial condition, however the general pattern of the conformal factor having two localized maxima is clearly observed. 

Hersch's result that the standard metric on $\mathbb S^2$ is the only metric up to isometry attaining $\Lambda^t_1(0)$ is supported in the computational results \cite{hersch1970}.  This numerical result gives just one representative from the isometric class; see Remark \ref{rem:isom}, where a  conformal factor on the sphere, isometric to the uniform conformal factor, is constructed that gives the same first topological eigenvalue. 
For $k=2$, it was shown in \cite{Nadirashvili2002}  that the maximum is approached by a sequence of surfaces degenerating to a union of two identical round spheres, a configuration we refer to as two kissing spheres,  with second eigenvalue $\Lambda^t_2(1)= 16\pi \approx 50.26$. The value $\Lambda_2^\star = 50.78$, obtained numerically is slightly larger. As discussed in \S\ref{sec:homer}, the finite element method used overestimates the Laplace-Beltrami eigenvalues and hence the value of the maximum. After having solved this problem on a sequence of increasingly fine meshes, we believe that this small discrepancy is the result of numerical error. The conformal factor on $\mathbb S^2$ corresponding to ``two kissing spheres'' is  the one shown in the top right panel of  Figure  \ref{fig:optSpheres}. 

From  Figure  \ref{fig:optSpheres}, we further observe that the $k$-th eigenvalue is large precisely when the metric has $k$ localized regions with large value. This corresponds to the ``$k$-kissing spheres'' surface as described in \S\ref{sec:KissingSpheres}. Observe that for increasingly large $k$, the regions where the metric is localized is increasingly small. Since it is possible for the eigenfunctions to become very concentrated at these  regions of concentrated measure, we reason that for larger values of $k$, to improve accuracy we should use a finer mesh at these regions or, equivalently, deform the surface at these points to locally enlarge the volume. We choose the later option, and consider a mesh consisting of $k$ spheres ``glued'' together which approximates $k$ kissing spheres. For example, to construct the mesh for $k=2$, we remove one element (triangle) from the mesh representing each sphere and then identify the edges associated with the missing faces of the two punctured  balls. On those glued meshes, we again maximize $\Lambda_1$ as a function of the conformal factor, $\omega$. The best conformal factor found for $k=1,\ldots,6$ is plotted in Figure  \ref{fig:KissingSpheres}. In each case, the conformal factor is very flat and the optimal values obtained are very close to $ 8 \pi k$.

To further test these optimal conformal factors, we consider  configurations of spheres with different sizes; see Figure \ref{fig:KissingSpheres2}. 
For $\Lambda_1$, we consider a mesh approximating a sphere with radius $1/2$. 
For $\Lambda_2$ to $\Lambda_6$, we consider meshes approximating glued spheres. The larger spheres have radius 1 and the smaller spheres have radius $1/2$. 
In each case,  we verified that the constructed surface has genus $\gamma = 0$ using the Euler characteristic of the mesh. 
In each case, the conformal factor is very flat on each sphere and the optimal values obtained are very close to $ 8 \pi k$.

\begin{figure}[t]
\begin{center}
\begin{minipage}{0.33\linewidth}
\centering $\Lambda_1 = 25.13$
\end{minipage}\hfill
\begin{minipage}{0.33\linewidth}
\centering $\Lambda_2 = 50.28$
\end{minipage}\hfill
\begin{minipage}{0.33\linewidth}
\centering $\Lambda_3 = 75.39$
\end{minipage}\hfill\\
\begin{minipage}{0.33\linewidth}
\includegraphics[width=.9\linewidth]{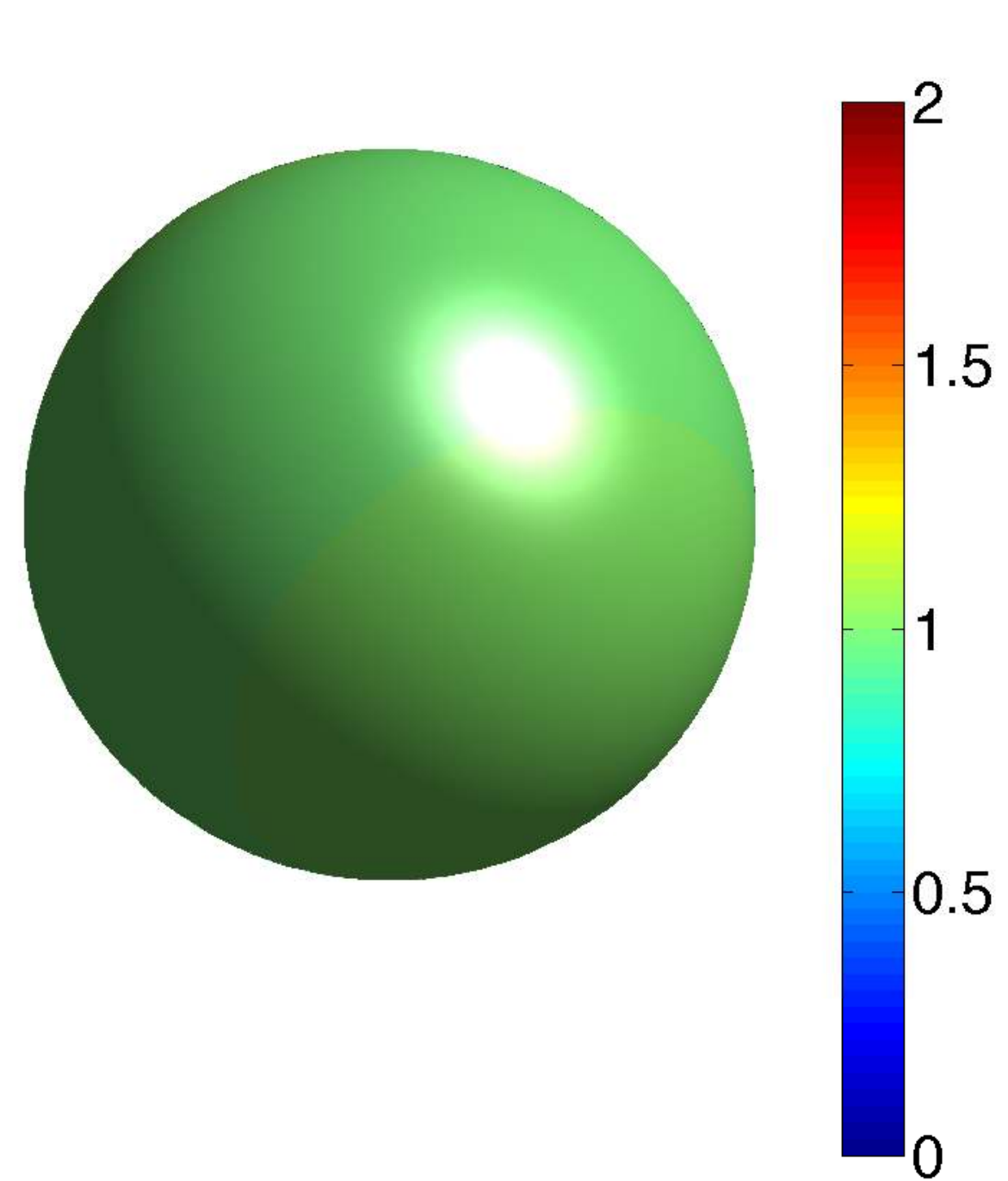}
\end{minipage}\hfill
\begin{minipage}{0.33\linewidth}
\includegraphics[width=1\linewidth]{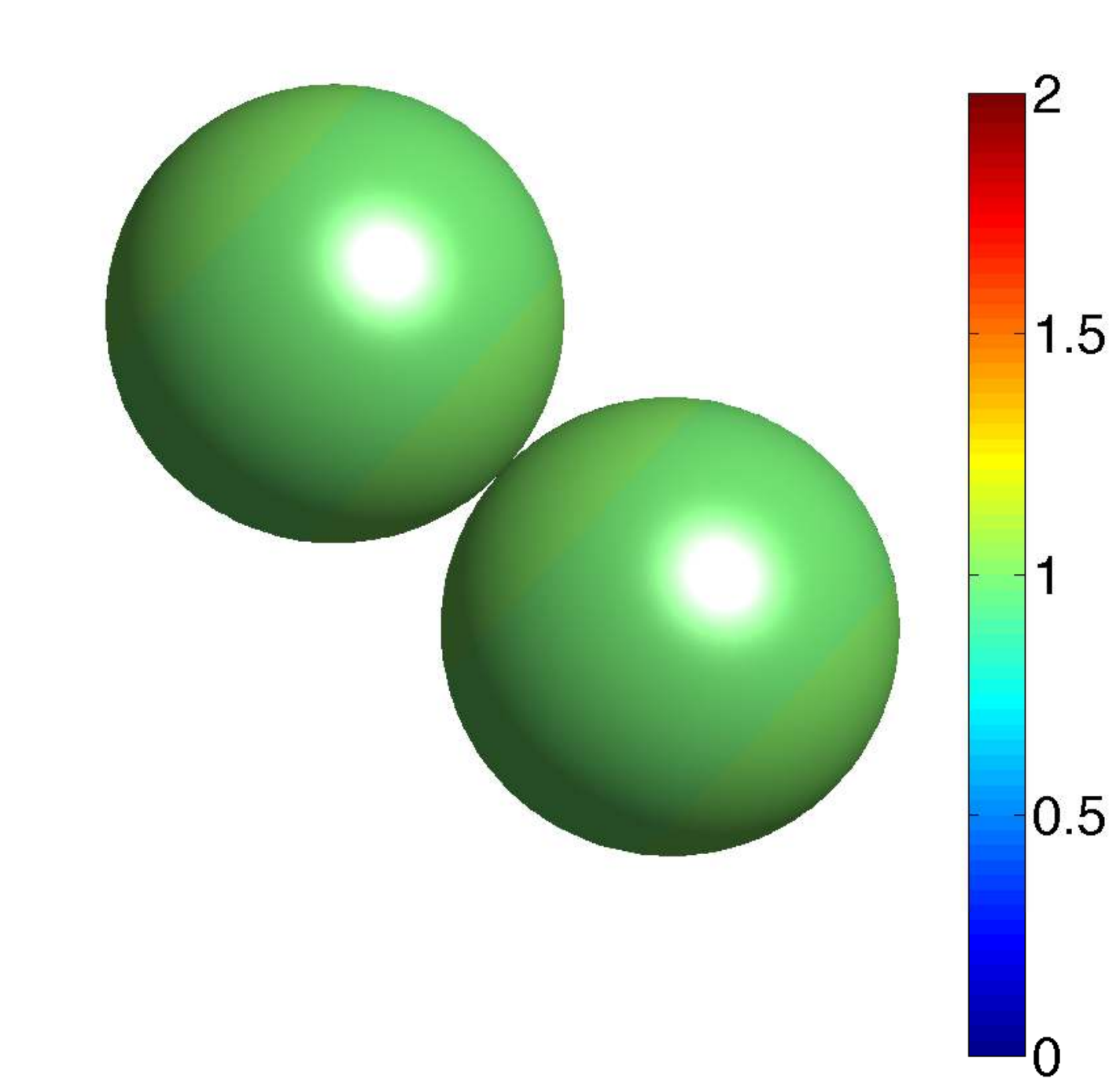}
\end{minipage}\hfill
\begin{minipage}{0.33\linewidth}
\includegraphics[width=1\linewidth]{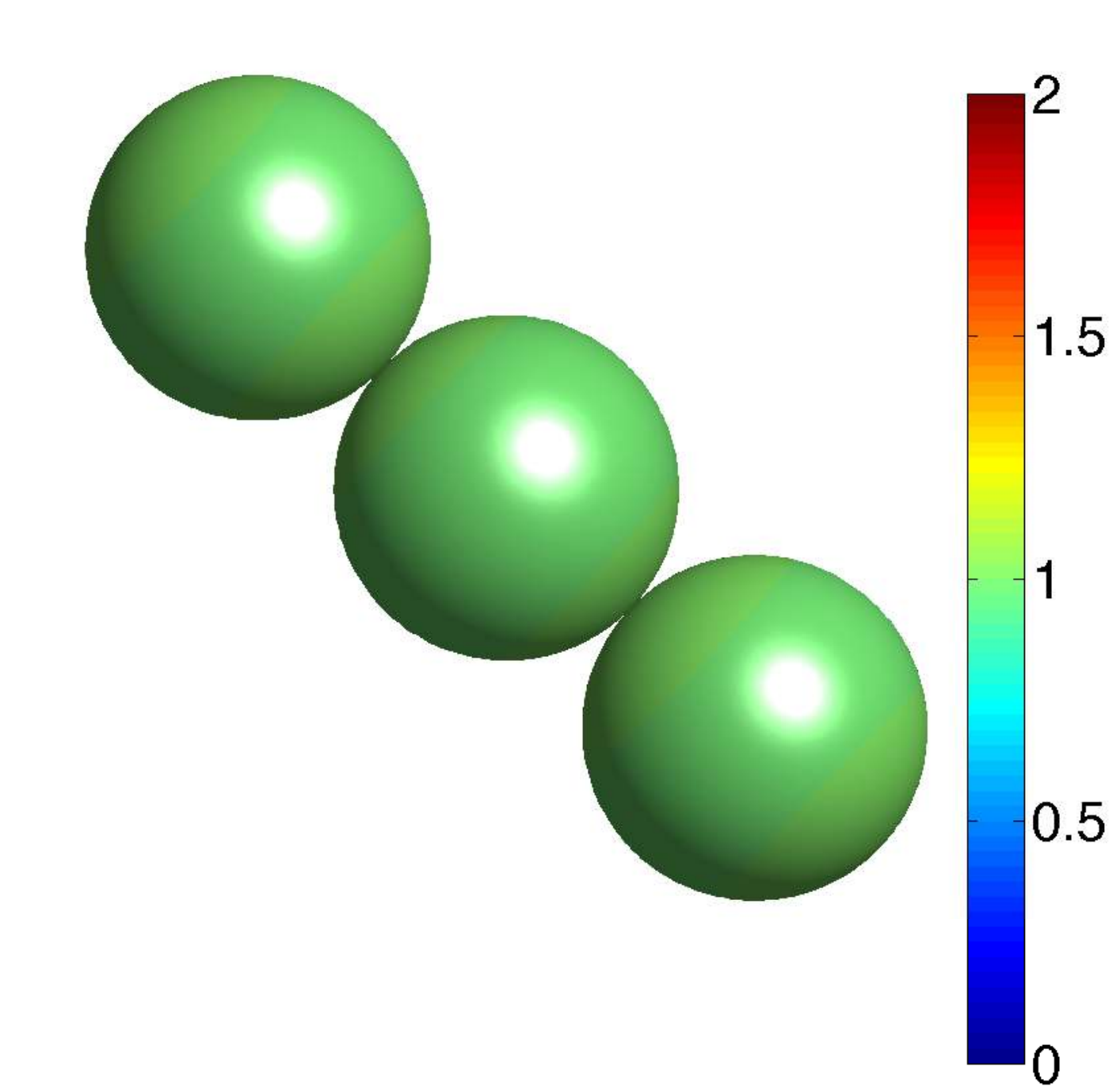}
\end{minipage}\hfill\\
\vspace{0.5cm}
\begin{minipage}{0.33\linewidth}
\centering $\Lambda_4 = 100.52$
\end{minipage}\hfill
\begin{minipage}{0.33\linewidth}
\centering $\Lambda_5 = 125.66$
\end{minipage}\hfill
\begin{minipage}{0.33\linewidth}
\centering $\Lambda_6 = 150.76$
\end{minipage}\hfill\\
\begin{minipage}{0.33\linewidth}
\includegraphics[width=1\linewidth]{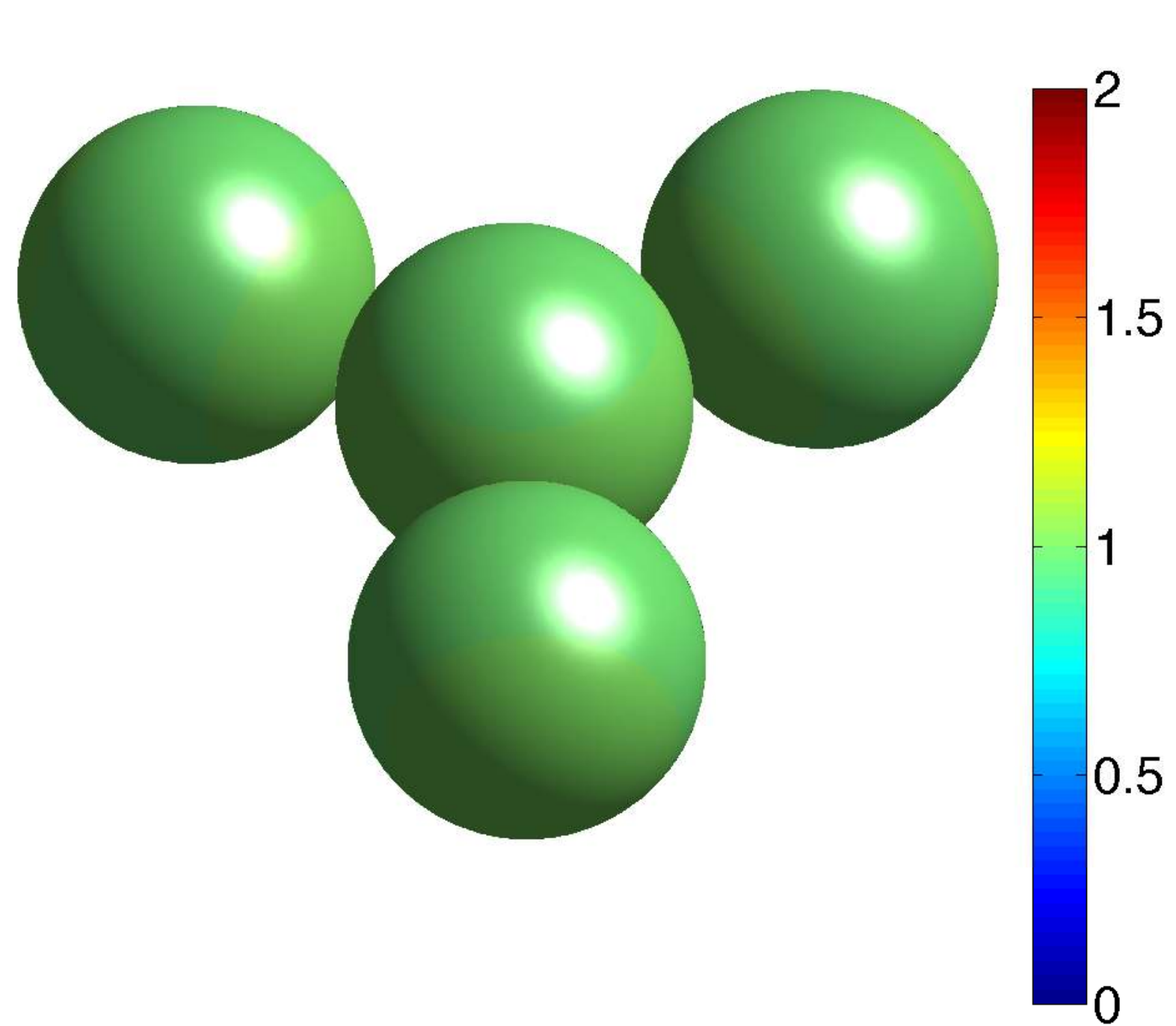}
\end{minipage}\hfill
\begin{minipage}{0.33\linewidth}
\includegraphics[width=1\linewidth]{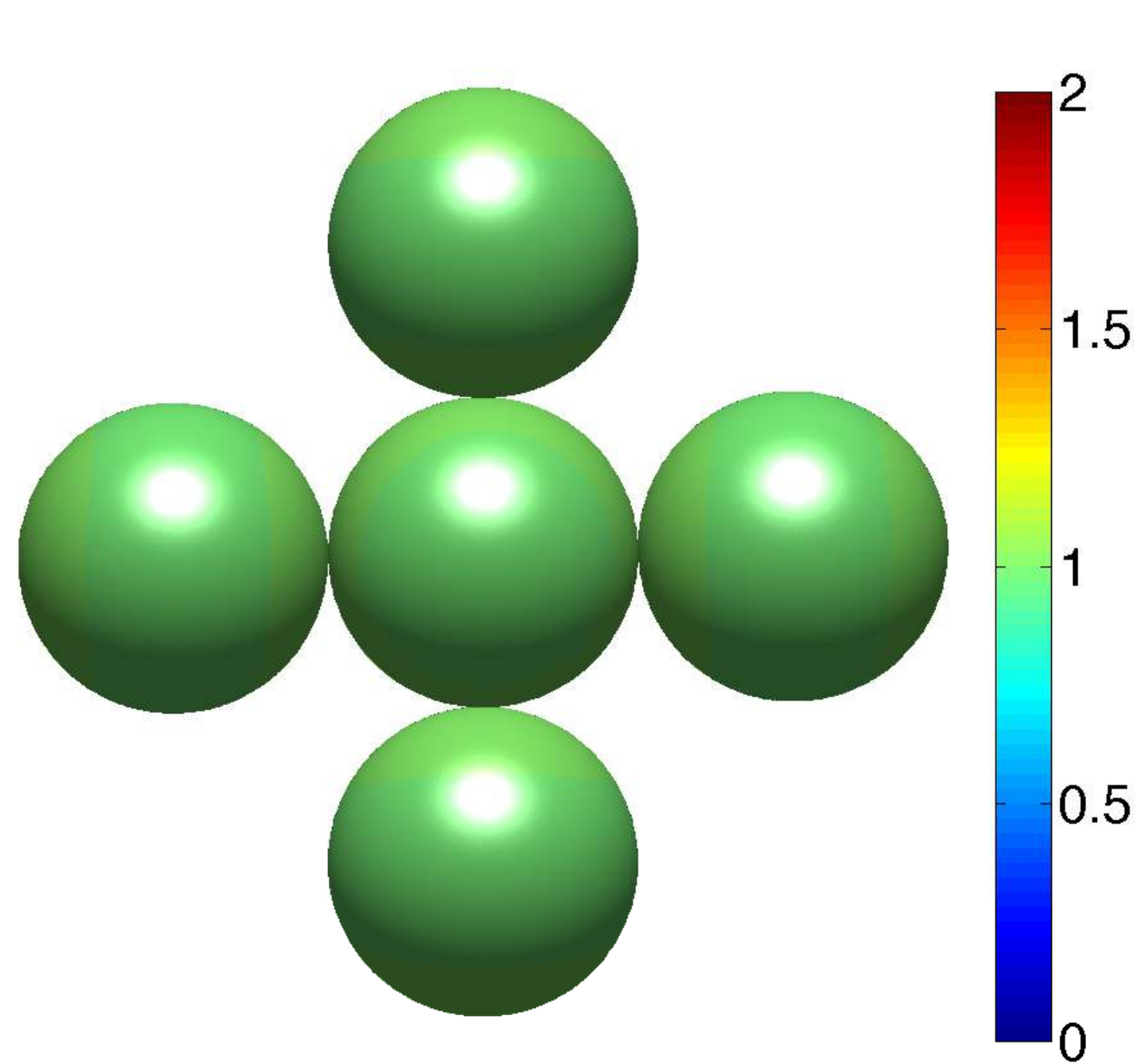}
\end{minipage}\hfill
\begin{minipage}{0.33\linewidth}
\includegraphics[width=1\linewidth]{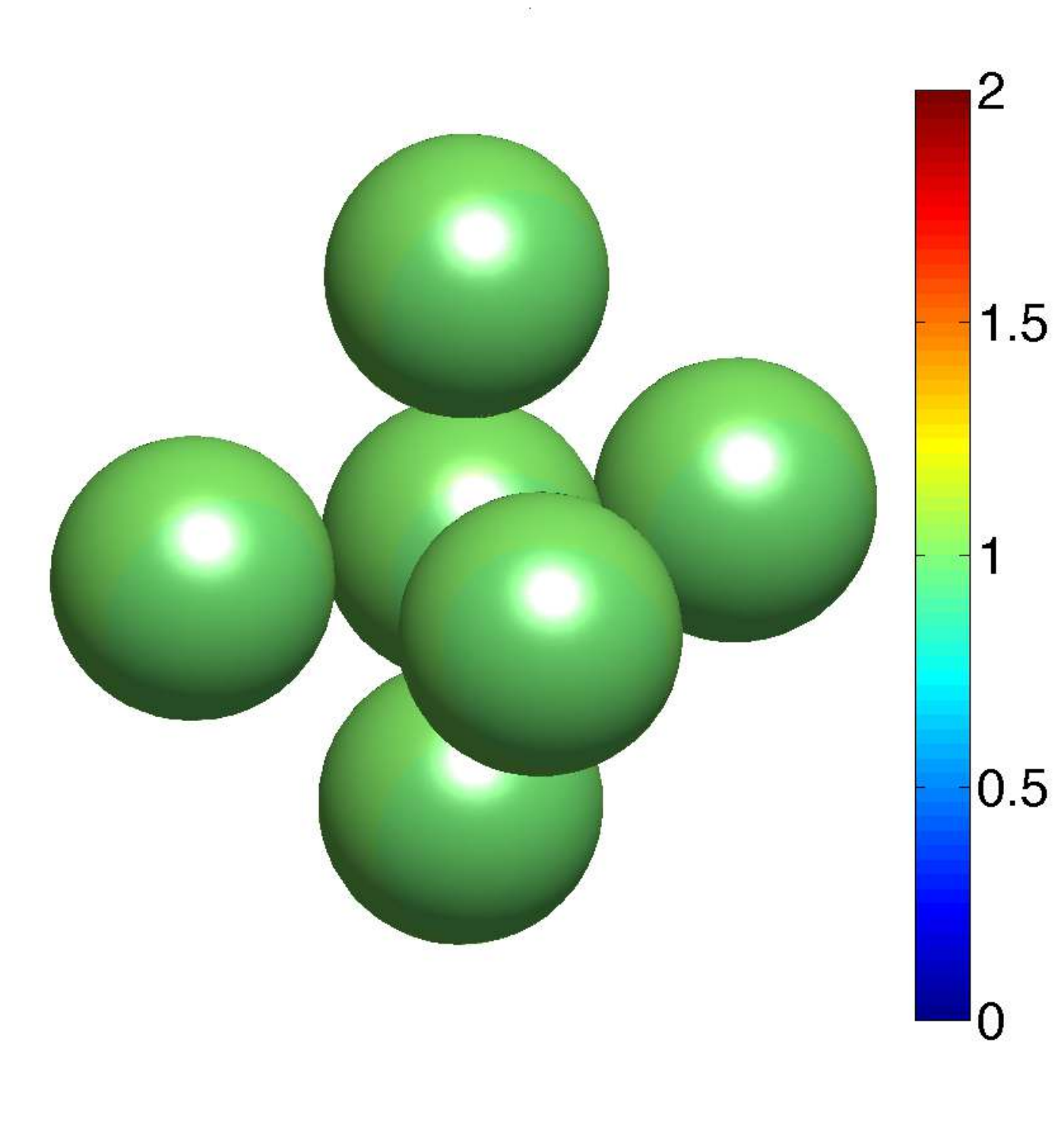}
\end{minipage}\hfill
\caption{The best conformal factors found for $\Lambda_k$, $k=1,\ldots 6$ on genus zero meshes representing $k$-kissing spheres. See \S\ref{sec:genus0}. }
\label{fig:KissingSpheres}
\end{center}
\end{figure}

\begin{figure}[t]
\begin{center}
\begin{minipage}{0.33\linewidth}
\centering $\Lambda_1 = 25.13$
\end{minipage}\hfill
\begin{minipage}{0.33\linewidth}
\centering $\Lambda_2 = 50.26$
\end{minipage}\hfill
\begin{minipage}{0.33\linewidth}
\centering $\Lambda_3 = 75.39$
\end{minipage}\hfill\\
\begin{minipage}{0.33\linewidth}
\includegraphics[width=.9\linewidth]{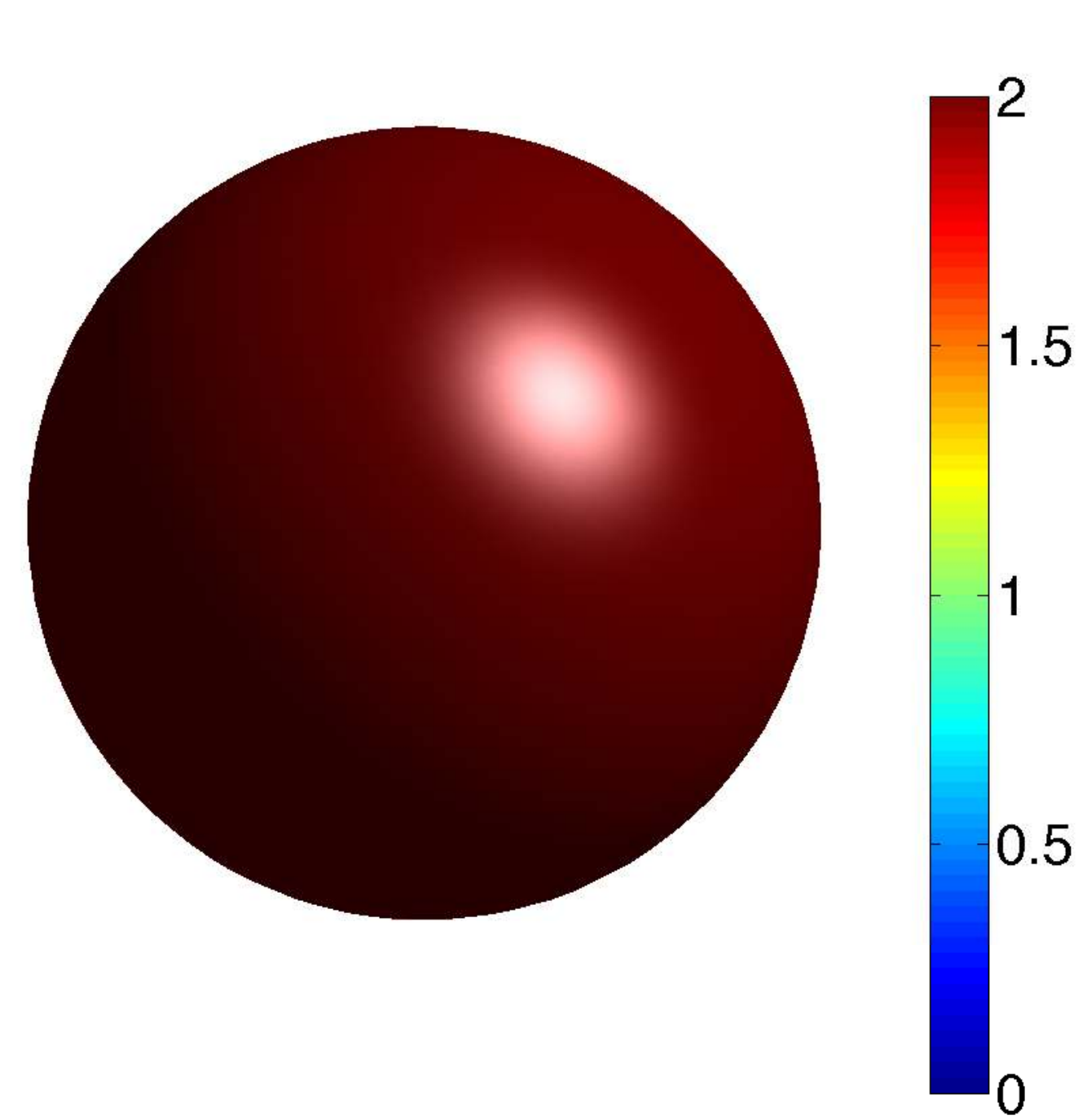}
\end{minipage}\hfill
\begin{minipage}{0.33\linewidth}
\includegraphics[width=1\linewidth]{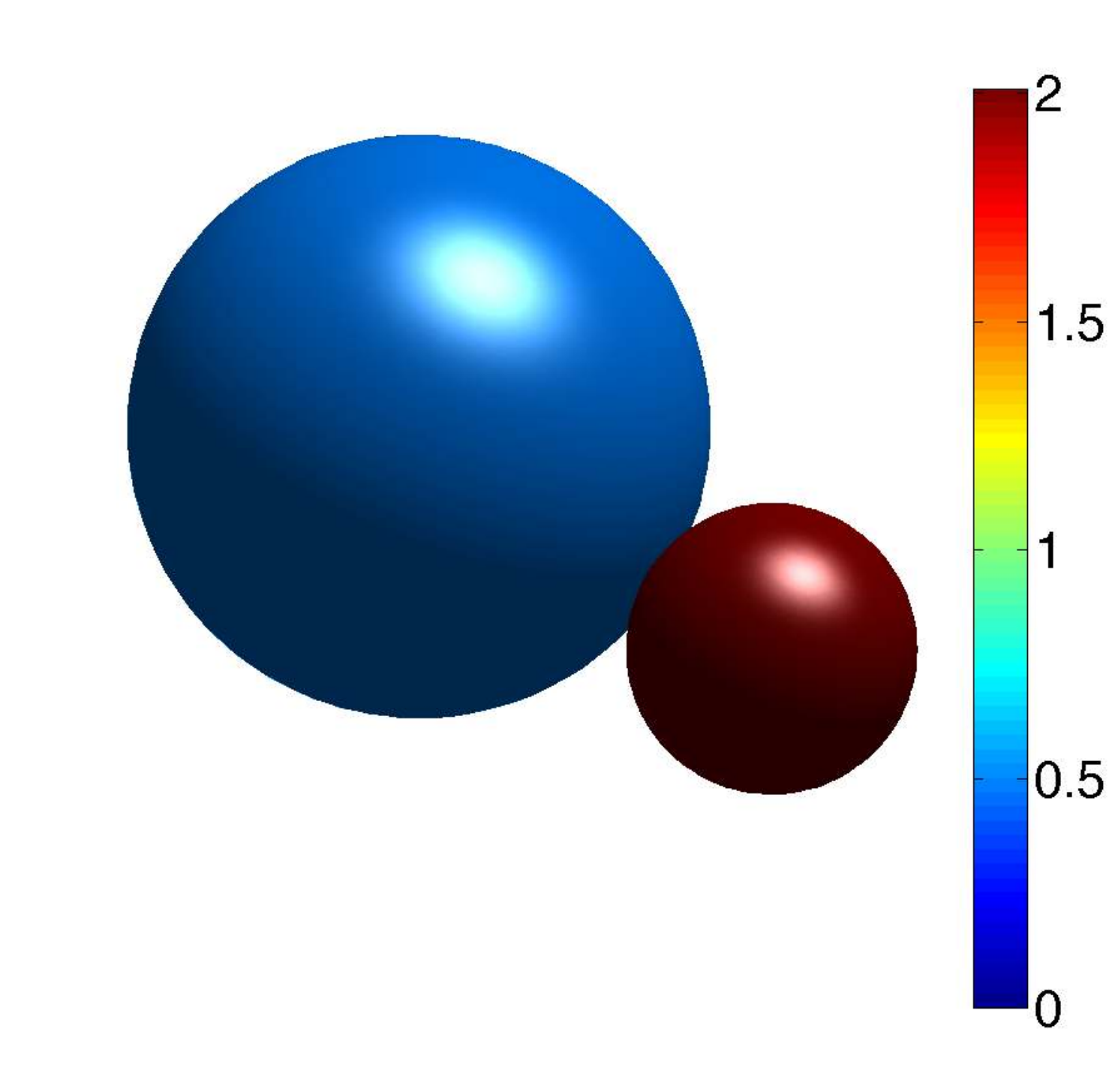}
\end{minipage}\hfill
\begin{minipage}{0.33\linewidth}
\includegraphics[width=1\linewidth]{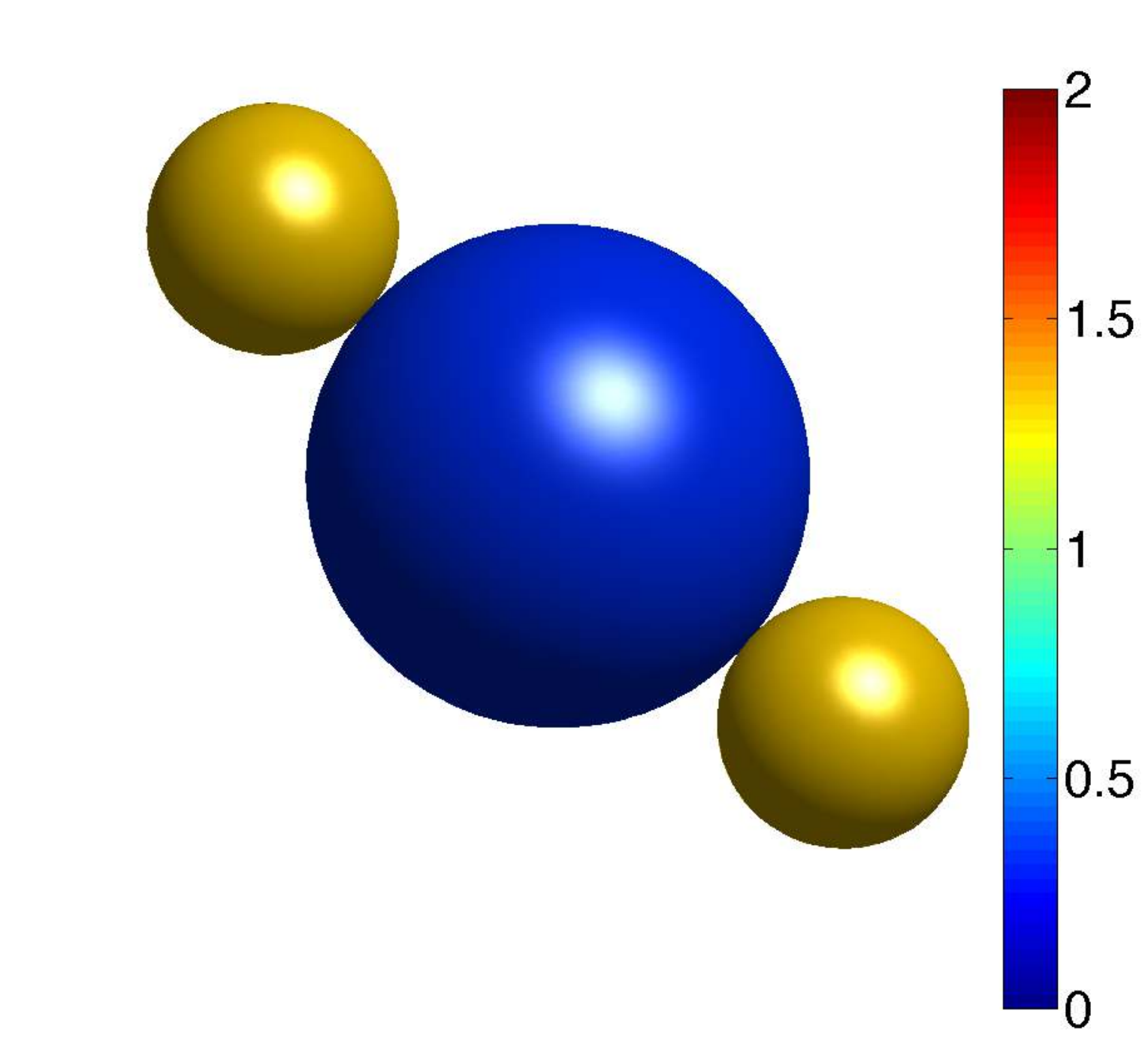}
\end{minipage}\hfill\\
\vspace{0.5cm}
\begin{minipage}{0.33\linewidth}
\centering $\Lambda_4 = 100.50$
\end{minipage}\hfill
\begin{minipage}{0.33\linewidth}
\centering $\Lambda_5 = 125.63$
\end{minipage}\hfill
\begin{minipage}{0.33\linewidth}
\centering $\Lambda_6 = 150.75$
\end{minipage}\hfill\\
\begin{minipage}{0.33\linewidth}
\includegraphics[width=1\linewidth]{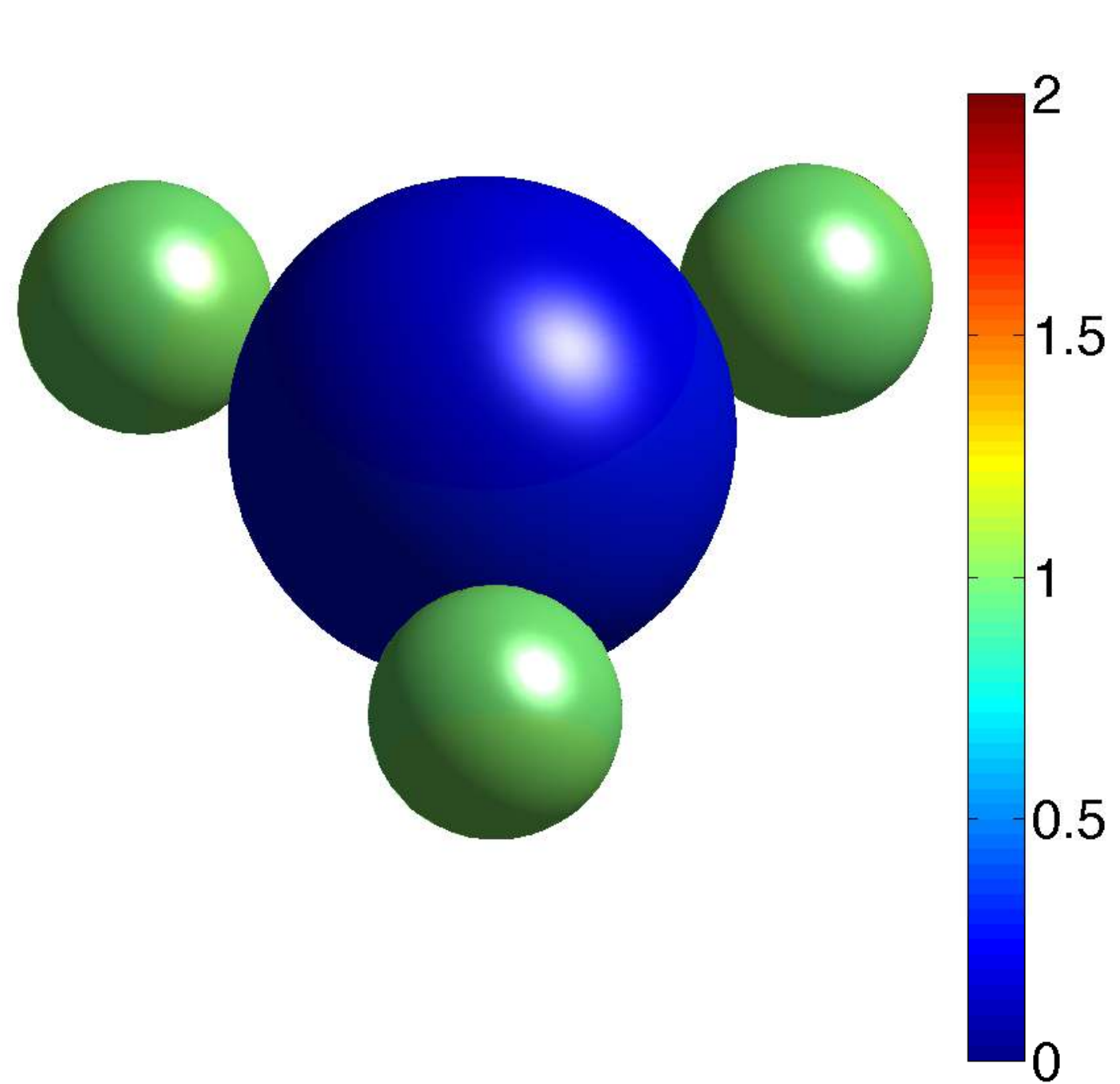}
\end{minipage}\hfill
\begin{minipage}{0.33\linewidth}
\includegraphics[width=1\linewidth]{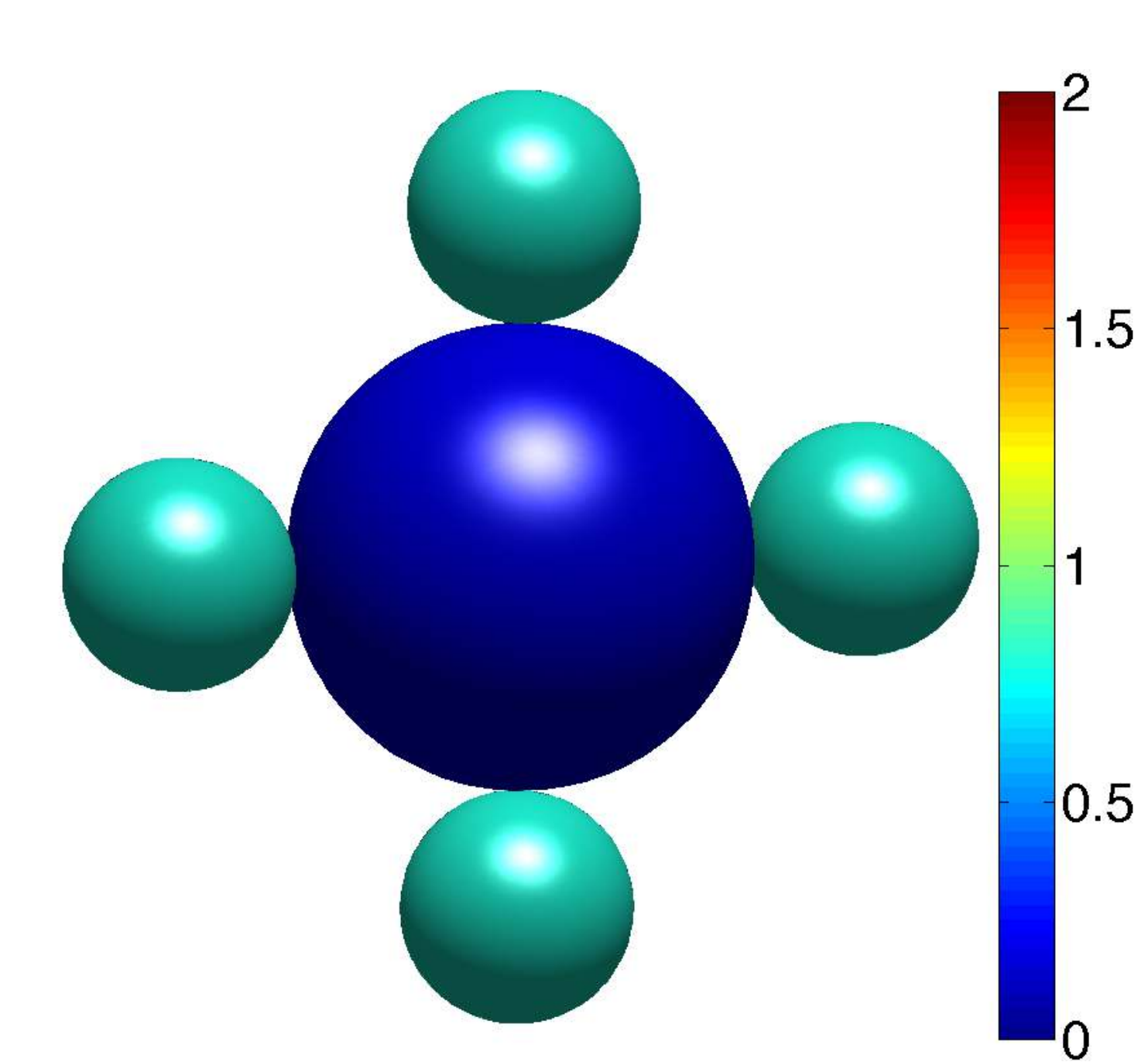}
\end{minipage}\hfill
\begin{minipage}{0.33\linewidth}
\includegraphics[width=1\linewidth]{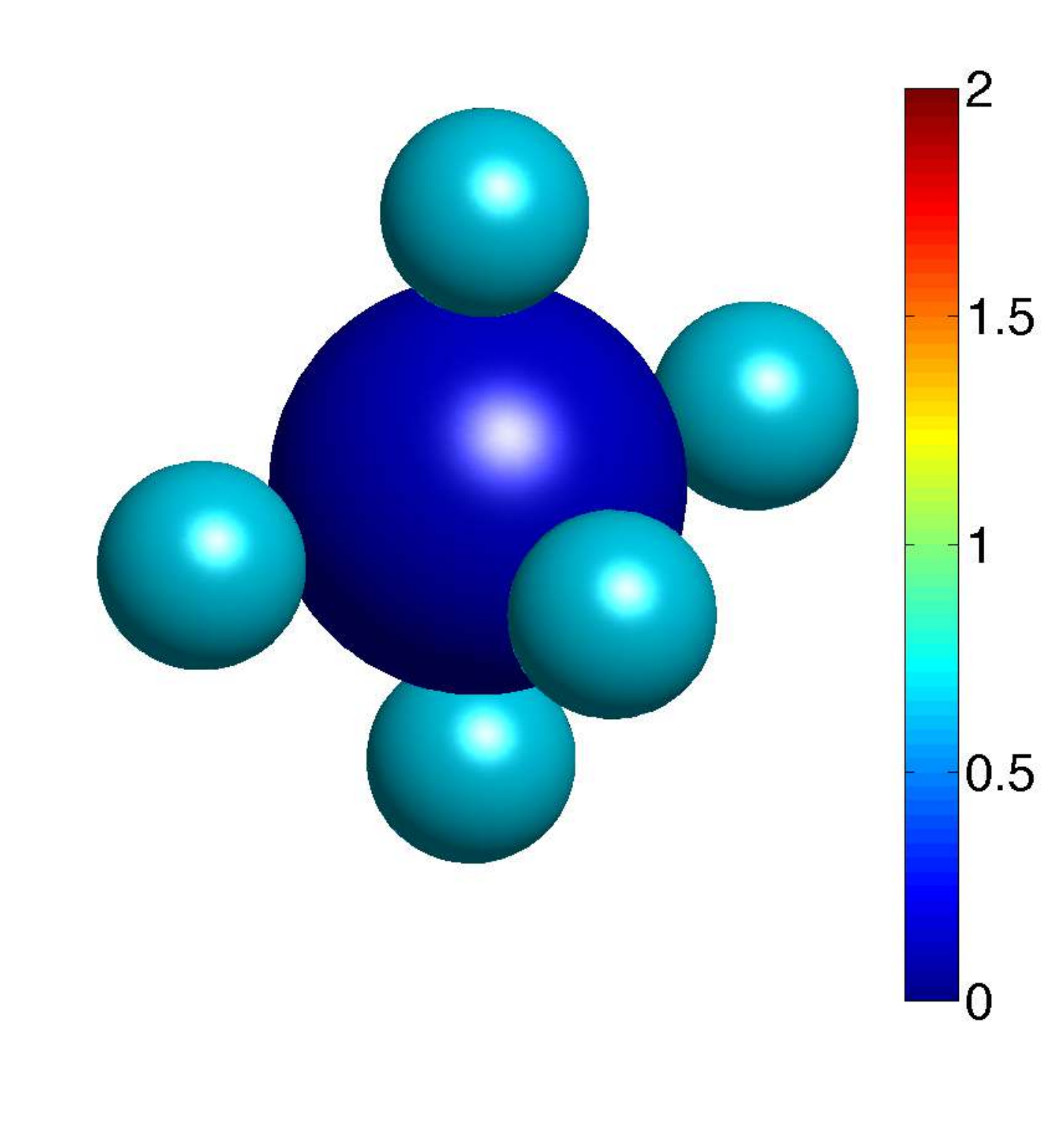}
\end{minipage}\hfill
\caption{The best conformal factors found for $\Lambda_k$, $k=1,\ldots 6$  on genus zero meshes representing a unit sphere kissing with $k-1$ spheres with radius $1/2$. See \S\ref{sec:genus0}. }
\label{fig:KissingSpheres2}
\end{center}
\end{figure}

\begin{figure}[t]
\begin{center}
\begin{minipage}{0.49\linewidth}
\includegraphics[width=.55\linewidth]{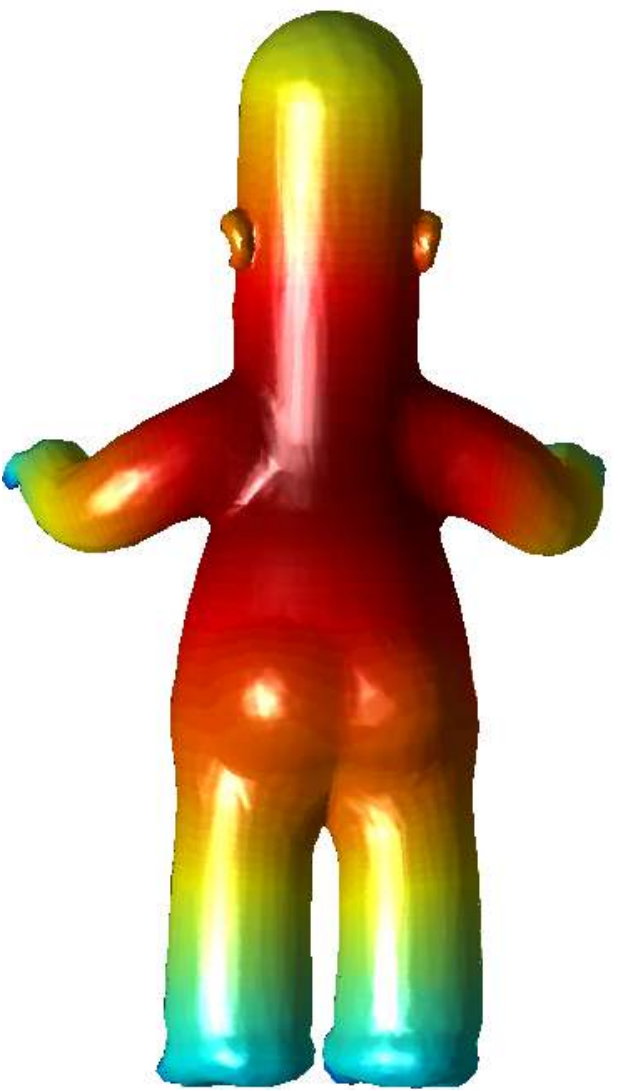}
\end{minipage}%\hfill
\begin{minipage}{0.49\linewidth}
\includegraphics[width=.75\linewidth]{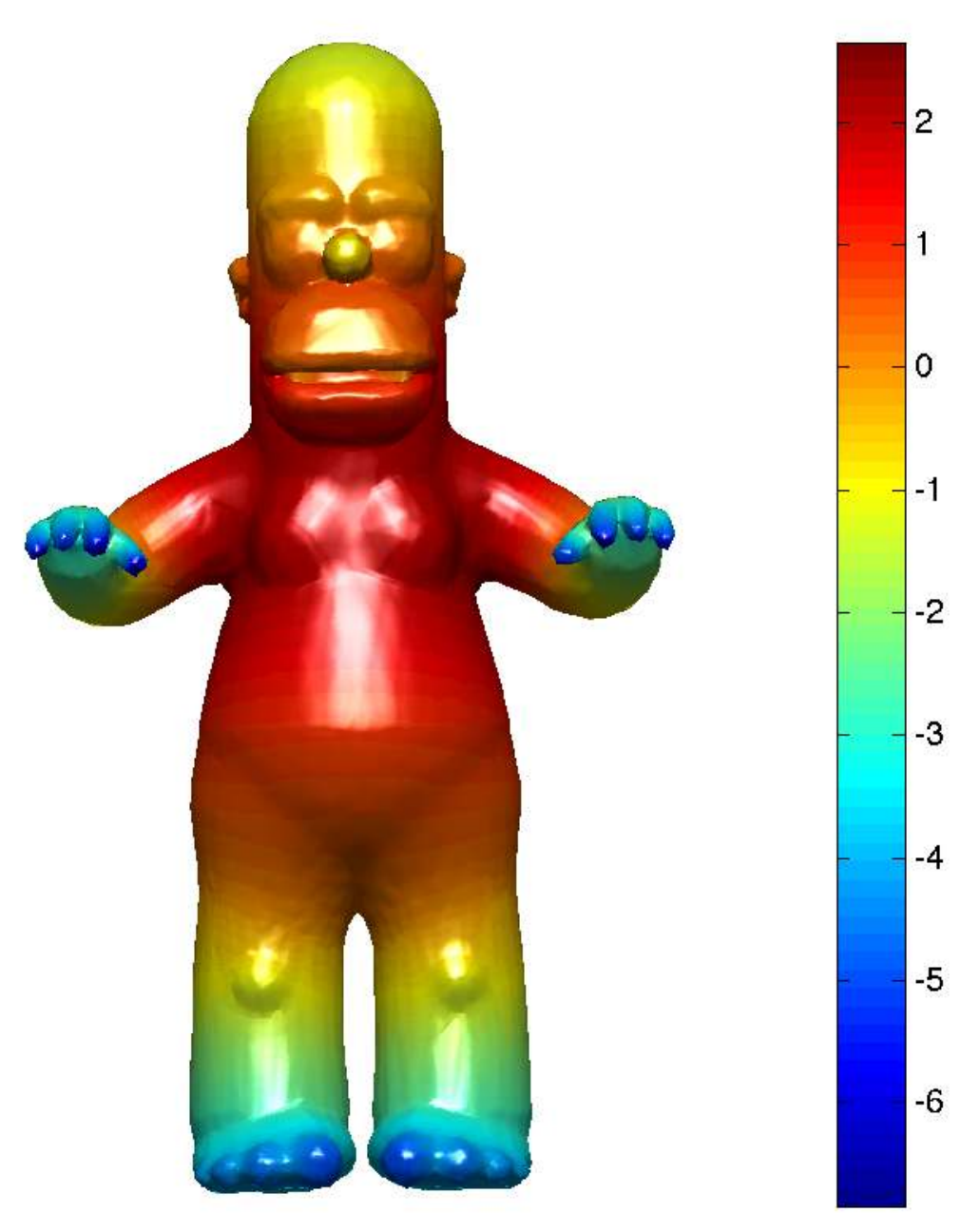}
\end{minipage}\hfill\\
\caption{A plot of the function $u = \log(\omega^\star)/2$, where $\omega^\star$ is the conformal factor corresponding to the first conformal eigenvalue $\Lambda^c_1$,  for a ``Homer Simpson'' mesh.   See \S\ref{sec:genus0}.}
\label{fig:homerGlobalParam}
\end{center}
\end{figure}

\medskip

As another computational experiment, we again consider the mesh of ``Homer Simpson'',  as discussed in \S\ref{sec:homer}. 
For this mesh, we  compute a conformal factor $\omega^\star$ which attains $\Lambda^c_1$ and plot 
the function $u = \log(\omega^\star)/2$ in Figure  \ref{fig:homerGlobalParam}. 
The first eight non-zero  eigenvalues computed for this conformal factor  are given by $2.01$, $2.01$, $2.01$, $5.93$, $6.03$, $6.04$, $6.12$ and $6.17$. 
 We see that the first three eigenvalues are close to the first three eigenvalues of the unit sphere ($\lambda = 2.00$). The 4th--8th eigenvalues are near to the 4th--8th eigenvalues of the unit sphere ($\lambda = 6.00$). This discrepancy in the higher eigenvalues may be explained by (i) we only approximately solve the optimization problem and (ii)  higher eigenvalues are more sensitive to perturbations in the conformal factor.

\subsection{The conformal spectrum of flat tori} \label{sec:FundConfEig}
In this section, we study the first conformal eigenvalue of the $(a,b)$-flat tori for various values of $(a,b)$.  For all computations, we use the spectral method described in  \S\ref{sec:compMeth}.  For a comparison, we first compute the  first non-trivial eigenvalue of the $(a,b)$-flat tori. 

\begin{figure}[t]
\begin{center}
\begin{minipage}{0.49\linewidth}
\includegraphics[width=1\linewidth]{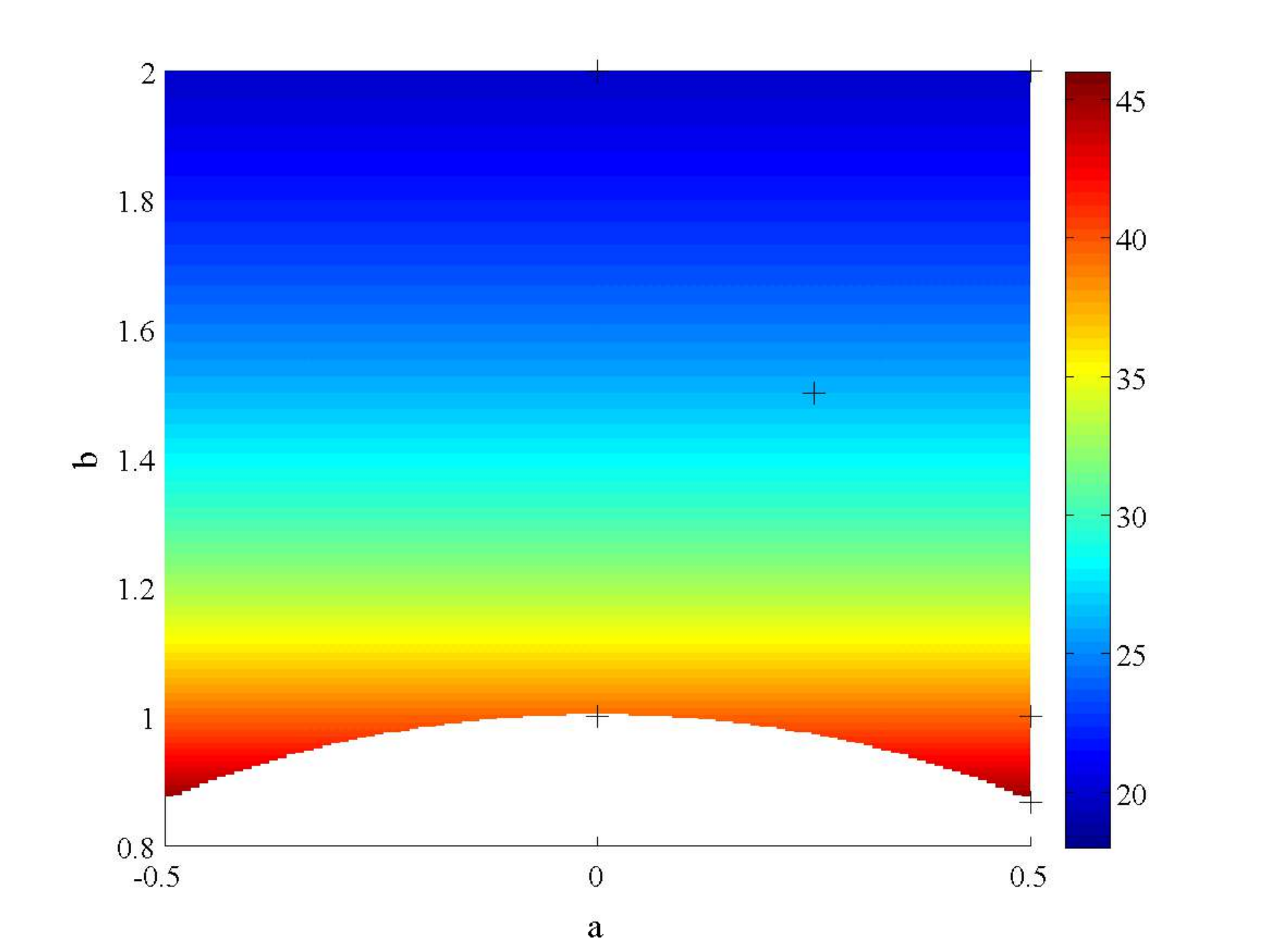}
\end{minipage}%\hfill
\begin{minipage}{0.49\linewidth}
\ 
\includegraphics[width=1\linewidth]{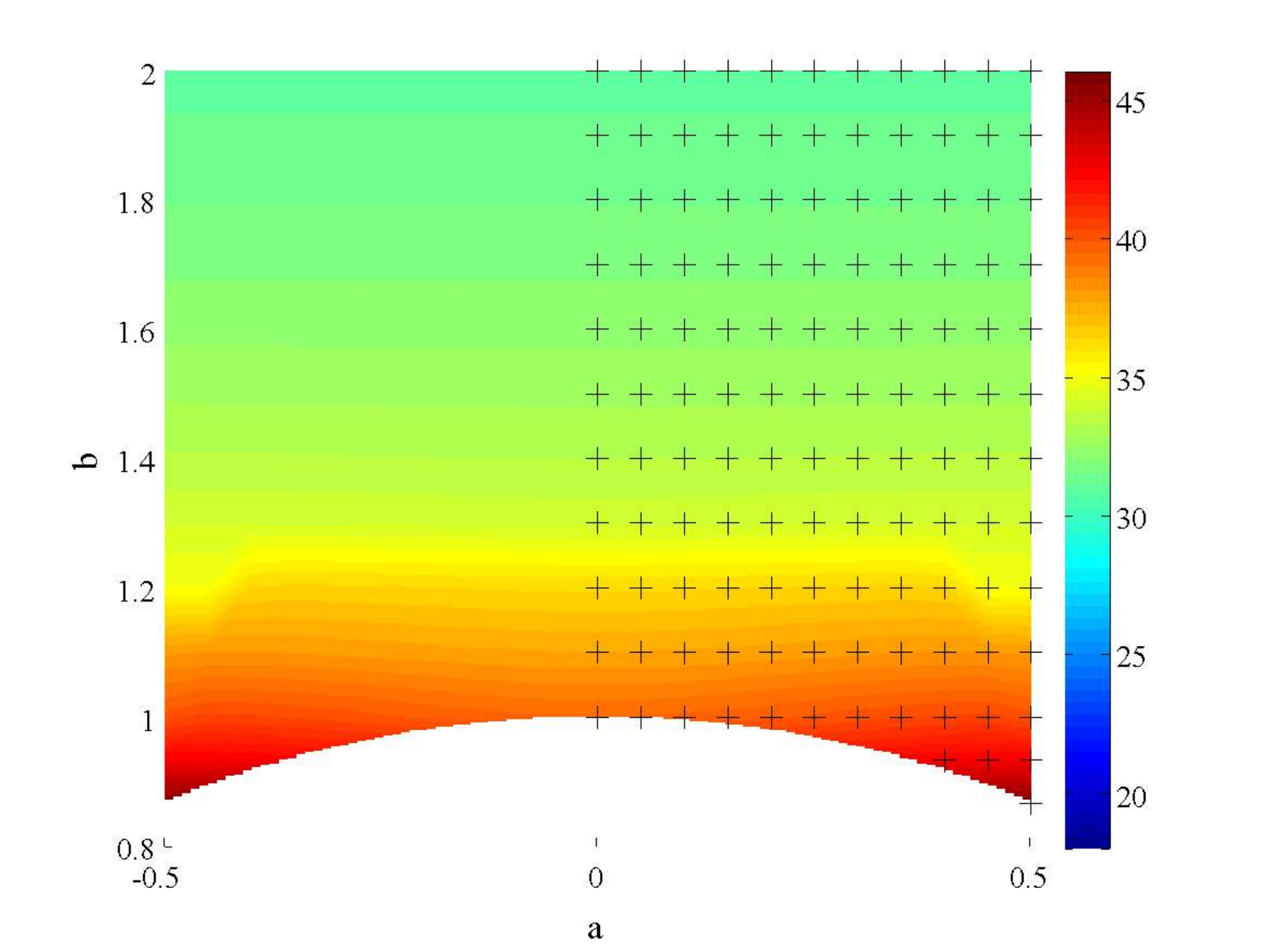}
\end{minipage}  \\

\begin{tabular}{c c| c c  }
& $(a,b)$ & $\Lambda_1(a,b)$ & $\Lambda^c_1 (a,b)$ \\
\hline
A& $(\frac{1}{2},\frac{\sqrt 3}{2})$ & 45.58 &45.58   \\
B& $(0,1)$ & 39.48 & 39.48  \\
C&$(\frac{1}{2},1)$ & 39.48 & 40.33  \\
D&$(\frac{1}{4},\frac{3}{2})$ & 26.32 & 33.45  \\
E&$(0,2)$ & 19.74 & 30.97  \\
F&$(\frac{1}{2},2)$ & 19.74 & 30.95  
\end{tabular} 
\caption{ {\bf (left)} The first eigenvalue  of $(a,b)$-flat tori, $\Lambda_1(a,b)$, for values of $(a,b) \in F$. Selected values of $(a,b)$, indicated by crosshairs, `$+$', are tabulated below for reference. 
{\bf (right)} The first conformal eigenvalue of $(a,b)$-flat tori, $\Lambda^c_1(a,b)$, for values of $(a,b) \in F$. Selected values are tabulated below for reference. An eigenvalue optimization problem was solved for values of $(a,b)$ indicated by crosshairs, `$+$'; the other values were obtained by interpolation. 
{\bf (bottom)} Tabulated values of $\Lambda_1(a,b)$ and $\Lambda^c_1(a,b)$ for selected values of $(a,b)$. The conformal factors attaining the given values of $\Lambda^c_1(a,b)$ plotted in Figure \ref{fig:flatToriFundConfFact}. See \S\ref{sec:FundConfEig}.}
\label{fig:flatToriFundEig}
\end{center}
\end{figure}

\begin{figure}[h!]
\begin{center}
\includegraphics[width=.47\linewidth]{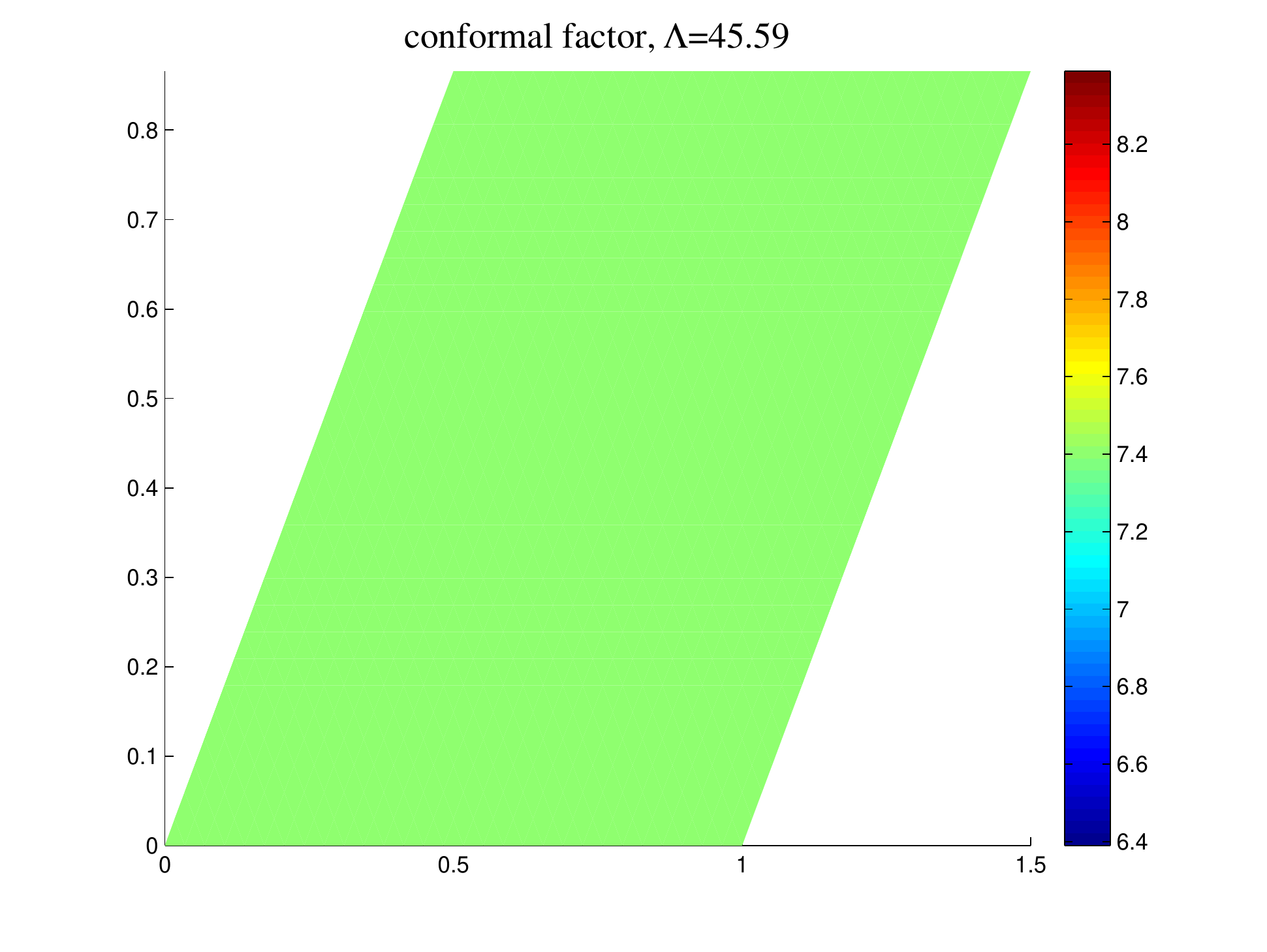}
\includegraphics[width=.47\linewidth]{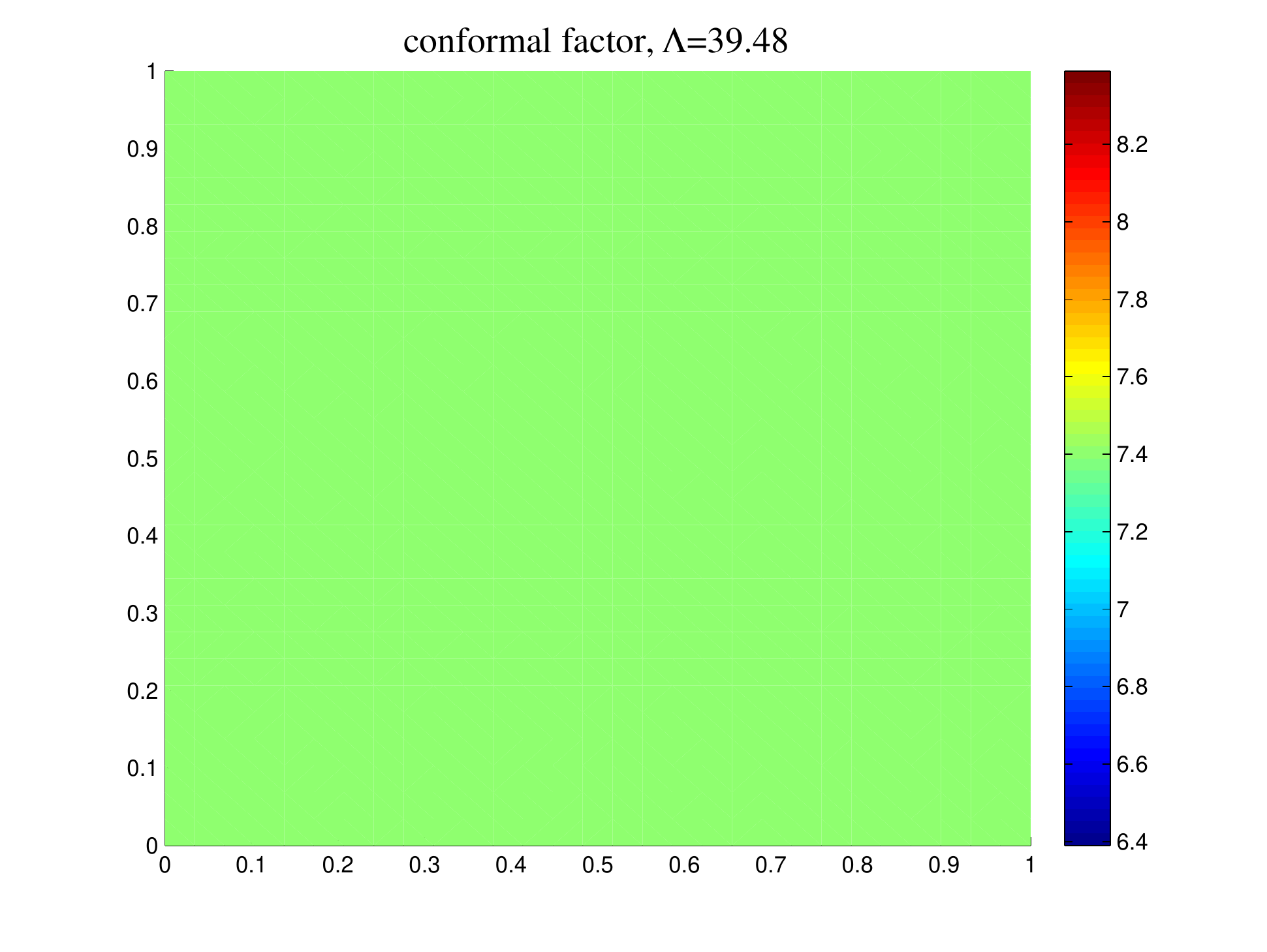}\\
\vspace{-.5cm}
 (A) \qquad \qquad \qquad \qquad \qquad \qquad \qquad \qquad \qquad  \qquad \qquad (B)\\
\includegraphics[width=.47\linewidth]{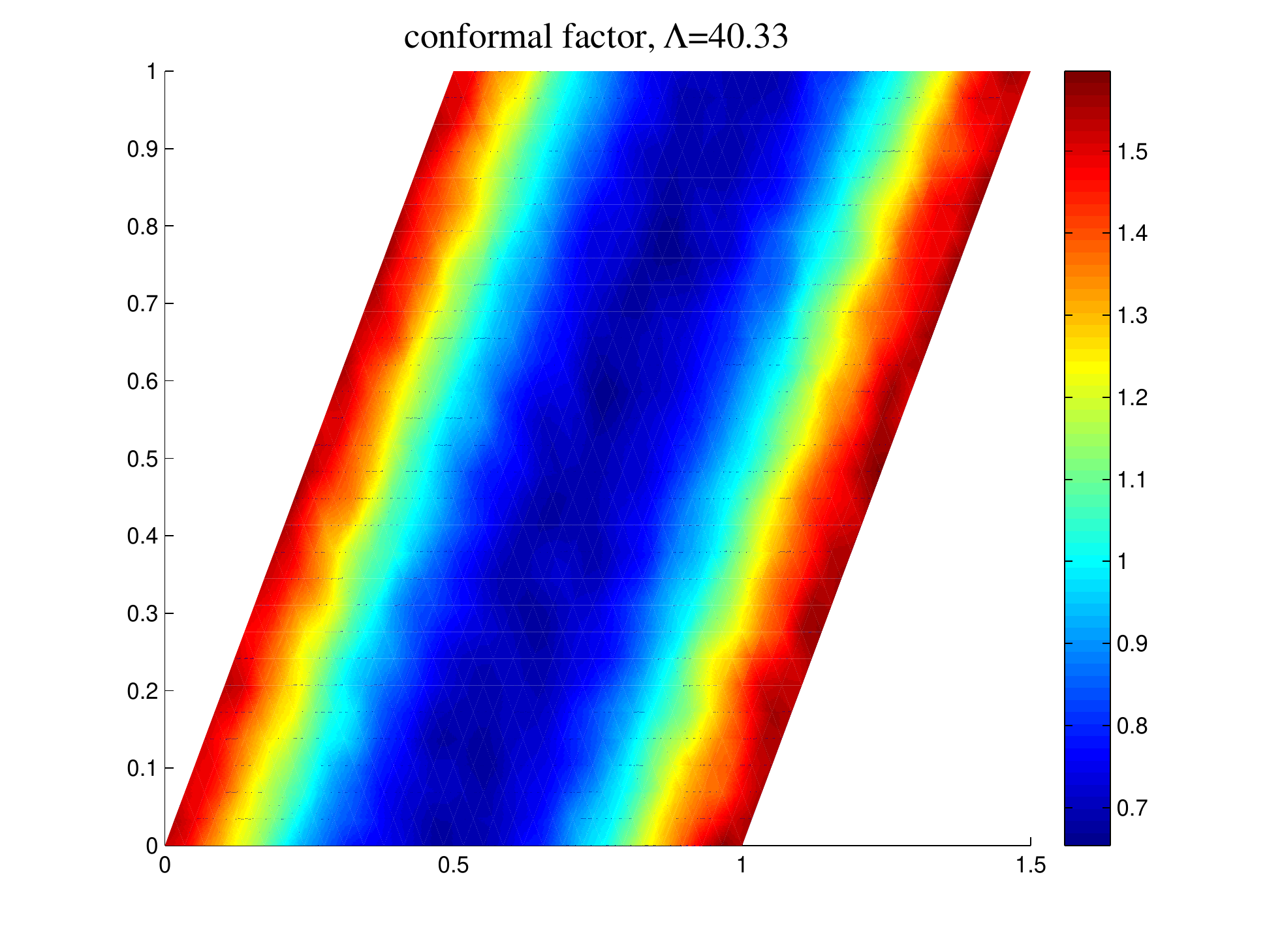}
\includegraphics[width=.47\linewidth]{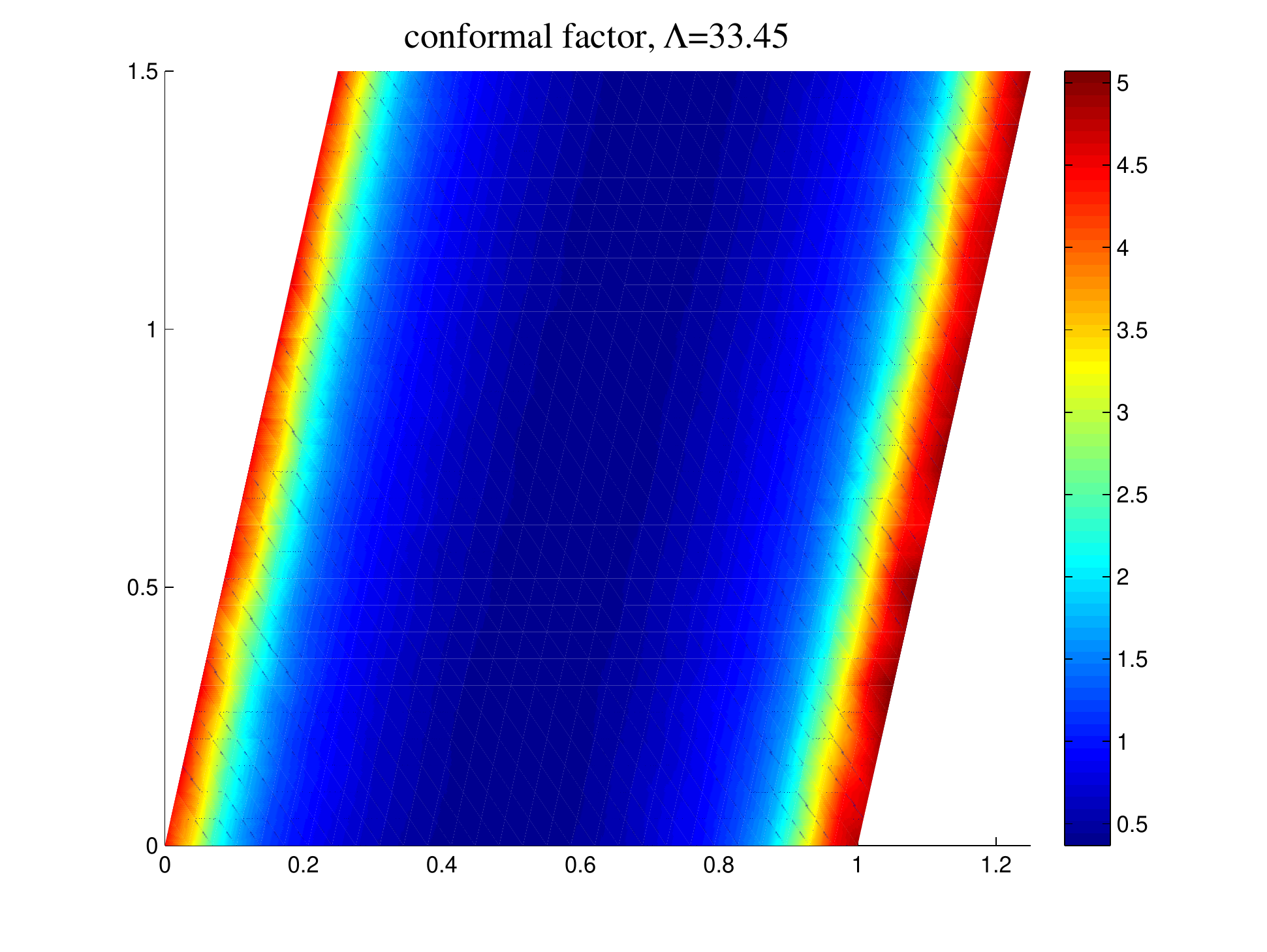}\\
\vspace{-.5cm}
 (C) \qquad \qquad \qquad \qquad \qquad \qquad \qquad \qquad \qquad  \qquad \qquad (D)\\
\includegraphics[width=.47\linewidth]{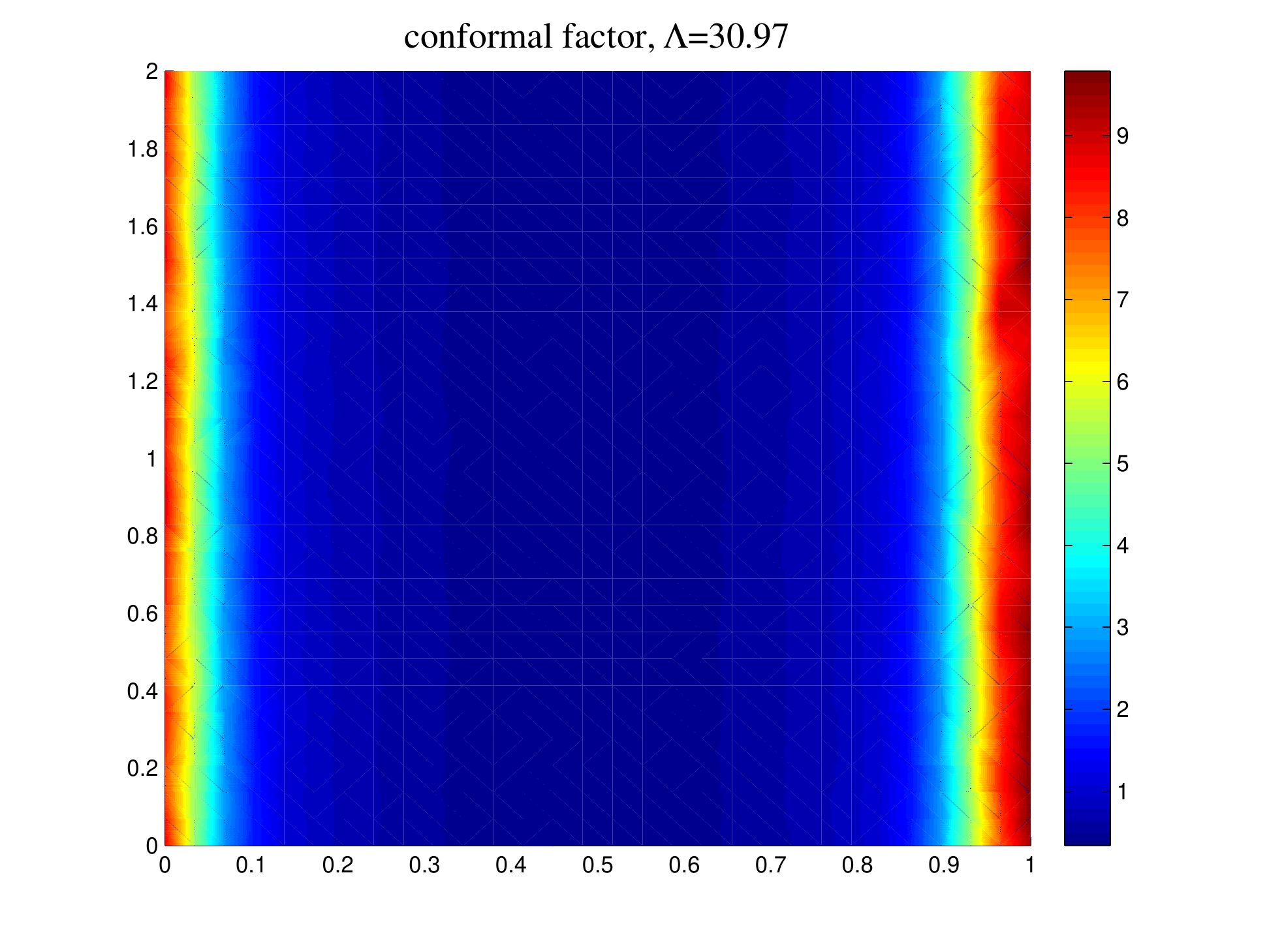}
\includegraphics[width=.47\linewidth]{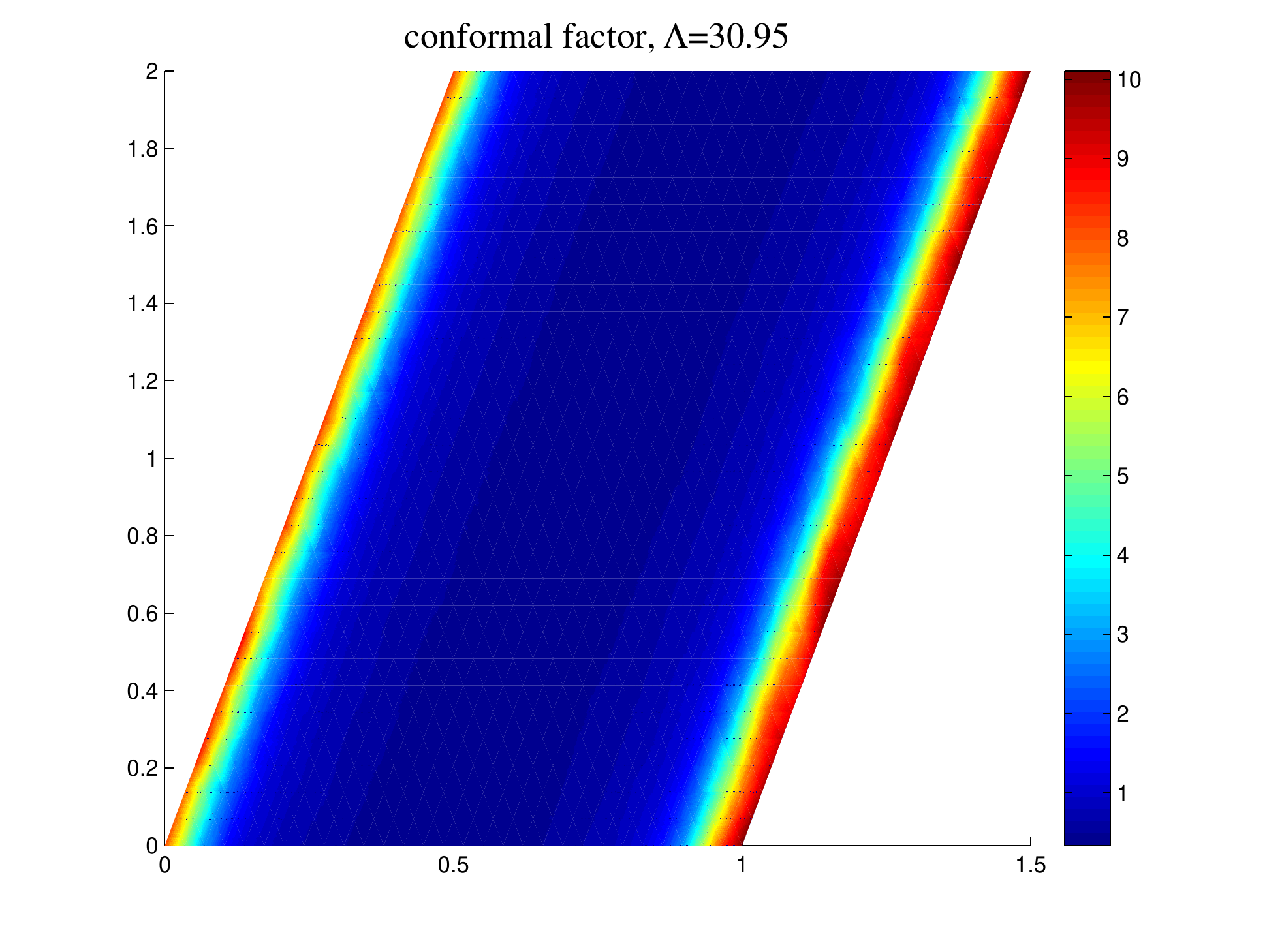}\\
\vspace{-.5cm}
 (E) \qquad \qquad \qquad \qquad \qquad \qquad \qquad \qquad \qquad (F)
 \vspace{.5cm}
\end{center}
\label{fig:flatToriFundConfFact}
\caption{A plot of the function $u = \log(\omega^\star)/2$, where $\omega^\star$ is the conformal factor attaining $\Lambda^c_1 (a,b)$ for the values of $(a,b)$ in the Table in Figure \ref{fig:flatToriFundEig}. See \S\ref{sec:FundConfEig}.}
%\vspace{.5cm}
\end{figure}

From \eqref{eq:EigsFlatTori}, it is not difficult to show that the first non-trivial eigenvalue of the $(a,b)$-flat torus is $\lambda_1(a,b) = \frac{4 \pi^2}{b^2}$. Thus, the volume normalized eigenvalue is given by $\Lambda_1(a,b) = \frac{4 \pi^2}{b}$. Note that $\Lambda_1(a,b)$ is monotone decreasing in $b$ and does not depend on $a$.  When $b=\frac{\sqrt{3}}{2}$, we recover the optimal value $\Lambda_1^\star = \frac{8 \pi^2}{\sqrt{3}}$, as discussed in \S\ref{sec:FlatTori}. A plot of $\Lambda_1(a,b)$ for $(a,b)\in F$ is given in Figure \ref{fig:flatToriFundEig}(left).  Note that this is the same as the top left panel of Figure \ref{fig:eig_FlatTori}, except the range of values of $b$ is smaller. Values of $\Lambda_1(a,b)$ for a small selection of parameters $(a,b)$ are also tabulated Figure \ref{fig:flatToriFundEig}. The parameters $(a,b)$ chosen are indicated by crosshairs, `$+$', in  Figure \ref{fig:flatToriFundEig}(left).

We abbreviate the first conformal eigenvalue of the flat torus, $\Lambda^c_1(T_{a,b},[g_0])$, by $\Lambda^c_1(a,b)$. 
We recall from \eqref{eq:SpectralGap} that $\Lambda^c_1(a,b) > 8\pi \approx 25.13$.  Clearly we have that $\Lambda^c_1(a,b) \leq \frac{8 \pi^2}{\sqrt 3} \approx 45.58$ with equality only at $(a,b) = (\frac{1}{2}, \frac{\sqrt 3}{2})$.  
\begin{prpstn} \label{prop:lam1cdecreasing}
For fixed $a$,  $\Lambda^c_1(a,b)$ is a non-increasing  function in $b$. 
\end{prpstn}
\begin{proof}
The Rayleigh quotient for the first nonzero eigenvalue can be written
$$
\lambda_1(a,b,\omega) = \min_{\substack{\int \psi \omega = 0 \\ \int  \psi^2 \omega = 1}}  \ \ 4 \pi^2 \int_{[0,2\pi]^2} \ \frac{1}{b^2} (a \psi_x - \psi_y)^2 +  \psi_x^2 \ dx dy. 
$$
Let $a$ and $\omega$ fixed and let $b_2 \geq b_1$. Let $\psi$ be an eigenfunction corresponding to $\lambda_1(a,b_1,\omega)$ (which could have multiplicity greater than one). Then we have that
\begin{align*}
\lambda_1(a,b_2,\omega) &\leq 4 \pi^2 \int_{[0,2\pi]^2} \ \frac{1}{b_2^2} (a \psi_x - \psi_y)^2 +  \psi_x^2 \ dx dy \\
& \leq 4 \pi^2 \int_{[0,2\pi]^2} \ \frac{1}{b_1^2} (a \psi_x - \psi_y)^2 +  \psi_x^2 \ dx dy \\
& = \lambda_1(a,b_1,\omega)
\end{align*}
We conclude that for $a$ and $\omega$ fixed, $\lambda_1(a,b,\omega)$ is a non-increasing function in $b$. 

Fix $a$. Take $b_2 > b_1$ and let $\omega_2$ be a conformal factor attaining  $\Lambda^c_1(a,b_2)$. Then,
\begin{align*}
\Lambda^c_1(a,b_1) 
& \geq \lambda_1 (a,b_1, \omega_2) \qquad \text{by  optimality} \\
& \geq \lambda_1(a, b_2, \omega_2) \qquad  \text{by the monotonicity of $ \lambda_1(a,b,\omega_2)$ in $b$}\\
&= \Lambda^c_1(a,b_2).
\end{align*}
\end{proof}

In Figure \ref{fig:flatToriFundEig}(right), we plot  values of $\Lambda^c_1(a,b) $ for $(a,b) \in F$, computed on a $40\times40$ mesh. An eigenvalue optimization problem was solved for the values of $(a,b)$ indicated by crosshairs, `$+$'; the other values were obtained by interpolation. Values of $\Lambda^c_1(a,b)$ for a small selection of parameters $(a,b)$ are also tabulated.  
We observe that for fixed $a$, the value of $\Lambda^c_1(a,b)$ is non-increasing in $b$, as proved in Proposition \ref{prop:lam1cdecreasing}. We also observe that $\Lambda^c_1(a,b)$ varies smoothly with $(a,b)$. 
In Figure  \ref{fig:flatToriFundConfFact}, the optimal conformal factors are plotted on the $(a,b)$-tori for these values of $(a,b)$. The flat metric attains the maximal value obtained for the square torus, $(a,b) = (0,1)$, and equilateral torus, $(a,b) = (1/2, \sqrt 3 / 2)$. As $b$ increases and the torus becomes long and thin, the best conformal factors found have structure which have higher density  along a thin strip.  We observe that the optimal conformal factor continuously deforms as the parameters $(a,b)$ change. This is in contrast with other eigenvalue optimization problems where the optimizing structure can be discontinuous with changing objective function parameters  \cite{OK12,OK12b}.

\subsection{The topological spectrum of genus one Riemannian surfaces } \label{sec:TopSpecOne}
In this section, we approximate  $\Lambda^t_k(1)$  using the computational methods described in \S\ref{sec:compMeth}. We proceed with several numerical studies. First we use a spectral method to identify approximate maximizers  by varying $(a,b,\omega)$ on a flat torus. By examining the structure of the minimizers, we  recognize that the minimizer is obtained by a configuration consisting of the union of an equilateral flat torus and $k-1$ identical round spheres. We then use a finite element method on a mesh given by this configuration to  provide further evidence  that this is the optimal configuration. 

In this first numerical study, the Laplace-Beltrami eigenvalues  of a fixed surface satisfying \eqref{eq:specMethEigEq} are computed using a spectral method on a $60\times 60$ mesh. As discussed in \S\ref{sec:moduliSp}, the moduli space for $\gamma=1$,  as shown by the shaded area in Figure  \ref{fig:coordDefs}(right), parameterizes the conformal classes of metrics $[g_0]$. Thus, any genus $\gamma=1$ surface can be described by a triple $(a,b,\omega)$ where $(a,b)\in F$ as in \eqref{eq:Mad} and $\omega$ a smooth positive function. The optimization problem is solved using a quasi-Newton  method, where the gradient of the eigenvalues with respect to the triple $(a,b,\omega)$  is computed via Proposition  \ref{prop:FTderivs}. 

\begin{figure}[t]
\begin{center}
\includegraphics[width=5.3cm]{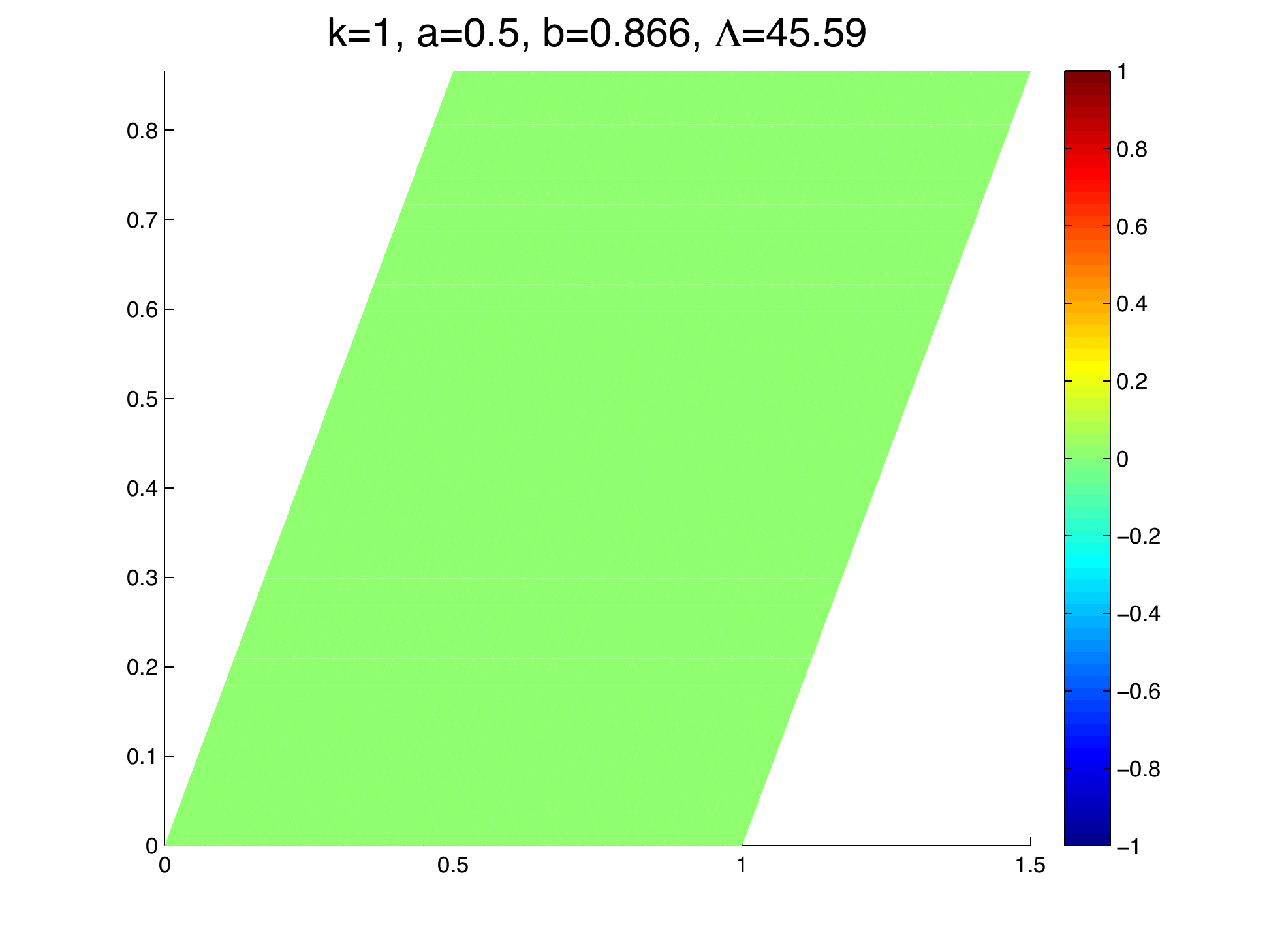} 
\includegraphics[width=5.3cm]{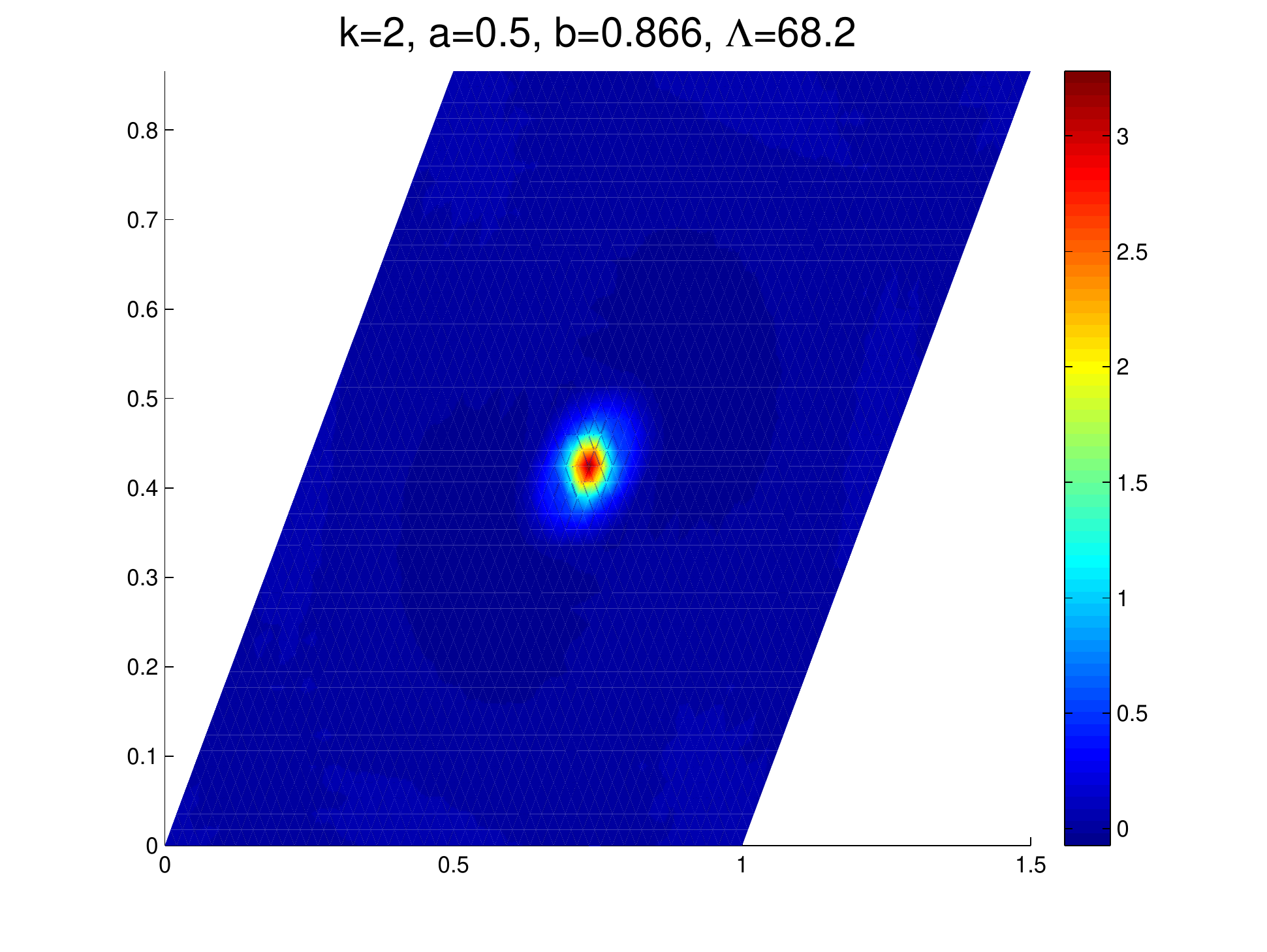} 
\includegraphics[width=5.3cm]{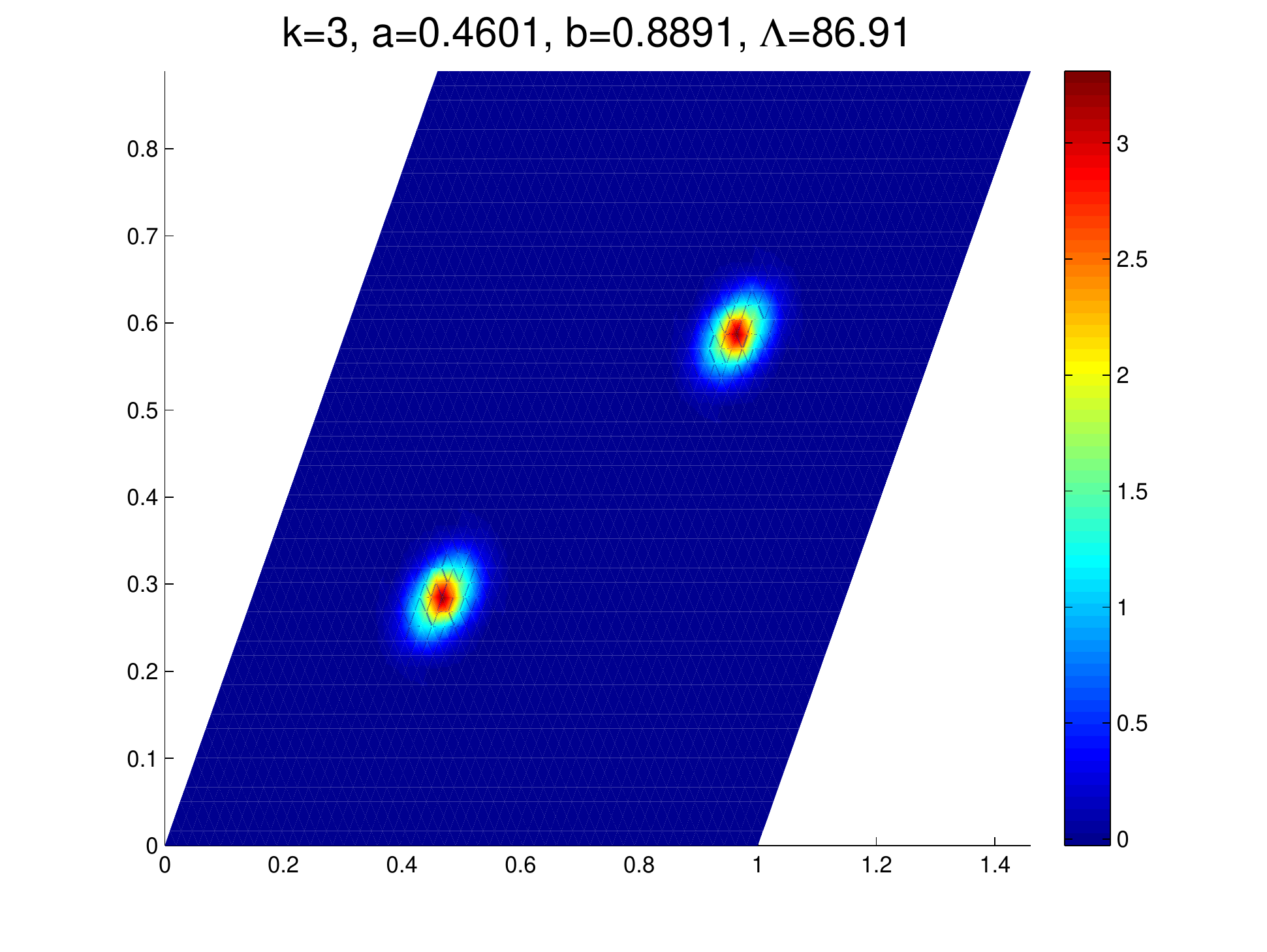} 
\caption{Maximal triples $(a,b,\omega)$ obtained for $k=1$ (left), $k=2$ (center), and $k=3$ (right). The color represents the conformal factor, $\omega$.  See \S\ref{sec:TopSpecOne}. }
\label{fig:optTori}
\end{center}
\end{figure}

Using this computational method, the best triples  $(a,b,\omega)$  found for $k=1,2,$ and $3$ are presented in Figure  \ref{fig:optTori}. 
To obtain these triples, we chose many different initializations. The initial conditions used for Figure  \ref{tab:ConfofmalSpectrum} were the sum of localized Gaussians located at distributed points on the torus.  To further illustrate our computational method, we consider a randomly initialized conformal factor. In Figure  \ref{fig:sphereSequence} (right), we plot for $k=2$, the 0th, 5th, 24th, and 30th  iterates of the conformal factor on the $(a,b)$-torus.

The  computational results in Figure  \ref{fig:optTori}(left) support Nadirashvili's result  that 
$\Lambda^t_1(1) = \frac{8\pi^2}{\sqrt{3}} \approx 45.58$
is attained only by the flat metric on the equilateral flat torus, $(a,b) = \left(\frac{1}{2},\frac{\sqrt{3}}{2}\right)$  \cite{Nadirashvili1996}. 
For $k=2$, the optimal conformal factor found is mostly flat with one localized maximum. The value obtained ($\Lambda_2 = 68.2$) is very close to the value found for the disconnected union of an equilateral flat torus and a sphere of appropriate volumes, 
$\Lambda_2 = \Lambda^t_1(1) + \Lambda^t_1(0) \approx 70.72$ (see \S\ref{sec:disconnectedUnion}). 
For $k=3$, the optimal conformal factor found is mostly flat with two localized maximum. The value obtained ($\Lambda_2 = 86.91$) is not as close to 95.85, the value for the disconnected union of an equilateral flat torus and two spheres. 
For larger values of $k$, we observe that optimal metrics are mostly flat, but have $k-1$ localized regions with large value.  However, as for the genus $\gamma=0$ case described in \S\ref{sec:genus0}, the computational problem becomes increasingly difficult with larger values of $k$ because the localized regions are increasingly small. It is thus very difficult to realize metrics which correspond to this configuration using this method.

\begin{figure}[t]
\begin{center}
\includegraphics[width=8cm]{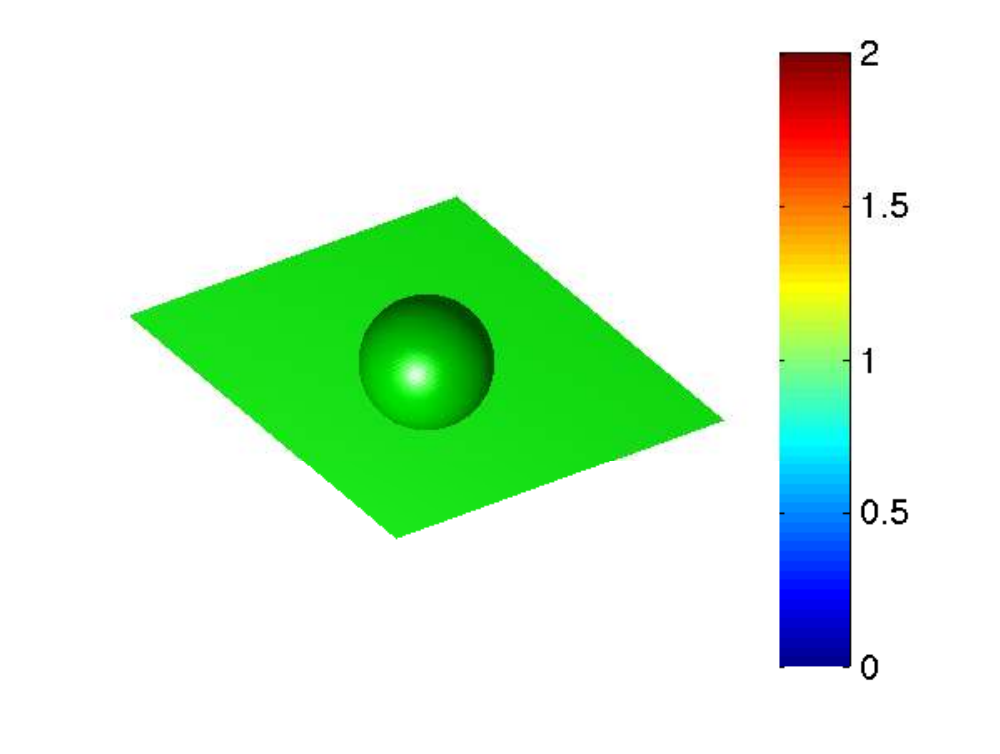} 
\includegraphics[width=8cm]{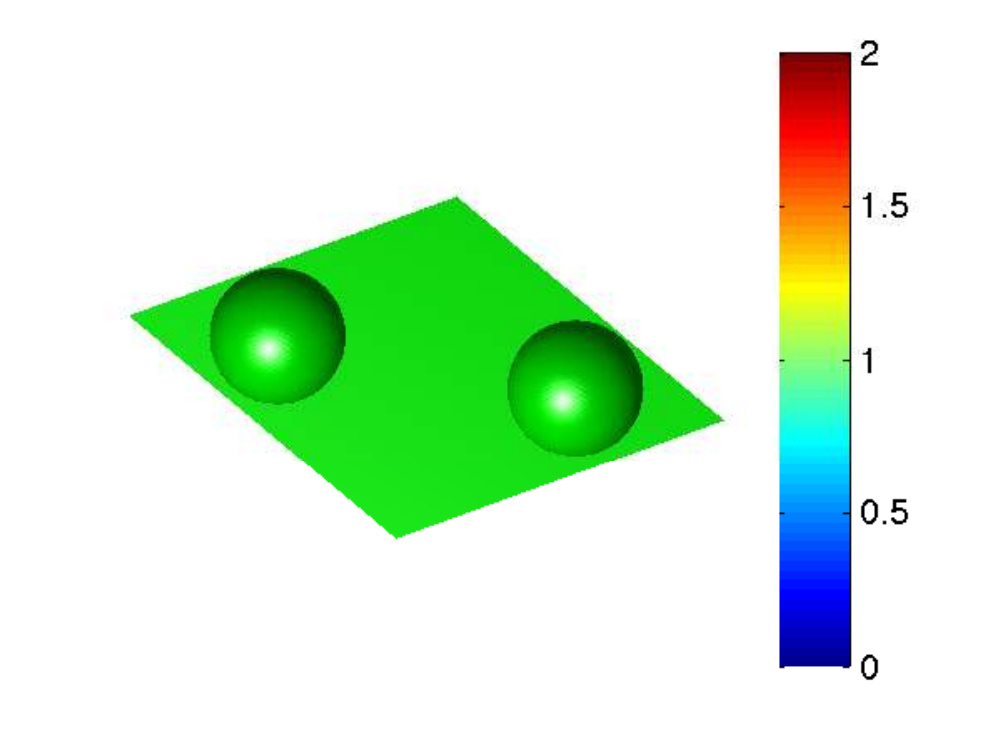} 
\includegraphics[width=8cm]{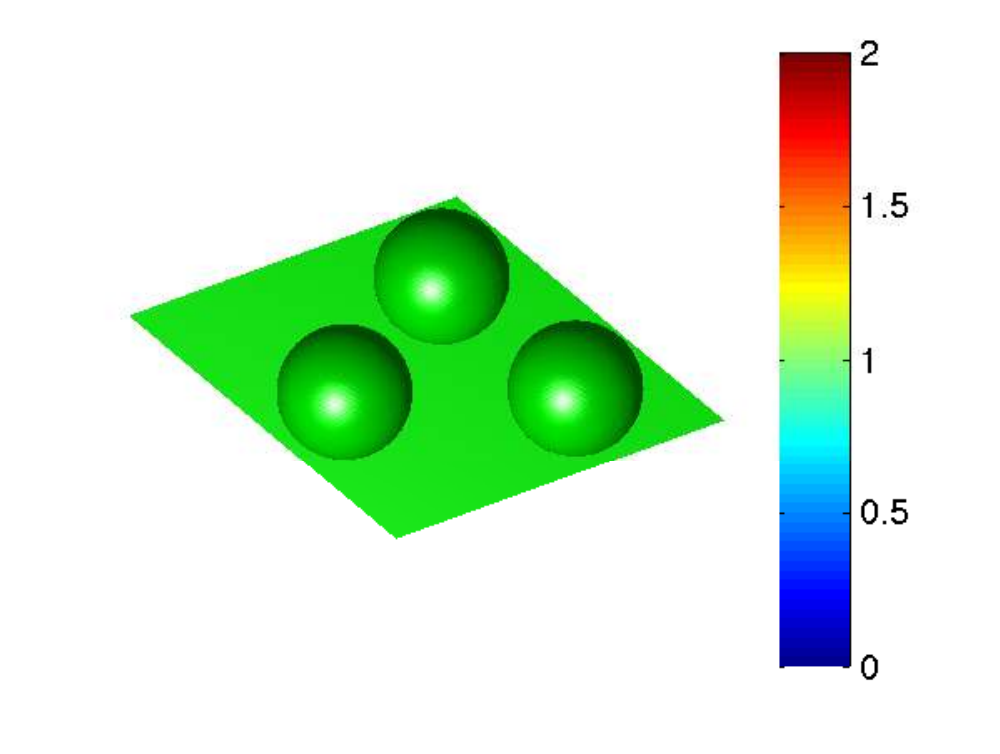} 
\includegraphics[width=8cm]{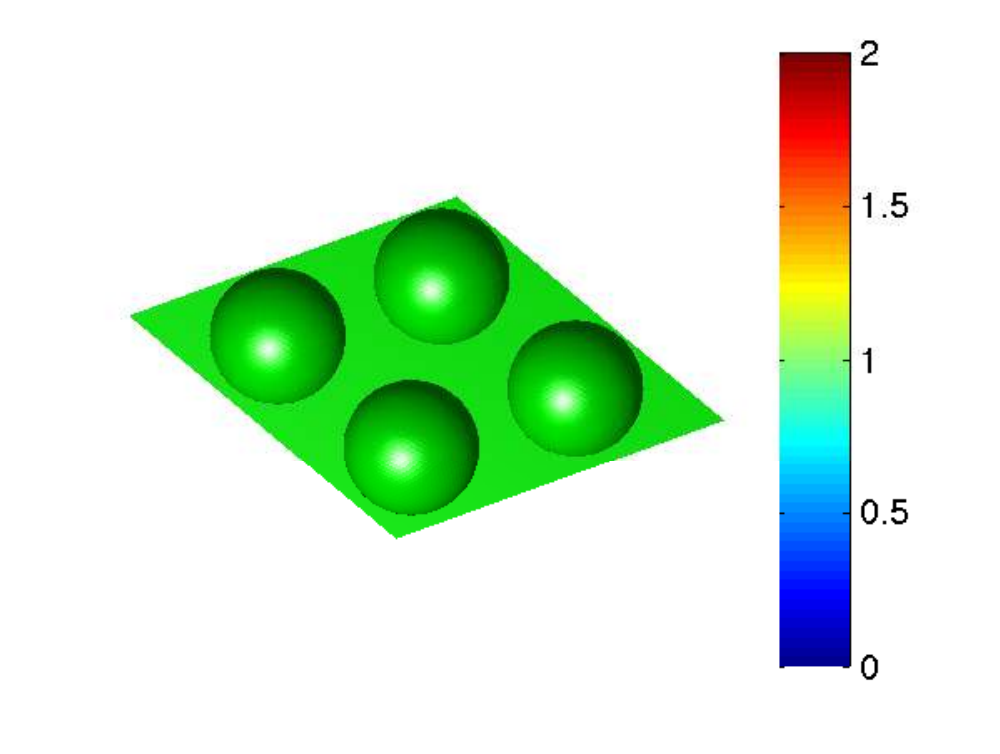} 
\caption{To maximize $\Lambda_{k}$, for $k=2$, 3, 4, and 5, we consider a mesh of a flat torus glued to 1, 2, 3, and 4 kissing spheres. The optimal conformal factors found, displayed here, are nearly constant. See \S\ref{sec:TopSpecOne}.}
\label{fig:KissingSphereFlatTorus1}
\end{center}
\end{figure}

\begin{figure}[t]
\begin{center}
\includegraphics[width=8cm]{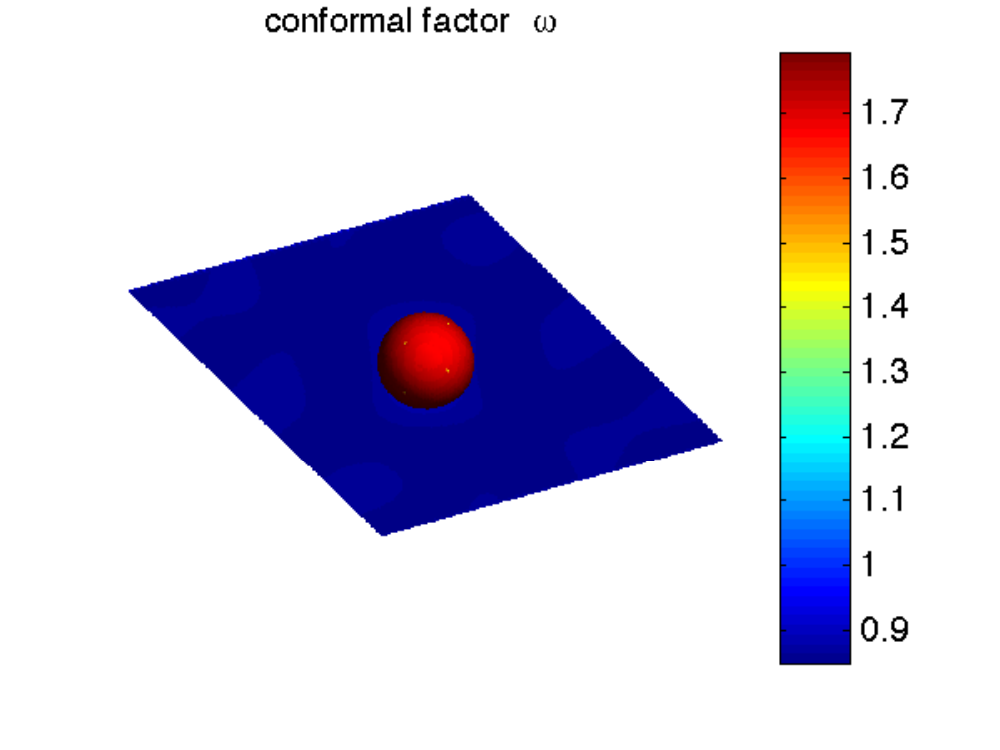} 
\includegraphics[width=8cm]{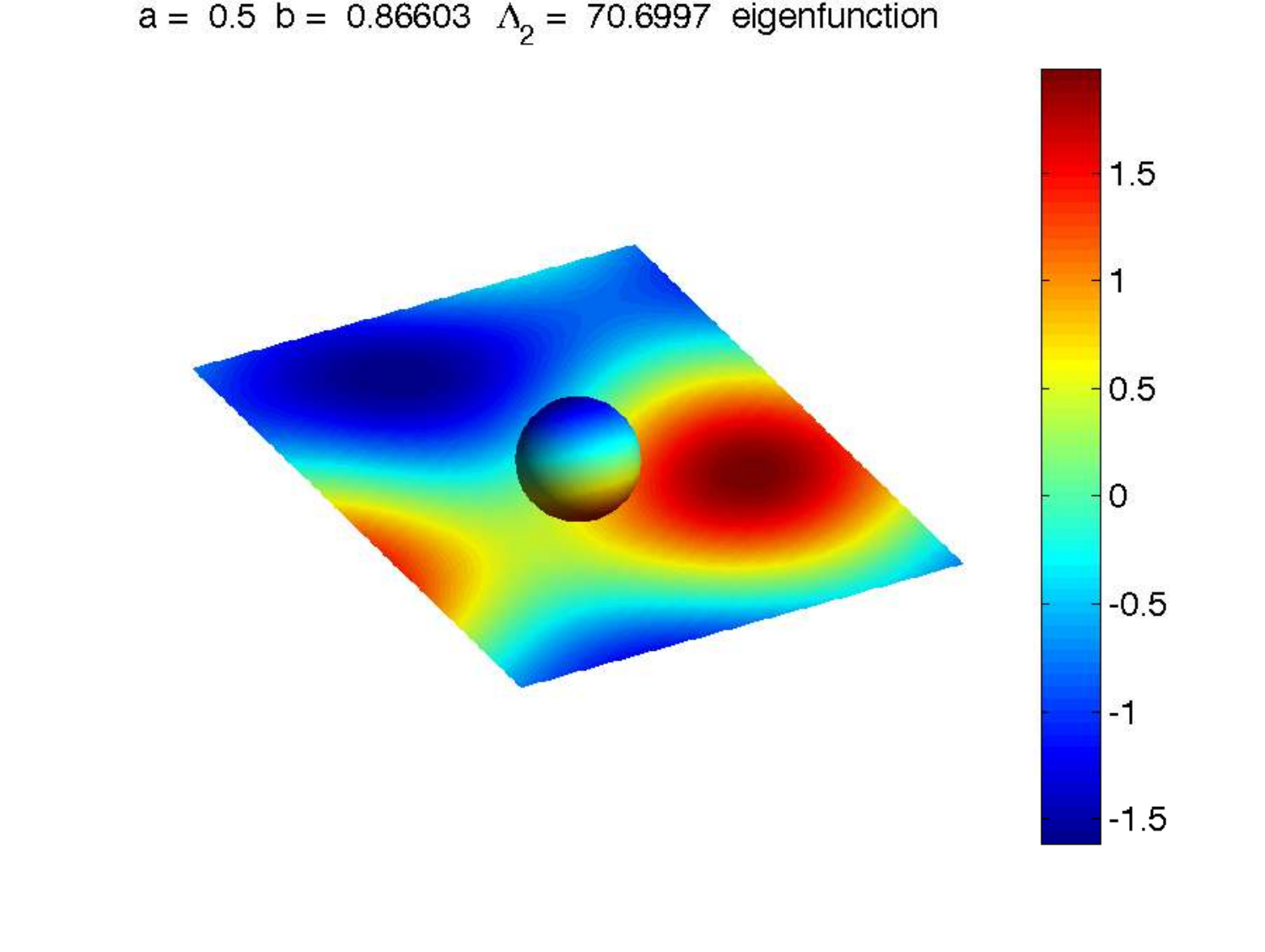} 
\includegraphics[width=8cm]{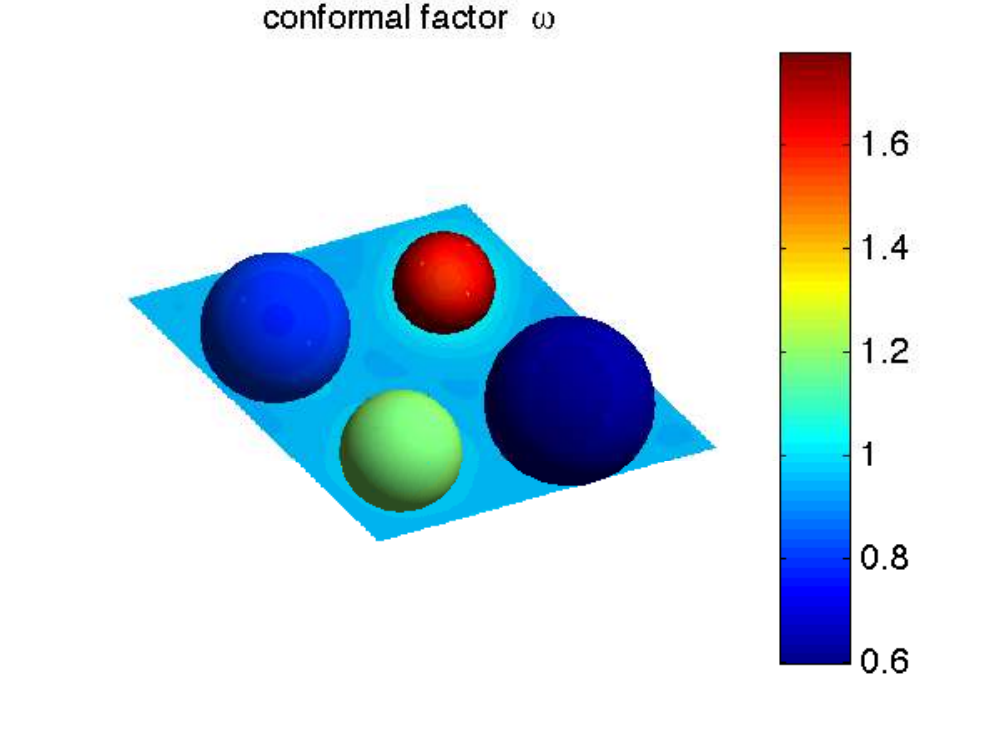} 
\includegraphics[width=8cm]{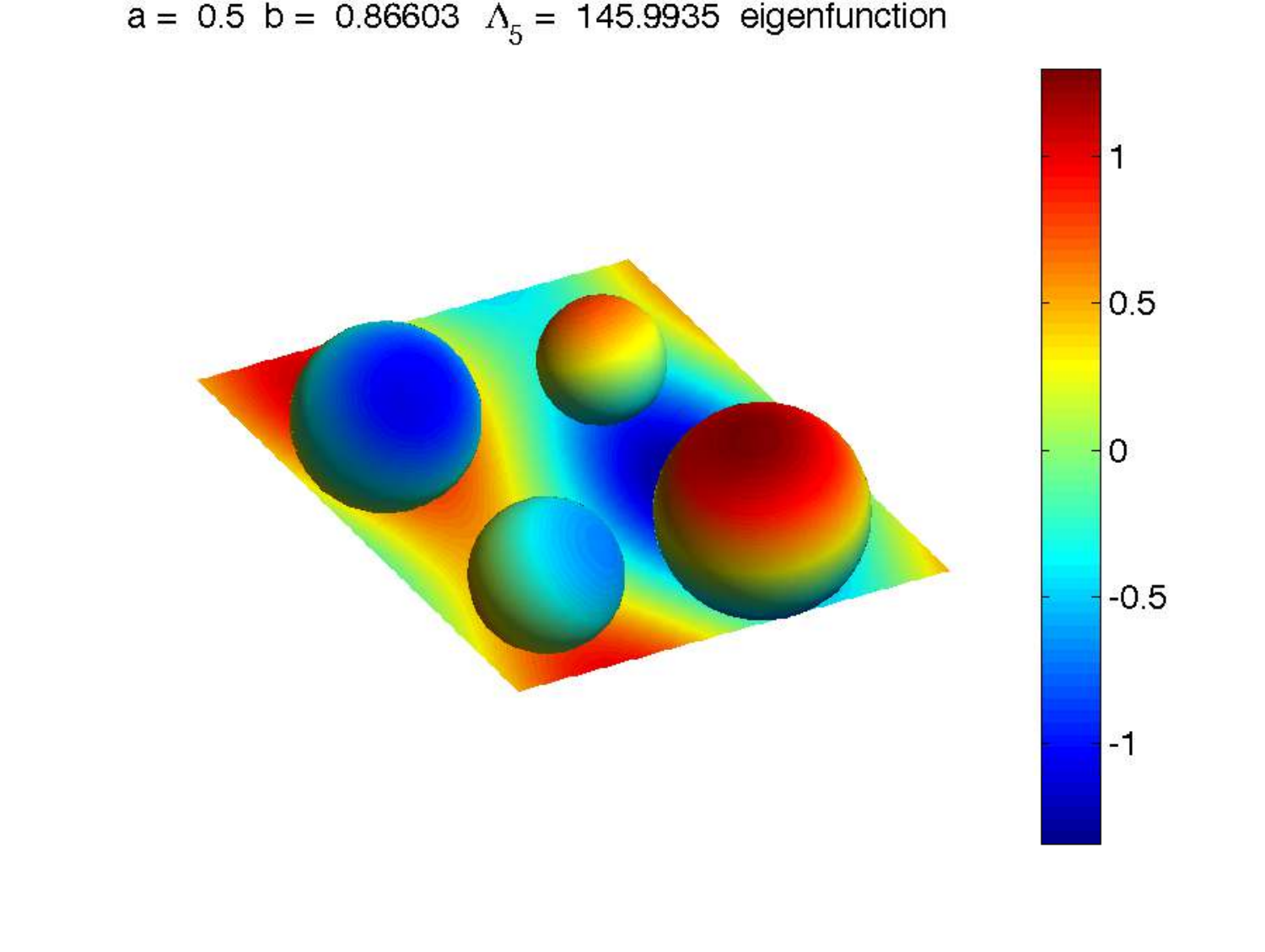} 
\caption{A mesh of a flat torus glued to one (top) and four (bottom) kissing spheres, used for maximizing $\Lambda_{2}$ and $\Lambda_{5}$ respectively. The figures on the left display the  optimal conformal factor and the figures on the right display an eigenfunction corresponding to $\lambda_k$. See \S\ref{sec:TopSpecOne}.  }
\label{fig:KissingSphereFlatTorus2}
\end{center}
\end{figure}

To compute optimal  configurations  for larger values of $k$, we proceed as follows as in \S\ref{sec:genus0} and use a mesh which consists of a torus which has been deformed locally at $k-1$ points. In effect, this mesh approximates the configuration of $k-1$ spheres ``kissing'' a flat tori. 
For example, to construct this mesh for $k=2$, we remove one face from the mesh representing the flat tori and one face from the mesh representing the sphere. We then identify the edges associated with the missing faces of these two punctured  surfaces. 
As discussed in \S\ref{sec:disconnectedUnion}, for an equilateral flat torus, $(a,b) = \left(\frac{1}{2},\frac{\sqrt{3}}{2}\right)$, and spheres of appropriate size, we can obtain $k$-th eigenvalue at least as large as 
\begin{equation}\label{eq:EqTorusSpheres} 
\Lambda_k = \frac{8\pi^2}{\sqrt{3}} + 8\pi (k-1). 
\end{equation} 
On this mesh, we use the finite element method to compute the Laplace-Beltrami eigenvalues and initialize a quasi-Newton optimization method using a random conformal factor. We observe that the maximal eigenvalue is achieved when the conformal factor is nearly constant over the mesh. See Figure  \ref{fig:KissingSphereFlatTorus1}, where the optimal values are given by $\Lambda_{2} = 70.70 $, $\Lambda_{3} = 95.80$, $\Lambda_{4} = 120.94$, and $\Lambda_{5} = 146.06$ which are indeed very close to those given in \eqref{eq:EqTorusSpheres}. 

To further test these optimal solutions, we consider  several other configurations of spheres and flat tori. As shown in  Figure \ref{fig:KissingSphereFlatTorus2}, we take $k=2$ and consider a torus glued to a sphere with radius  a factor of 0.7 of the optimal radius. Initializing the optimization method with a constant uniform conformal factor, an optimal conformal factor is achieved where the sphere has a relatively high conformal value and the torus has relatively low conformal value. See Figure   \ref{fig:KissingSphereFlatTorus2} (top left). An  eigenfunction associated to $\lambda_2$ is plotted in Figure   \ref{fig:KissingSphereFlatTorus2} (top right).
For $k = 5$, we consider a torus glued to 4 spheres which have radii a factor of 
0.75, 0.9, 1.1, and 1.25 of the optimal radius. Again initializing the optimization method with a constant uniform conformal factor, we obtain the conformal factor in 
Figure   \ref{fig:KissingSphereFlatTorus2} (bottom left). An eigenfunction associated to $\lambda_5$ is plotted in Figure   \ref{fig:KissingSphereFlatTorus2} (bottom right). 
In these two experiments, the optimal numerical values $\Lambda_{2} = 70.6997$ and $\Lambda_{5} = 146.9935$ are close to the values given in  \eqref{eq:EqTorusSpheres}. 

\medskip

Finally, we report the results for one additional computational experiment. Recall that our proposed numerical method is only able to find local maxima of the non-concave optimization problem \eqref{eq:opt:Linfty}. In addition to many randomly initialized initial configurations, configurations with localized Gaussians, and configurations consisting of glued spheres and tori, we initialized the method using one additional  configuration, two kissing flat equilateral tori. For a moment, consider two embedded tori  stacked on each other so that the holes are aligned and the contact is smooth (Homer would think of a stack of donuts). This configuration is of a different type than two kissing balls since the perturbation occurs along a one-dimensional submanifold rather than at a single point. Since this type of perturbation is more difficult to analyze, we thought that it would be useful to check this configuration numerically. However, as we demonstrated in  \S\ref{sec:FlatTori} and \S\ref{sec:tori} the eigenvalues associated with embedded tori are generally not as large as those associated with flat tori. Thus, we consider gluing two  equilateral flat tori along a strip as shown in Figure \ref{fig:KissingTori}. Here colors and arrows indicate the glued edges. Numerically, we remove a strip from each of the two flat tori and identify element vertices and element edges along the cut edges. This constructed  surface has genus $\gamma=1$, as verified numerally using the Euler characteristic of the mesh. The first few eigenvalues of this configuration are given in Table \ref{tab:ConfofmalSpectrum}. The value of $\Lambda_{2}$ is very small as compared to \eqref{eq:EqTorusSpheres} with $k=2$.

\begin{figure}[t]
\begin{center}
\includegraphics[width=8cm]{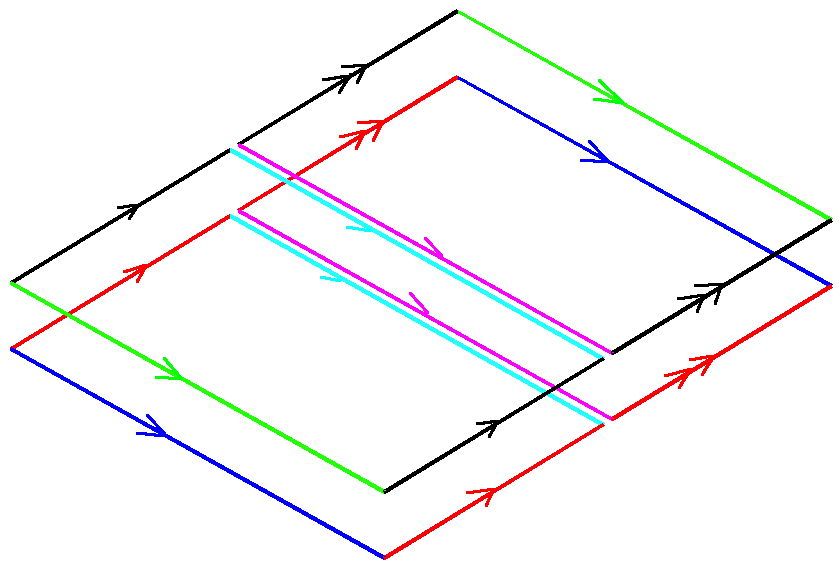} 
\caption{Kissing  equilateral flat tori. Edges with the same color are glued together. See \S\ref{sec:TopSpecOne}.}
\label{fig:KissingTori}
\end{center}
\end{figure}

\section{Discussion and further directions} \label{sec:Disc}
We have presented a computational method for approximating the conformal and topological spectra, as defined in \eqref{eq:ConEig} and \eqref{eq:TopEig}.  Our method is based on a relaxation, given in \eqref{eq:opt:Linfty}, for which we  prove existence of a minimizer (see Proposition \ref{prop:Existence}).  Based on the results of extensive computations, we make the following conjecture. 
\begin{cnjctr} The following hold for the topological spectrum:
\begin{itemize}
\item $\Lambda^t_k(0) = 8 \pi k$, attained by a sequence of surfaces degenerating to a union of $k$ identical round spheres. 
\item $\Lambda^t_k(1) = \frac{8\pi^2}{\sqrt{3}} + 8\pi (k-1)$, attained by a sequence of surfaces degenerating into a union of an equilateral flat torus and $k-1$ identical round spheres. 
\end{itemize}
\end{cnjctr}
The first part of this conjecture was also stated by Nadirashvili in \cite{Nadirashvili2002}. A proof of the conjecture involving $\Lambda^t_k(0)$  would imply that  the lower bound, $\Lambda^t_k(0) \geq 8 \pi k$, proven in \cite[Corollary 1]{Colbois2003}, is tight. This conjecture is proven for $k=1$ and $k=2$ in   \cite{hersch1970} and \cite{Nadirashvili2002} respectively. The conjecture involving $\Lambda^t_k(1)$ agrees with the eigenvalue gap estimate \eqref{eq:SpectralGap}, proven in \cite{Colbois2003}, and the result of  \cite{Nadirashvili1996} for $k=1$. The relatively large value of $\Lambda_k$ for the configuration consisting of a  union of an equilateral flat torus and $k-1$ identical round spheres was  recently used by A. Karpukhin  to show that a number of extremal metrics are not maximal \cite{karpukhin2013b}. 

For dimension $n=2$, Weyl's law states that for any fixed surface $(M,g)$, $\Lambda_k(M,g) \sim 4 \pi k$ as $k\to \infty$. The conjectured topological spectrum for genus $\gamma=0,1$ has asymptotic behavior $\Lambda^t_k(\gamma) \sim 8 \pi k$. Thus, the conjecture implies that for fixed $k$, there exist surfaces with $k$-th eigenvalue which exceed the asymptotic estimate given by  Weyl's law by no more than a factor of two. 
As a comparison, we proved in \S\ref{sec:FlatTori} that among flat tori, $\Lambda_k$ has a local maximum with value  
$ 4\pi^2 \left\lceil \frac{k}{2} \right\rceil^2 \left( \left\lceil \frac{k}{2} \right\rceil^2  - \frac{1}{4}\right)^{-\frac{1}{2}}$. 
For $k$ large, we obtain $\Lambda_k \sim 2 \pi^2 k$. Noting that $4\pi< 2 \pi^2 <  8 \pi$, this rate lies between Weyl's estimate and the conjectured topological spectrum for genus $\gamma=1$. 

In \S\ref{sec:TopSpecOne}, we used an explicit parameterization of the genus one moduli space to compute the topological spectrum. Higher genus moduli spaces ({\it e.g.} $\gamma=2$) could in principle be treated in the same way \cite{Imayoshi1992,Buser2010}, although we do not attempt this here. For genus $\gamma=2$, \eqref{eq:lam1genus2} and the spectral gap \eqref{eq:SpectralGap} together imply that  
$$\Lambda^t_k(2) \geq 8 \pi (k+1)$$
 where the lower bound is attained by attaching $k-1$ spheres to a Bolza surface (a singular surface which is realized as a double branched covering of the sphere). It was observed by B. Colbois and A. El Soufi  \cite[Corollary 3.1]{Colbois2012} that this bound is not tight; the union of two equilateral flat tori gives a  higher second eigenvalue than the union of a Bolza surface with a round sphere. The high genus topological spectrum is largely open and is a very interesting  future direction.

\clearpage

\begin{table}[t]
\begin{center}
\begin{tabular}{c|c c c c c c}
& & square & equilateral & horn emb. &Homer  & kissing   \\
$k$& sphere & flat torus & flat torus& torus &Simpson & tori \\
\hline
1 & 25.13 & 39.47 &45.58 &23.21 &7.464 &34.21 \\
2 & 25.13 & 39.47 &45.58 &23.21 &16.45 &34.21\\
3 & 25.13 & 39.47 &45.58 &30.63 &19.94 &43.98\\
4 & 75.39 & 39.47 &45.58 &66.58 &20.89  & 43.98\\
5 & 75.39 & 78.95 &45.58 &66.58  &41.23 &78.22 \\
6 & 75.39 & 78.95 &45.58 &78.80  &69.83 &78.22 \\
7 & 75.39 & 78.95 &136.7 &83.71 &85.40 &78.22\\
8 & 75.39 & 78.95 &136.7 &83.71 &92.32 & 78.22
\end{tabular}
\bigskip

\begin{tabular}{c|c    c c c   }
&   $k$ kissing &  best flat & best embedded & Equil. torus   \\
&   spheres & torus  & torus & and $k-1$  \\
$k$ & \eqref{eq:kissingspheres}& \eqref{eq:flattori_exact}&\eqref{eq:supEmbTori} & spheres \eqref{eq:EqTorusSpheres} \\
\hline
1  &25.13 & 45.58   &23.47 & 45.58\\
2 &50.26 &45.58   &23.47 & 70.71\\
3 &75.39 &81.55   &65.09 & 95.85\\
4 &100.5 &81.55  &65.09 & 120.9 \\
5  &125.6 &120.1  &108.34 & 146.1\\
6 &150.7 &120.1  &108.34 & 171.2\\
7  &175.9 &159.2 &150.25 & 196.3\\
8  &201.0 &159.2 &150.25 & 221.5
\end{tabular}
\bigskip

\begin{tabular}{c| c c  }
$k$ & $\Lambda^t_k (0)$ & $\Lambda^t_k (1)$ \\
\hline
1 &25.13 &45.58   \\
2 &50.26 & 70.71 \\
3 &75.39 & 95.85 \\
4 &100.5 & 120.9  \\
5 &125.6 &146.1 \\
6 &150.7 & 171.2 
\end{tabular} 
\end{center}
\caption{A comparison of various volume-normalized eigenvalues, $\Lambda_k(M,g) = \lambda_k(M,g)\cdot \vol(M,g)$. This is equivalent to $\lambda_k(M,g)$ after the metric has been normalized to have unit volume. 
The first table are the Laplace-Beltrami  eigenvalues of the sphere, square flat torus $(a,b) = \left( 0, 1 \right)$,  equilateral flat torus $(a,b) = \left( \frac{1}{2}, \frac{\sqrt{3}}{2} \right)$, horn embedded torus,   Homer Simpson, and kissing equilateral flat tori as discussed  in 
\S\ref{sec:Sphere},  \S\ref{sec:FlatTori},  \S\ref{sec:FlatTori}, \S\ref{sec:tori},  \S\ref{sec:homer}, and \S\ref{sec:TopSpecOne}
respectively. 
The second table are the Laplace-Beltrami eigenvalues for varying Riemannian surfaces: 
$k$ kissing spheres, best flat tori, best embedded tori, and the disjoint union of an equilateral torus and $k-1$ spheres as defined in 
\eqref{eq:kissingspheres},   \eqref{eq:flattori_exact}, \eqref{eq:supEmbTori}, and \eqref{eq:EqTorusSpheres}.
The third table gives the computed  topological spectra for genus $\gamma=0$ and $\gamma=1$ surfaces. }
\label{tab:ConfofmalSpectrum}
\end{table}%

\clearpage

\subsection*{Acknowledgements} We would like to thank Ahmad El Soufi, Alexandre Girouard, Richard Laugesen, Peter Li, and Stan Osher for useful conversations. We would like to thank Melissa Liu for help with the example in Remark \ref{rem:isom}. We would also like to thank the referees for their valuable comments.

{\small 
\bibliographystyle{amsalpha}
\bibliography{ManifoldOpt}}

\end{document}